\documentclass[10pt,reqno]{amsart}
\usepackage{amsmath,amsthm,amssymb,amscd,url,enumerate}
\usepackage{mathtools}
\usepackage{graphicx}
\usepackage{afterpage}
\usepackage{tikz-cd}
\usepackage{todonotes}
\usepackage[all,cmtip]{xy}
\usepackage[colorlinks=true]{hyperref}
\usepackage[width=5.7in,height=8.5in, centering]{geometry}

\newcommand{\TITLE}{The Apollonian structure of Bianchi groups}
\newcommand{\TITLERUNNING}{The Apollonian structure of Bianchi groups}
\newcommand{\DATE}{\today}

\theoremstyle{plain} 
\newtheorem{theorem}{Theorem} 
\newtheorem{conjecture}[theorem]{Conjecture}
\newtheorem{proposition}[theorem]{Proposition}
\newtheorem{lemma}[theorem]{Lemma}

\theoremstyle{definition}
\newtheorem{definition}[theorem]{Definition}

\theoremstyle{remark}

\numberwithin{theorem}{section}

  {\end{list}}

  {\end{list}}

\newcommand{\tightoverset}[2]{%
  \mathop{#2}\limits^{\vbox to -.5ex{\kern-1.05ex\hbox{$#1$}\vss}}}

\def\Acal{{\mathcal A}}
\def\Bcal{{\mathcal B}}
\def\Ccal{{\mathcal C}}
\def\Dcal{{\mathcal D}}

\def\Gcal{{\mathcal G}}

\def\Pcal{{\mathcal P}}

\def\Scal{{\mathcal S}}

\newcommand{\CC}{\mathbb{C}}

\newcommand{\HH}{\mathbb{H}}

\newcommand{\PP}{\mathbb{P}}
\newcommand{\QQ}{\mathbb{Q}}
\newcommand{\RR}{\mathbb{R}}
\newcommand{\MM}{\mathbb{M}}
\newcommand{\ZZ}{\mathbb{Z}}

\def \bfv{{\mathbf v}}

\renewcommand{\gcd}{{\operatorname{gcd}}}

\newcommand{\MOD}[1]{~(\textup{mod}~#1)}
\renewcommand{\pmod}{\MOD}
\newcommand{\Norm}{\operatorname{N}}

\newcommand\PSL{\operatorname{PSL}}

\newcommand\PGL{\operatorname{PGL}}
\newcommand\OK{\mathcal{O}_K}

\newcommand{\OSK}{\widehat{\Scal}_K}
\newcommand{\SK}{{\Scal}_K}

\newcommand{\Mob}{\text{M\"ob}}
\newcommand{\Mobplus}{\text{M\"ob}_+}

\newcommand{\GM}{\text{GM}^*}

\newcommand{\Circ}{\mathcal{C}irc}
\newcommand{\Zcl}{\operatorname{Zcl}}

\newcommand\sensual{AppPrevious}

\newcommand\ntone{MR1971245}

\newcommand\gttwo{MR2183489}
\newcommand\gtone{MR2173929}

\newcommand\SarnakMonthly{MR2800340}

\newcommand\Conway{MR1478672}

\newcommand\Oh{MR3173439}
\newcommand\Soddy{Soddy}
\newcommand\circlesummer{circlesummer}
\newcommand\Senia{Senia}
\newcommand\KontSphere{KontSphere}
\newcommand\BFuchs{MR2813334}
\newcommand\GuMa{MR2675919}
\newcommand\Glowing{Glowing}
\newcommand\Indra{MR1913879}
\newcommand\Xin{Xin}
\newcommand\BuGrGuMa{MR2679050}
\newcommand\REUG{circlesummer}
\newcommand\BishopJones{MR1484767}
\newcommand\BourgainKontorovich{MR3211042}
\newcommand\FuchsBulletin{MR3020827}
\newcommand\FuchsSanden{MR2859897}

\newcommand\Kocik{kocikminkowski}
\newcommand\Lewis{MR656449}
\newcommand\KontBulletin{MR3020826}
\newcommand\Schmidt{MR0422168}
\newcommand\SchmidtEisenstein{MR698164}
\newcommand\SchmidtEleven{MR0485715}
\newcommand\SchmidtTwo{MR2811562}
\newcommand\SchmidtFarey{MR0245525}
\newcommand\EGM{MR1483315}

\title[\TITLERUNNING]{\vspace*{-1.3cm} \TITLE}

\author{Katherine E. Stange}

\date{\DATE}
\address{%
Department of Mathematics, University of Colorado,
Campux Box 395, Boulder, Colorado 80309-0395}
\email{kstange@math.colorado.edu}

\thanks{
The author's work was supported by the National Security Agency H98230-14-1-0106. 
}

\begin{document}


\begin{abstract}
        We study the orbit of $\widehat{\RR}$ under the M\"obius action of the Bianchi group $\PSL_2(\OK)$ on $\widehat{\CC}$, where $\OK$ is the ring of integers of an imaginary quadratic field $K$.  The orbit $\SK$, called a Schmidt arrangement, is a geometric realisation, as an intricate circle packing, of the arithmetic of $K$.  We give a simple geometric characterisation of certain subsets of $\SK$ generalizing Apollonian circle packings, and show that $\SK$, considered with orientations, is a disjoint union of all primitive integral such \emph{$K$-Apollonian packings}.  These packings are described by a new class of thin groups of arithmetic interest called \emph{$K$-Apollonian groups}.  We make a conjecture on the curvatures of these packings, generalizing the local-to-global conjecture for Apollonian circle packings.
\end{abstract}

\maketitle

\section{Introduction}

Let $K$ be an imaginary quadratic field with ring of integers $\OK$.  The Bianchi group $\PSL_2(\OK)$ acts on the extended complex plane $\widehat{\CC} = \CC \cup \{\infty \}$ via M\"obius transformations.  Its action permutes the collection of circles of $\widehat{\CC}$ (where straight lines are considered circles through $\infty$).  The orbit of the extended real line $\widehat{\RR}$ under $\PSL_2(\OK)$ forms a delicately intertwined collection of circles, some examples of which are shown in Figures \ref{fig:exampleSK}, \ref{fig:exampleSK2} and \ref{fig:exampleSK3}.  This is the \emph{Schmidt arrangement} of $K$, denoted $\SK$, named, in \cite{VisOne}, in honour of the work of Asmus Schmidt generalizing continued fractions to the complex setting \cite{\SchmidtFarey,\Schmidt,\SchmidtEleven,\SchmidtEisenstein,\SchmidtTwo}.  The individual images of $\widehat{\RR}$ are called $K$-Bianchi circles.

The Schmidt arrangement naturally exhibits various aspects of the number theory of $\OK$, and this was the topic of study of an earlier paper of the author \cite{VisOne}.  Most vividly, $\SK$ is connected if and only if $\OK$ is Euclidean.

\begin{figure}
        \includegraphics[height=3in]{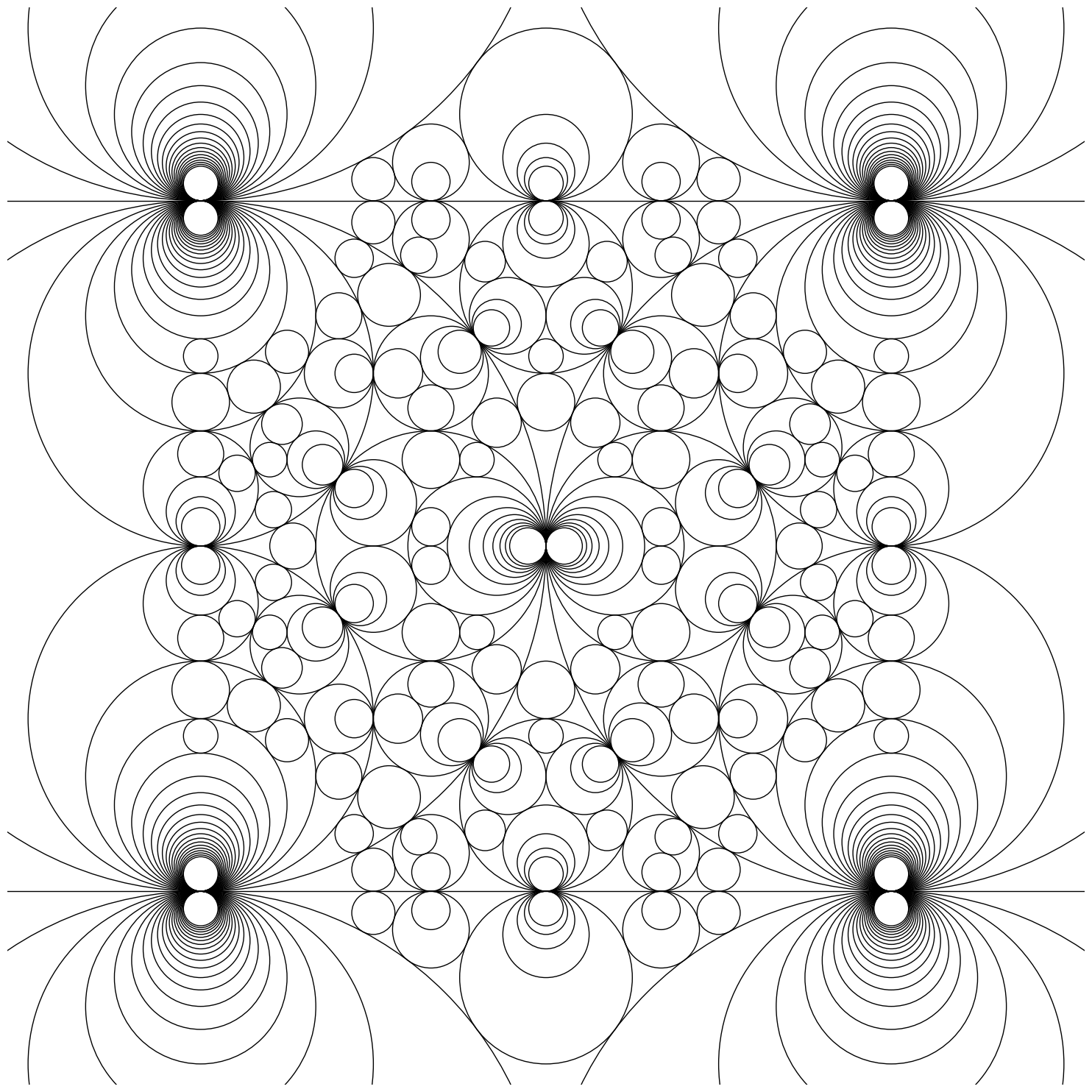}
        \caption{The Schmidt arrangement of ${\Scal}_{\QQ(i)}$.  
        The image includes those circles of curvature $\le 20$ intersecting the closure of the fundamental parallelogram of the ring of integers.}
\label{fig:exampleSK}
\end{figure}

The author's interest arose from $\Scal_{\QQ(i)}$ (Figure \ref{fig:exampleSK}), which made its appearance in \cite{\gttwo} as an \emph{Apollonian superpacking} in the study of Apollonian circle packings (apparently the authors of \cite{\gttwo} were unaware of the relation to Schmidt's work; the picture is the same but the definition is different).  This paper gives a new description of the relationship between $\Scal_{\QQ(i)}$ and Apollonian circle packings, and introduces a family of new circle packings arising from the Schmidt arrangements of imaginary quadratic fields in an analogous way.  These packings, which share many of the remarkable arithmetic properties of Apollonian circle packings, are isolated by a simple geometric characterisation, and give rise to associated thin subgroups of $\PSL_2(\OK)$; we refer to this as the Apollonian structure of $\PSL_2(\OK)$.  

\begin{figure}
        \raisebox{-.49\height}{\includegraphics[width=1in]{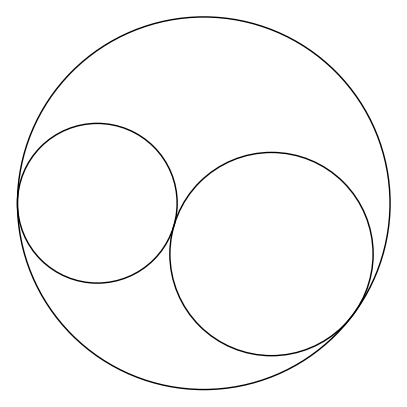}
\includegraphics[width=1in]{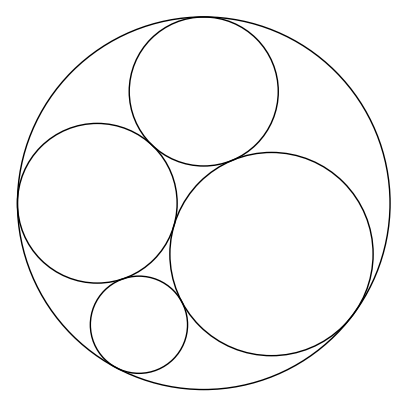}
\includegraphics[width=1in]{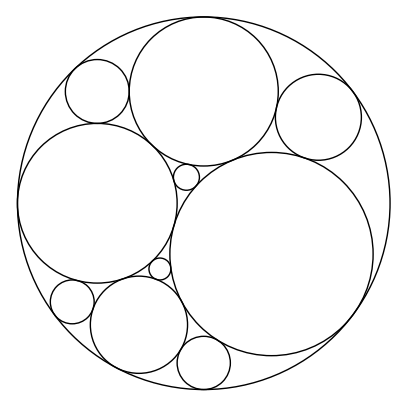}}
$\;\ldots\;$
\raisebox{-.49\height}{\includegraphics[width=1in]{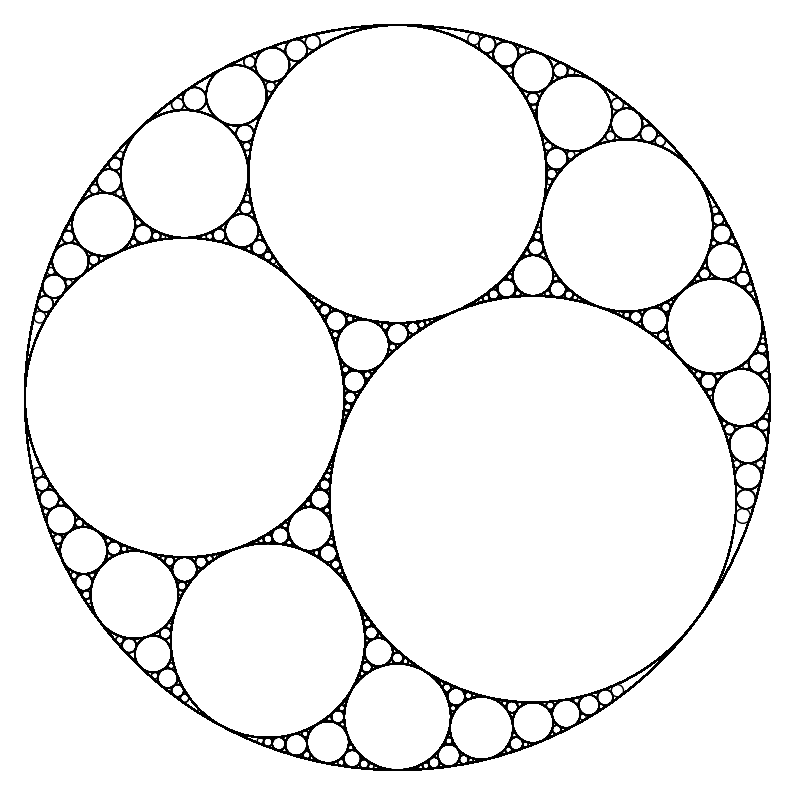}}
\caption{The iteration process generating an Apollonian circle packing.} 
\label{fig:app}
\end{figure}

A \emph{Descartes quadruple} of circles in $\widehat{\CC}$ is a quadruple such that every pair of circles is tangent.  The curvatures (inverse radii) $a,b,c,d$ of such a quadruple satisfy the \emph{Descartes quadratic relation}, famously stated by Descartes in a 1643 letter to Princess Elizabeth of Bohemia:
\begin{equation}
        \label{eqn:descquad}
        (a+b+c+d)^2 = 2(a^2 + b^2 + c^2 + d^2).
\end{equation}
The study of Descartes quadruples has a lively history and there are several excellent expositions; see for example \cite{\FuchsBulletin, \SarnakMonthly}.  If one begins with three mutually tangent circles with curvatues $a,b,c$, then there are exactly two circles, called \emph{Soddy circles}, which complete the triple to a Descartes quadruple, with curvatures $d$ and $d'$ solving \eqref{eqn:descquad}; these satisfy
\[
        d + d' = 2(a+b+c).
\]
Adding these two new circles to our original triple, we have a set of 5 circles.  For each mutually tangent triple in the set, we can again find two Soddy circles which complete it to a quadruple.  Any of these which are not already included in our collection are now added, thereby expanding the set of circles.  If we continue this process ad infinitum, we obtain an infinite collection of circles which is called an \emph{Apollonian circle packing}.  See Figure \ref{fig:app}.  The remarkable fact for the number theorist is that if the first four circles had integer curvatures, then the generation rule entails that every circle in the packing has an integral curvature.  For an example, see Figure \ref{fig:example6}.  The natural question is to determine which integers occur as curvatures.

As it happens, $\mathcal{S}_{\QQ(i)}$ is the union of all possible primitive, integral Apollonian circle packings, considered up to suitable symmetries and an appropriate scaling (integral refers to integer curvatures; primitive means they share no common factor).  This is a result of \cite[Theorem 6.3]{\gttwo} (using the definition of the Schmidt arrangement as an Apollonian superpacking), further studied in \cite{\sensual}.  In particular, the curvatures of $\QQ(i)$-Bianchi circles are all integral.  

Apollonian circle packings have generated great interest recently, in large part because of their connection to thin groups.  The central conjecture is a local-global principle for the curvatures of a packing.  

\begin{conjecture}[Graham, Lagarias, Mallows, Wilks and Yan \cite{\ntone}, Fuchs and Sanden \cite{\FuchsSanden}]
        \label{conj:main}
        Let $\Pcal$ be a primitive integral Apollonian circle packing, and let $S$ be the set of residue classes modulo $24$ of the curvatures in $\Pcal$.  Then all sufficiently large integers with residues in $S$ occur as curvatures in $\Pcal$.
\end{conjecture}

\begin{figure}
        \includegraphics[height=5in]{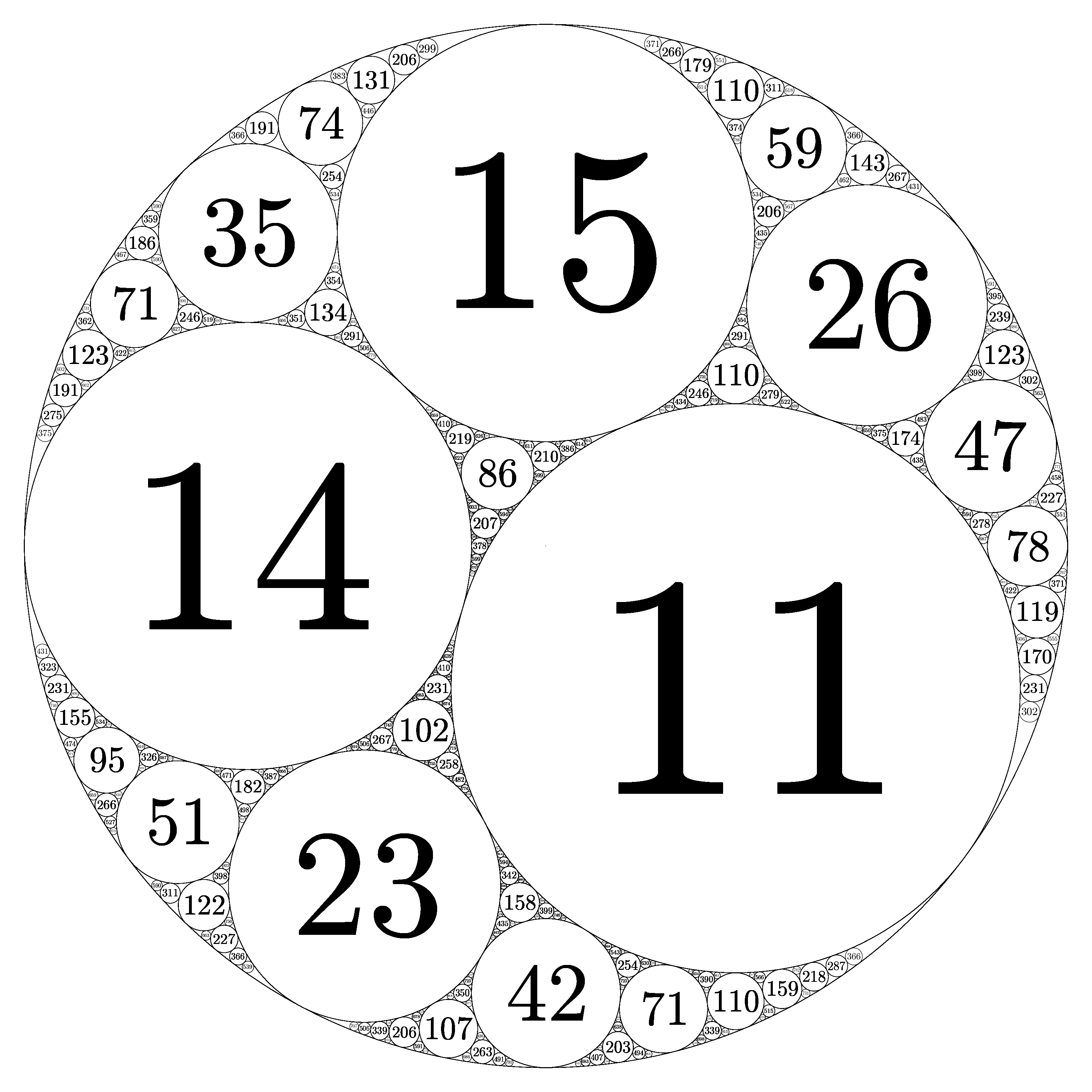}
        \caption[Apollonian packing with curvatures]{A primitive integral Apollonian packing with curvatures shown.  The outer circle has curvature $-6$ (indicating that its interior is outside).}
\label{fig:example6}
\end{figure}

Significant progress has been made toward this conjecture, most notably that it holds for a set of integers of density one \cite{\BourgainKontorovich} (positive density was first shown in \cite{\BFuchs}).  For an excellent overview and further references, see \cite{\FuchsBulletin}; see also the series of papers \cite{\ntone, \gtone,\gttwo} which are central to the field, and the exposition \cite{\SarnakMonthly}.  For the related question of the multi-set of integral curvatures appearing in a packing, a gateway to the literature is the survey \cite{\Oh}.   

These results depend on an analysis of the Apollonian group, a matrix group which describes the relations between curvatures of tangent circles.  The Apollonian group is of infinite index in its own Zariski closure, in other words, it is a \emph{thin group}.  Thin groups are not as accessible as arithmetic groups, but, remarkably, still share some of their properties, most notably a version of strong approximation.  The Apollonian group has garnered so much interest in part because of its position as a `naturally occurring' thin group of arithmetic interest.  For an overview of the arithmetic of thin groups, and the discoveries rapidly unfolding in recent years, see \cite{\KontBulletin}.  It is one of the principal goals of this paper to place the Apollonian group in a new `naturally ocurring' infinite family of thin groups of arithmetic interest.

\begin{figure}
        \includegraphics[height=2.7in]{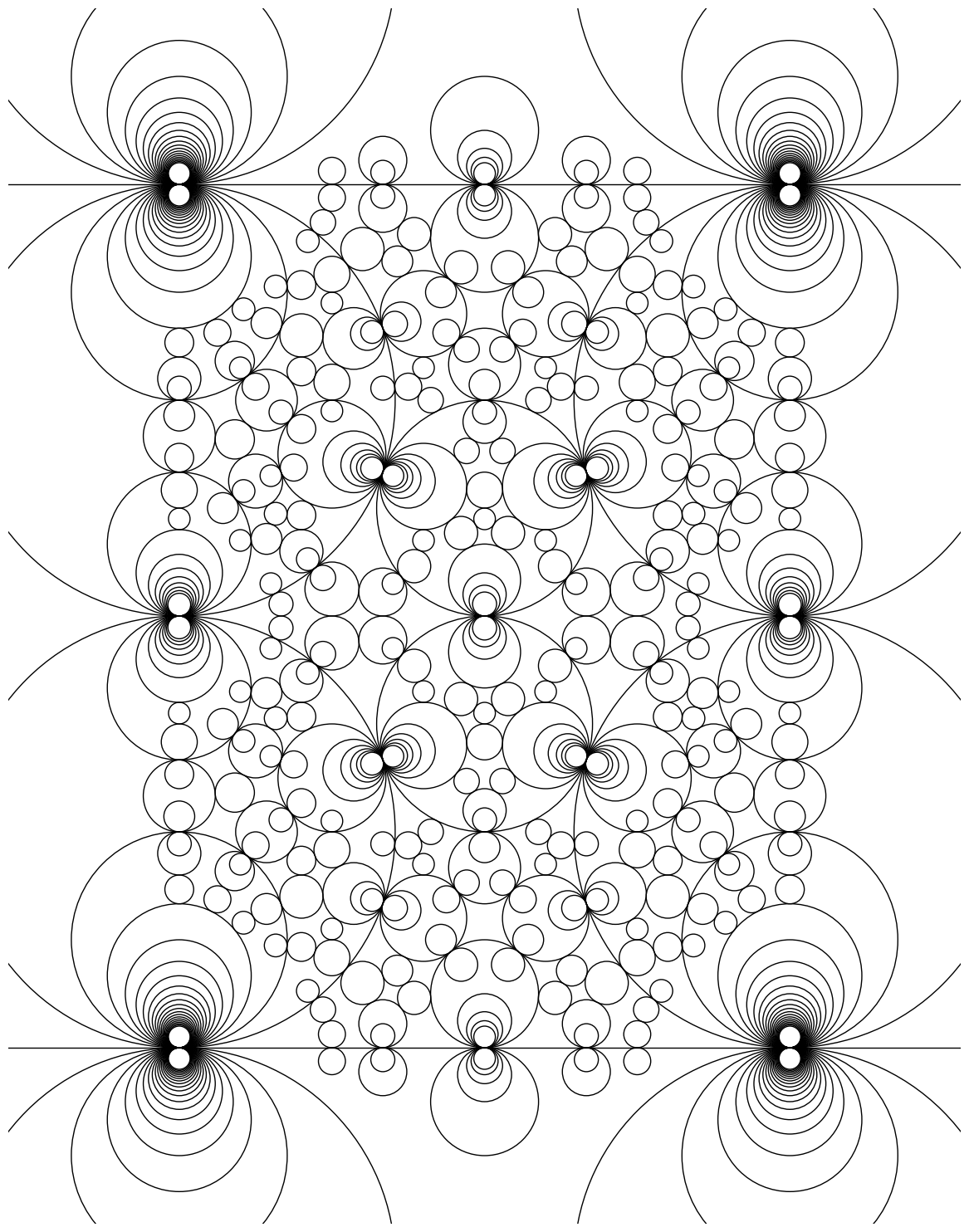} \quad
        \includegraphics[height=2.7in]{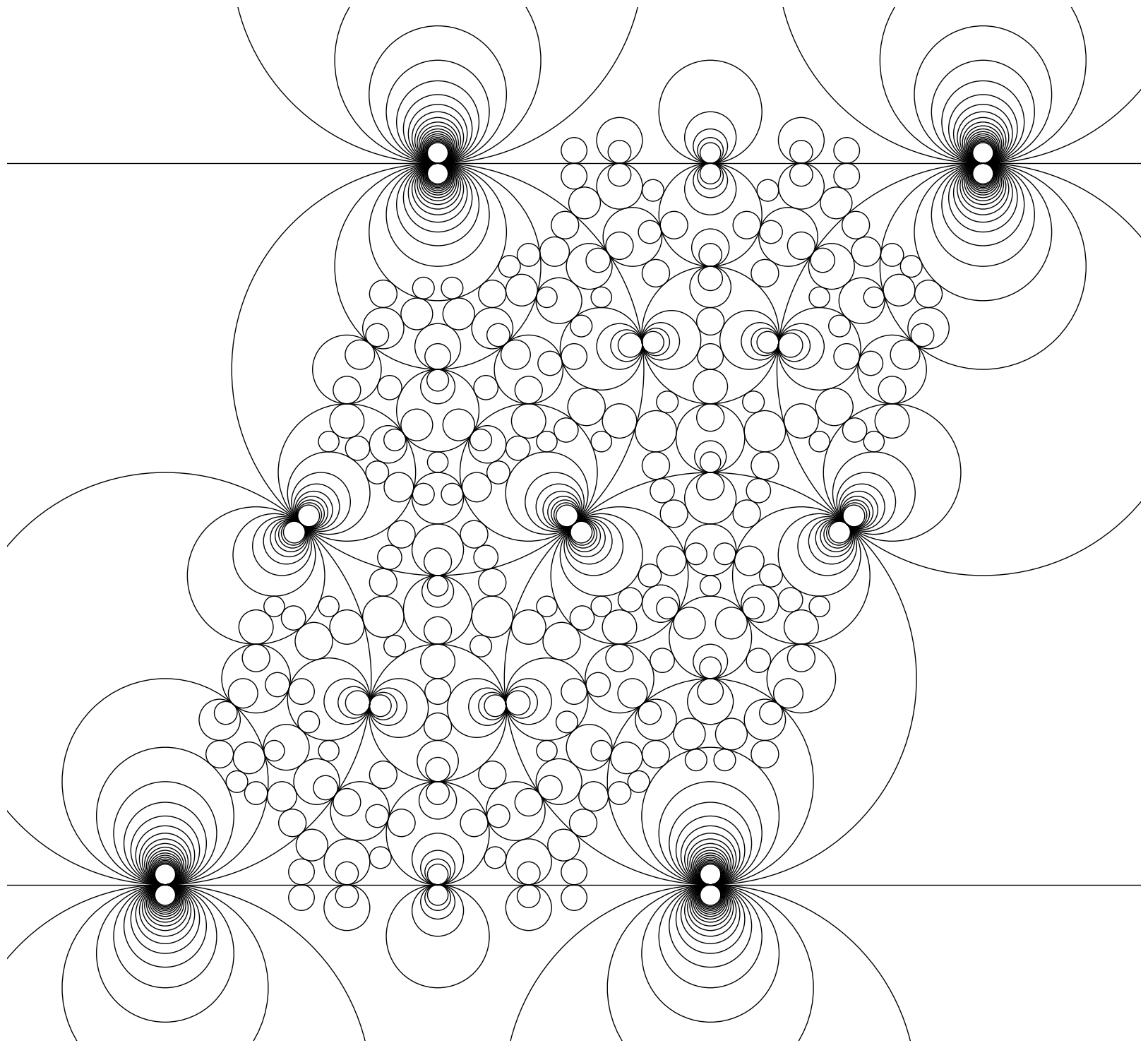} \\
        \includegraphics[height=2.9in]{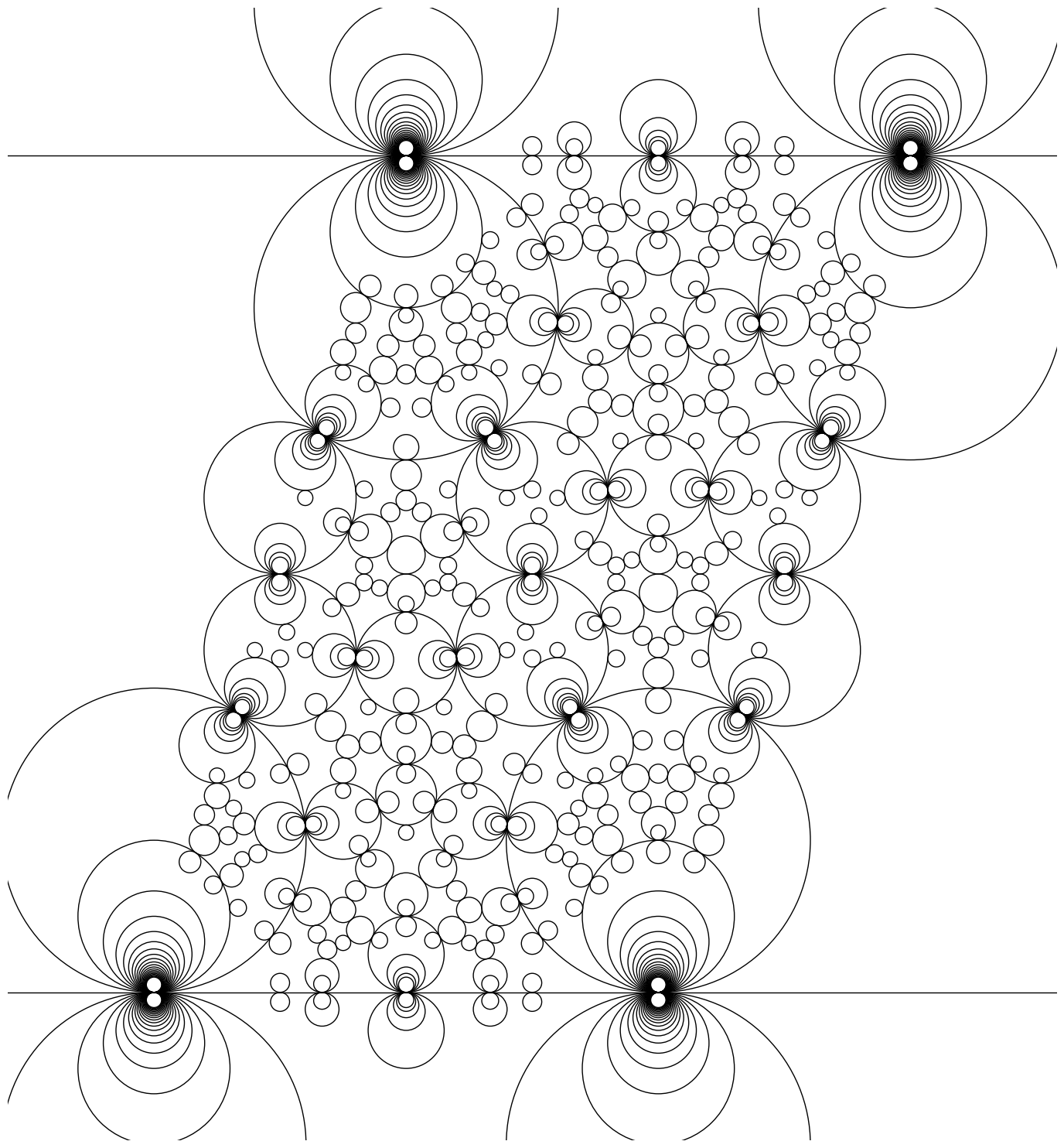} \quad
        \includegraphics[height=2.9in]{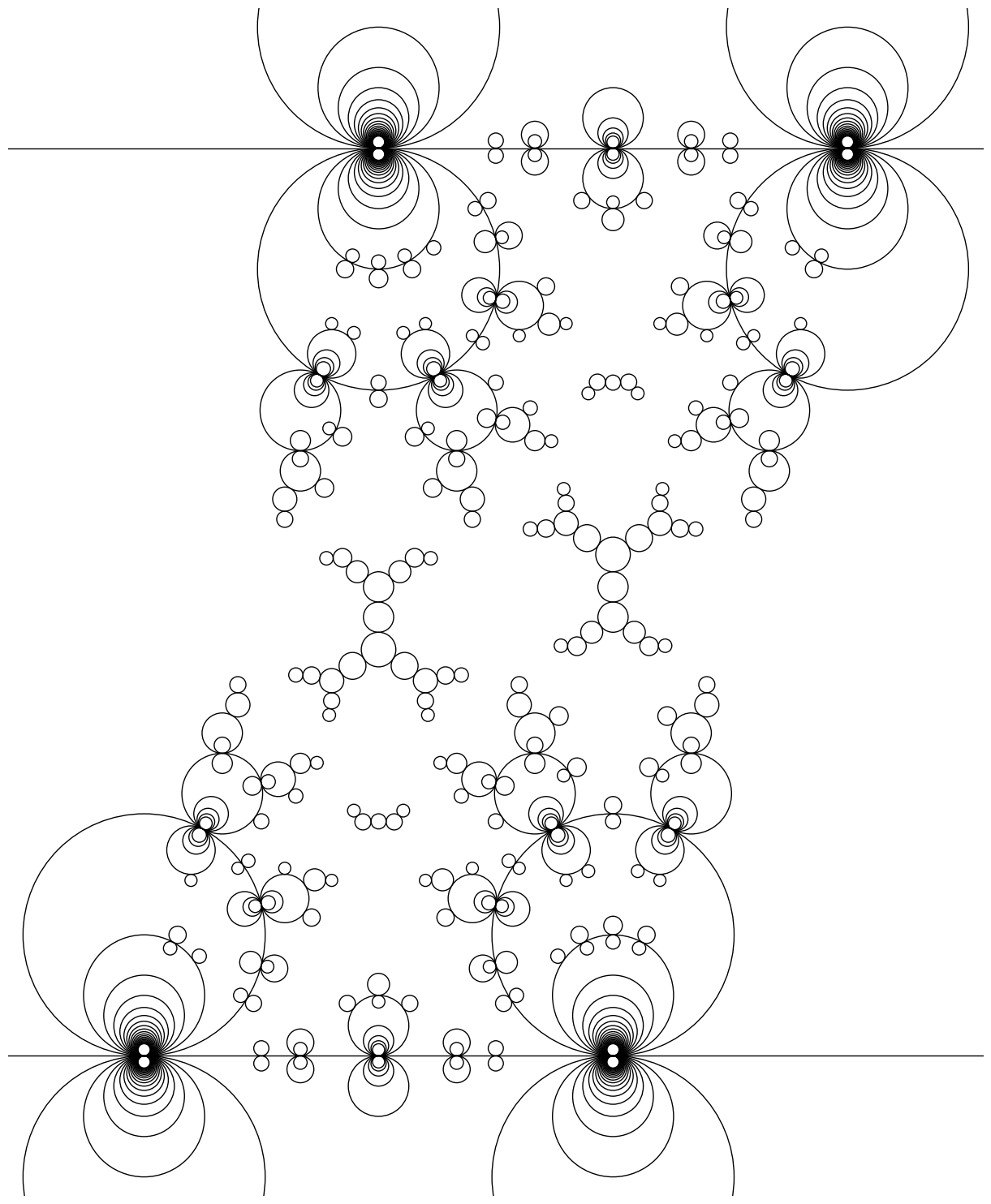} \\
        \caption{Schmidt arrangements $\SK$ of imaginary quadratic fields $K$.  Clockwise from top left:  $\QQ(\sqrt{-2})$, $\QQ(\sqrt{-7})$, $\QQ(\sqrt{-15})$, $\QQ(\sqrt{-11})$.  In each case, the image includes those circles of curvature $\le 20$ intersecting the closure of the fundamental parallelogram of $\OK$.}
\label{fig:exampleSK2}
\end{figure}

We begin by defining $K$-Apollonian packings for imaginary quadratic fields $K \neq \QQ(\sqrt{-3})$.

\begin{definition}
        \label{def:apppack}
        One says that a collection of circles $\mathcal{P}$ \emph{straddles} a circle $C$ if it intersects both the interior and exterior of $C$ nontrivially. 
        We say that a collection of circles $\mathcal{P}$ is \emph{tangency-connected} if the graph whose vertices are circles and whose edges indicate tangencies is a connected graph.
        We define a \emph{$K$-Apollonian packing} to be any maximal tangency-connected subset $\mathcal{P}$ of circles in $\mathcal{S}_K$ under the condition that $\mathcal{P}$ does not straddle any circle of $\Scal_K$.
\end{definition}

For an example, see Figure \ref{fig:app-in-sk}. 

\begin{figure}
        \includegraphics[height=3in]{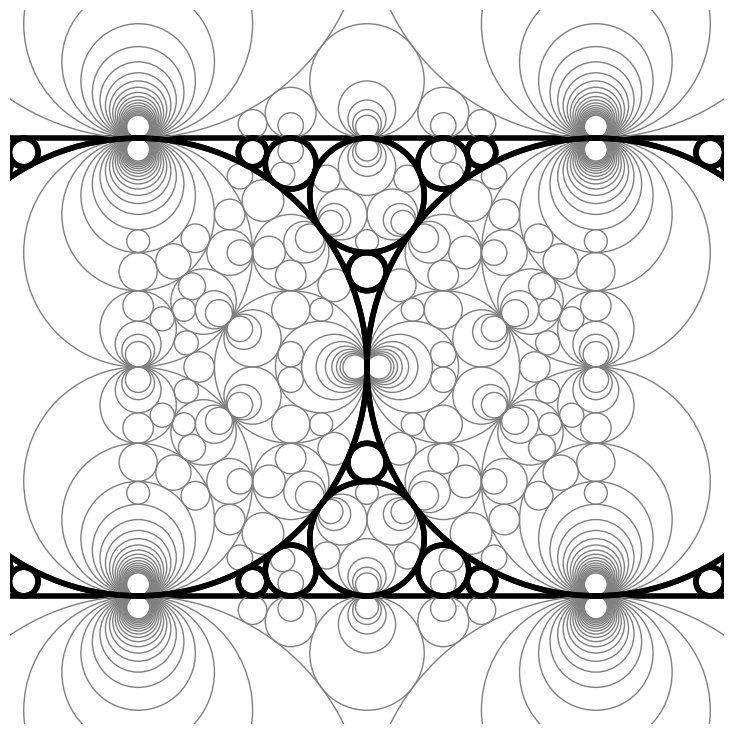}
        \caption[Apollonian packing (shown in darker lines) lying in the Schmidt arrangement of $\QQ(i)$]{An Apollonian packing (shown in darker lines) lying in the Schmidt arrangement of $\QQ(i)$ according to Definition \ref{def:apppack}.}
\label{fig:app-in-sk}
\end{figure}

\begin{figure}
        \includegraphics[height=5in]{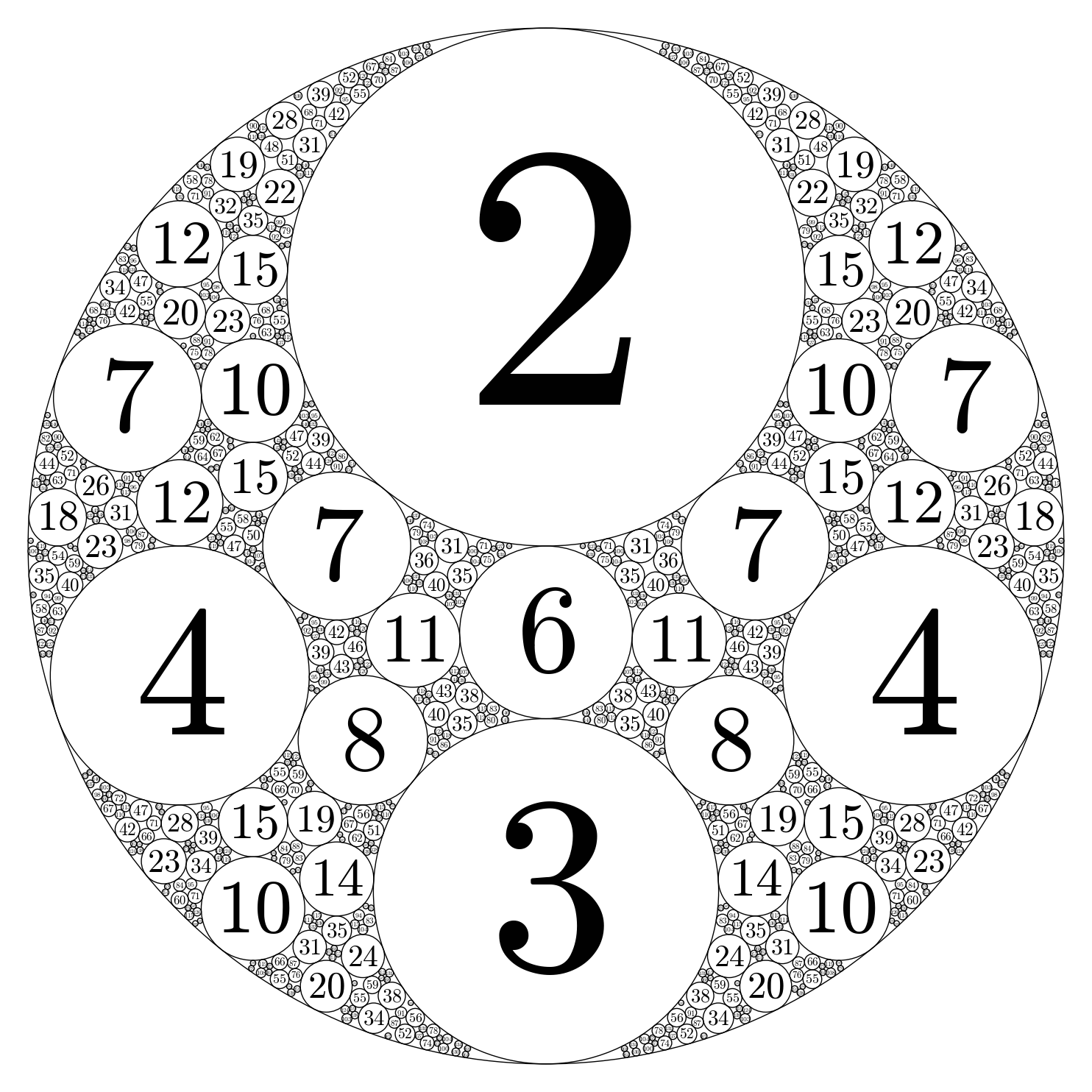} 
        \caption[$K$-Apollonian packing with curvatures]{An example $\QQ(\sqrt{-2})$-Apollonian packing, shown with reduced curvatures (the outer circle has reduced curvature $-1$).  In this case, it is conjectured that all sufficiently large integers that are not congruent to $1$ modulo $4$ appear as reduced curvatures.}
\label{fig:kapp-curvs}
\end{figure}

Note that, in particular, this implies that the circles in $\mathcal{P}$, oriented appropriately, can have disjoint interiors.  An equivalent definition of a $K$-Apollonian packing is that it corresponds to a connected component of the \emph{immediate tangency graph}, the graph whose vertices are oriented circles, and whose edges indicate \emph{immediate tangency}, that is, tangency between two circles of disjoint interiors which straddle no other $K$-Bianchi circle (Proposition \ref{prop:twodefs}).

Now we can state the fundamental relationship between Schmidt arrangements and $K$-Apollonian packings.

\begin{theorem}
        \label{thm:disjointunion}
        For an imaginary quadratic field $K \neq \QQ(\sqrt{-3})$, each circle of $\SK$ participates in exactly two $K$-Apollonian packings:  one disjoint with its interior and one disjoint with its exterior.  In particular, $\Scal_K$ is equal to the union of all $K$-Apollonian packings.
\end{theorem}

The Eisenstein case ($K = \QQ(\sqrt{-3})$) presents the difficulty that $\Scal_K$ has circles intersecting other than tangently, as in Figure \ref{fig:exampleSK3}.  {\bf For this reason, throughout the paper, we will assume that $K \neq \QQ(\sqrt{-3})$.}  This special case is being further investigated in \cite{\REUG}.

Let $\Delta < 0$ be the discriminant of $K$.  The curvature of a $K$-Bianchi circle is of the form $n \sqrt{-\Delta}$, for some $n \in \ZZ$; we may call $n$ the \emph{reduced curvature} of the circle.  With this definition, we may extend Conjecture \ref{conj:main} for Apollonian circle packings.

\begin{conjecture}
        \label{conj:new}
        Let $K$ be a imaginary quadratic field with $\Delta \neq -3$.
        Let $\Pcal$ be a $K$-Apollonian packing, and for any $M \in \ZZ$, let $S_M$ be the set of residue classes modulo $M$ of the reduced curvatures in $\Pcal$.  There exists an $M$ dividing $24$ such that all sufficiently large integers whose residues modulo $M$ lie in $S_M$ occur as curvatures in $\Pcal$.

        Furthermore, the minimal sufficient $M$ is determined by the following formulae for its valuations with respect to $2$ and $3$:
        \[
                v_2(M) = \left\{ 
                        \begin{array}{ll}
                        3 & \Delta \equiv 28 \pmod{32} \\
                        2 & \Delta \equiv 8,12,20,24 \pmod{32} \\
                        1 & \Delta \equiv 0,4,16 \pmod{32} \\
                        0 & \mbox{otherwise} \\
                \end{array} \right. , \quad
                v_3(M) = \left\{ 
                        \begin{array}{ll}
                        1 & \Delta \equiv 5,8 \pmod{12} \\
                        0 & \mbox{otherwise} \\
                \end{array} \right. .
        \]
\end{conjecture}

Further detailed predictions and supporting evidence is given in Section \ref{sec:curvatures}.  See Figure \ref{fig:kapp-curvs} for an example packing with curvatures shown.

The bulk of the paper is devoted to the study of $K$-Apollonian packings.
In particular, we define \emph{$K$-Apollonian groups} for every $K \neq \QQ(\sqrt{-3})$.  Let $\Mob$ denote the group of M\"obius transformations, including $\mathfrak{c}$, complex conjugation.  Then a \emph{$K$-Apollonian group} is one satisfying the conclusions of the following theorem.

\begin{theorem}[Summary of results of Sections \ref{sec:app}--\ref{sec:11}]
        Let $K \neq \QQ(\sqrt{-3})$ be an imaginary quadratic field.  Then there exists a finitely generated Kleinian group ${\Acal} < \langle \PSL_2(\OK), \mathfrak{c} \rangle < \Mob$ with the following properties:
        \begin{enumerate}
                \item The limit set of $\Acal$ is the closure of the $K$-Apollonian packing containing $\widehat{\RR}$.
                \item Any $K$-Apollonian packing is the orbit under ${\Acal}$ of some finite collection of circles.
                \item ${\Acal}$ is of infinite index in its Zariski closure (which is either $SO_{3,1}$ or $O_{3,1}$, under the isomorphism $\Mob \cong O^+_{3,1}(\RR)$).  In other words, it is thin.
                \item It is possible to define \emph{clusters} of circles, being unordered collections of $4 \le n < \infty$ circles in a certain geometric arrangement, so that the set of clusters in any $K$-Apollonian packing is a principal homogeneous space\footnote{By a \emph{prinicipal homogeneous space for a group $G$} we mean a non-empty set on which the group $G$ acts freely and transitively.} for $\Acal$.
        \end{enumerate}
\end{theorem}

The last property deserves some further explanation.  In the study of traditional Apollonian circle packings, the `clusters' are Descartes quadruples (any four circles which are mutually tangent).  We have seen that, given three mutually tangent circles, there are exactly two Descartes quadruples containing these three.  Therefore, given one Descartes quadruple and one circle of that quadruple, there is a \emph{swap} that replaces that circle with the unique choice that gives a new Descartes quadruple.  For each quadruple, there are four such swaps.  It turns out that the space of Descartes quadruples is a principal homogeneous space for the traditional Apollonian group, whose four generators encode the four swaps.  This group has two manifestations, sometimes called \emph{algebraic} and \emph{geometric}.  The geometric manifestation is the one discussed in the theorem above, but both aspects are discussed in this paper.

There are many $K$-Apollonian groups (even for the traditional Apollonian circle packing), and we give some particularly simple and symmetric examples with pleasing presentations.  In particular, in cases where $\OK$ is Euclidean, we give an example which is a free product of finitely many copies of $\ZZ/2\ZZ$.
In each case, we give generators explicitly, and describe their interpretation as `{swaps}' on `{clusters}'.

Various Apollonian-like circle packings have appeared in the literature.  Guettler and Mallows studied \emph{generalized Apollonian packings}, which have also come to be called \emph{Apollonian 3-packings} or \emph{octahedral Apollonian packings}, circle packings in which each curvilinear triangle is packed with three (not just one) new circles \cite{\GuMa}.  They generalize much of the general theory of Apollonian packings, including the analogue of the Apollonian super-packing \cite[Figure 4]{\GuMa}.  Zhang later extended some of the local-to-global theory to Apollonian 3-packings, most notably a density-one result on curvatures \cite{\Xin}; his paper contains an Apollonian-type group for the $3$-packings.  It appears that the Apollonian-$3$-packings are dual to $\QQ(\sqrt{-2})$-Apollonian packings, in the following sense:  in $\Scal_{\QQ(\sqrt{-2})}$, every loop of four tangent circles has its tangency points along a new circle; this collection of new circles is the super-$3$-packing suitably rotated and scaled\footnote{The inspiration for this observation arises from the beautiful \emph{Glowing Limit} poster \cite{\Glowing} advertising \emph{Indra's Pearls} \cite{\Indra}, which happens to depict this duality, albeit without reference to $\QQ(\sqrt{-2})$.}.

Butler, Graham, Guettler and Mallows study circle packings given by certain clusters (which they call \emph{configurations}) and recursive \emph{rules} for filling curvilinear triangles \cite{\BuGrGuMa}.  They are particularly interested in which fields are needed to define the curvatures and centres of such packings.  Some of the $K$-Apollonian packings in the present paper come under their rubric, but not all.  For example, some $K$-Apollonian packings are formed recursively by filling curvilinear $4$- or $6$-gons.

The study of a three-dimensional analogue of Apollonian circle packings with integer curvatures began with Soddy's \emph{Nature} article of 1937 entitled \emph{The Bowl of Integers and the Hexlet} \cite{\Soddy}.  The three-dimensional case has the distinction that the associated local-global conjecture is known to hold \cite{\KontSphere}.  Nakamura extends this result to \emph{orthoplicial Apollonian sphere packings} \cite{Nak}, which generalise the Apollonian-3-packing to three dimensions in the same way Soddy's bowl of integers generalises the classical Apollonian circle packing.

The methods of the paper give rise to a few results that may be of independent interest.  The following theorem strengthens the results of \cite{VisOne}.

\begin{theorem}[Theorem \ref{thm:16noloops}]
        \label{thm:forest}
        Suppose $\OK$ is non-Euclidean.  Then the immediate tangency graph of the Schmidt arrangement for $K$ is an infinite forest of trees of infinite valence.  In particular, the tangency graph of a $K$-Apollonian packing is a tree.
\end{theorem}

The following strengthens results of \cite{MR2805033}.

\begin{theorem}[Theorem \ref{thm:Ethin}]
        Whenever $\OK$ is non-Euclidean, the subgroup of $\PSL_2(\OK)$ generated by elementary matrices is a thin group.
\end{theorem}

And finally, the paper contains in Section \ref{sec:topographical} a discussion of \emph{topographical groups}.  These are subgroups of $\PGL_2(\ZZ)$ for which unordered superbases are a principal homogeneous space.  There are only two normal such subgroups, which are isomorphic under the outer automorphism of $\PGL_2(\ZZ)$.  They have a Cayley graph isomorphic to the \emph{topograph} of Conway and Fung \cite{\Conway}.  

It is natural to ask about extending the results of the present paper to Schmidt arrangements for congruence subgroups of the Bianchi group, which is work in progress \cite{\circlesummer}, and Bianchi groups for non-maximal orders, for which Schmidt arrangements have been studied by Sheydvasser \cite{\Senia}.

{\bf Acknowledgements.} The author is indebted to Arthur Baragar, Elena Fuchs, Jerzy Kocik, Alex Kontorovich, Hee Oh, and Peter Sarnak for helpful discussions in the course of this work.  Special thanks to Jonathan Wise for \LaTeX\, help with Proposition \ref{prop:algebraic}.

{\bf A note on the figures.} The figures in this paper were produced with Sage Mathematics Software \cite{Sage}.  Figures \ref{fig:app}, \ref{fig:exampleSK2}, \ref{fig:exampleSK3} and \ref{fig:ghostchain} appeared previously in \cite{AppPrevious, VisOne}.

\begin{figure}
        \includegraphics[height=3in]{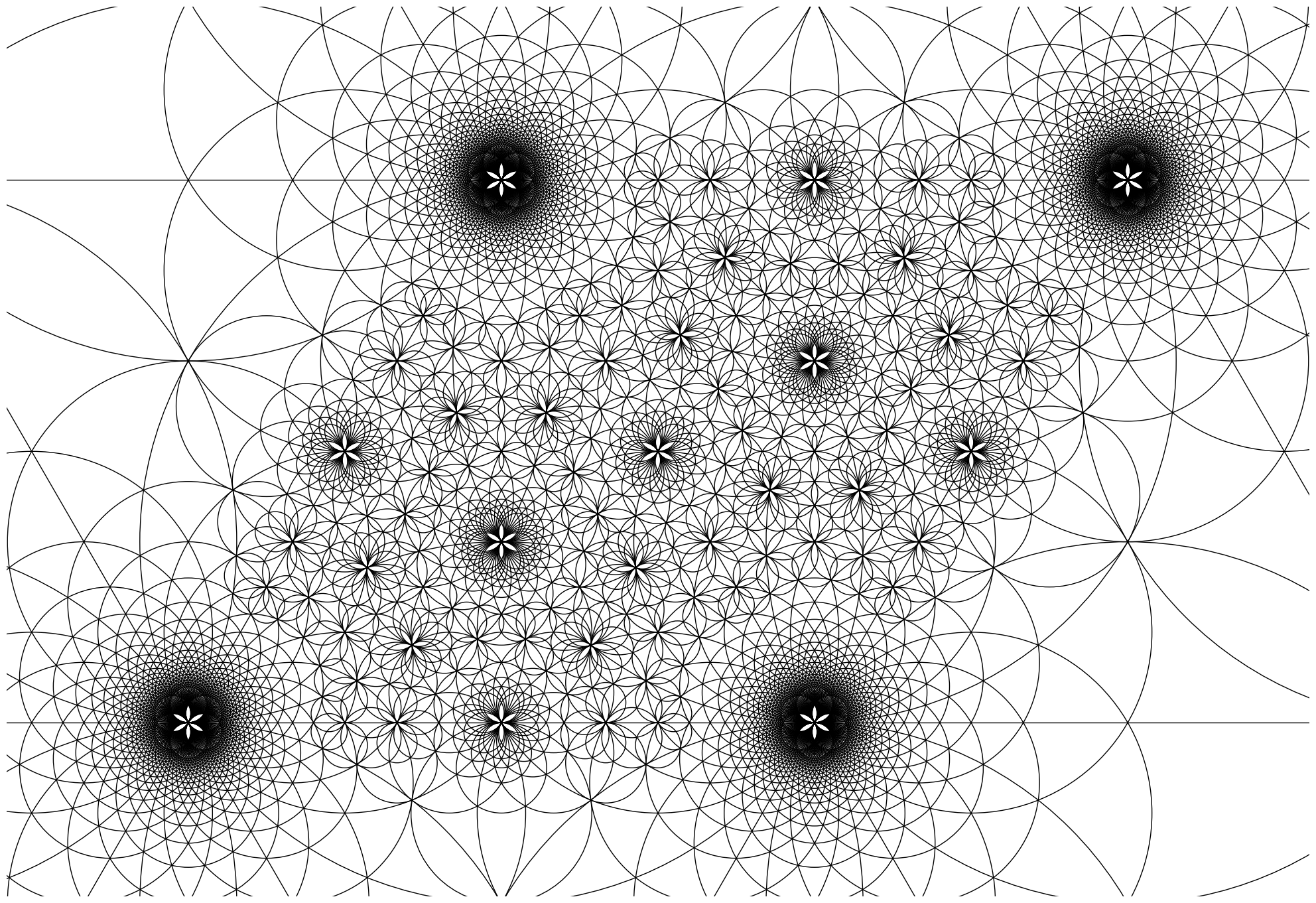} 
        \caption{Schmidt arrangement $\Scal_{\QQ(\sqrt{-3})}$.
        The image includes those circles of curvature bounded by $20$ intersecting the fundamental parallelogram of the ring of integers or its boundary.}
\label{fig:exampleSK3}
\end{figure}

\section{Notations}
\label{sec:notation}

Throughout the paper, $K$ is a imaginary quadratic field with discriminant $-3 \neq \Delta<0$ and ring of integers $\OK$.  The ring $\OK$ has an integral basis $1, \tau$, where
\[
        \tau ^2 = \left\{\begin{array}{ll}
                \Delta/4 & \Delta \equiv 0 \pmod 4 \\
                \tau + (\Delta-1)/4 & \Delta \equiv 1 \pmod 4 \\
        \end{array}
        \right. .
\]
It is convenient to write
\[
        \epsilon = Tr(\tau) = \left\{\begin{array}{ll}
                        0 & \Delta \equiv 0 \pmod 4 \\ 
                        1 & \Delta \equiv 1 \pmod 4 \\ 
                \end{array}
                \right. .
\]
In other words,
\[
        \tau,\overline{\tau} = \frac{\epsilon}{2} \pm \frac{\sqrt{\Delta}}{2} .
\]
Write $N:K \rightarrow \QQ$ for the norm map $\alpha \mapsto \alpha\overline{\alpha}$.  In particular,
\[
        N(\tau) = -\frac{\Delta - \epsilon}{4}.
\]

We follow \cite{\gtone} in writing $\Mob$ for the \emph{conformal group}, the group of conformal maps of $\widehat{\CC}$, including M\"obius transformations and reflections (i.e. allowing complex conjugation); and writing $\Mobplus$ for the group of M\"obius transformations without reflections, which is isomorphic to $\PSL_2(\CC)$ via the usual matrix representation of a M\"obius transformation.  We also write $\GM:= \Mob \times \{ \pm I \}$ for the \emph{extended M\"obius group}.  The group $\Mob$ is isomorphic to the isometry group of three-dimensional hyperbolic space, $\HH^3$, via its action on the boundary of the upper-half-space model.  

{\bf Warning:  For the remainder of the paper, $K \neq \QQ(\sqrt{-3})$.}

\section{Preliminaries}

We recall some basic facts about Schmidt arrangements from \cite{VisOne}, to which the reader is referred for further details.  Let $K$ be an imaginary quadratic field.  Recall that throughout the paper, $K \neq \QQ(\sqrt{-3})$.  Among the remaining cases, $K=\QQ(i)$ is special in several regards, as we will describe.  

\begin{definition}
        If $M \in \PSL_2(\OK)$, then $M(\widehat{\RR})$ is called a \emph{$K$-Bianchi circle}.  The collection of all $K$-Bianchi circles is called the \emph{Schmidt arrangement}, denoted $\SK$.  We will also refer to the union of these circles as $\SK$, without fear of confusion.
\end{definition}

It is convenient to consider oriented circles.

\begin{definition}[{\cite[Definitions 3.1, 3.2, 3.3]{VisOne}}]
        An \emph{oriented circle} is a circle together with an \emph{orientation}, which is a direction of travel, specified as either \emph{positive/counterclockwise} or \emph{negative/clockwise}.   The \emph{interior} of an oriented circle is the area to your left as you travel along the circle according to its orientation (in other words, the region not containing $\infty$ for a positively oriented circle).  For lines (circles through $\infty$) besides $\widehat{\RR}$, positive orientation indicates travel in the direction of increasing imaginary part.  For $\widehat{\RR}$, positive orientation indicates travel to the right.  
        An \emph{oriented $K$-Bianchi circle} is an oriented circle whose underlying circle is a $K$-Bianchi circle.  Write $\OSK$ for the collection of oriented $K$-Bianchi circles.
\end{definition}

The map $\OSK \rightarrow \SK$ which forgets orientation is two-to-one.

        Let us also set the convention that $\widehat{\RR}$, when considered an oriented circle, denotes the positively oriented circle (whose interior is the upper half plane).  
The group $\PGL_2(\CC)$ has a natural action on oriented circles via M\"obius transformation, so that, with this convention, $M(\widehat{\RR})$ is naturally endowed with an orientation.  The stabilizer of $\widehat{\RR}$ is $\PSL_2(\RR)$.  Furthermore, by Proposition 3.4 of \cite{VisOne}, this action restricts, in the case $K \neq \QQ(i)$, to a transitive action of $\PGL_2(\OK)$ on oriented $K$-Bianchi circles, with $\operatorname{Stab}(\widehat{\RR}) = \PSL_2(\ZZ)$.  
In the case of $K = \QQ(i)$, $\PSL_2(\OK)$ is already transitive on oriented $K$-Bianchi circles with the same stabilizer.  If one considers the orbit of $\widehat{\RR}$ under $\PGL_2(\ZZ[i])$, one obtains a strictly larger collection of circles than the Schmidt arrangement.  However, the new circles are not interesting:  the set consists of two copies of the Schmidt arrangement at right angles.  See \cite[Section 3]{VisOne}.

Oriented circles are best described as parameterized by four real parameters, as follows.

\begin{proposition}[{\cite[Proposition 3.5]{VisOne}}]
        \label{prop:gencirc}
        Let $C$ be an oriented circle in $\widehat{\CC}$ (including those through $\infty$).  Then the circle $C$ can be given uniquely in the form
        \[
                \left\{ X/Y \in \widehat{\CC} : bX\overline{X} - a Y \overline{X} - \overline{a} X \overline{Y} + b' Y \overline{Y} = 0 \right\}.
        \]
        where the following hold:
        \begin{enumerate}
                \item \label{item1}$b,b' \in \RR$, $a\in \CC$,
                \item we have
                        \begin{equation}
                \label{eqn:cocurv}
                b'b = a\overline{a}-1,
        \end{equation}
\item\label{item3} $b$ has sign equal to the orientation of $C$, and;
\item\label{item4} if $b=0$, in which case $C$ must be a line, then $a$, as a vector, is a unit vector pointing from exterior to interior orthogonal to $C$.
        \end{enumerate}
        Furthermore, if $C'$ is the image of $C$ under $z \mapsto 1/z$, and $C$ has parameters $(b,b',a)$ according to the requirements above, then $C'$ has parameters $(b',b,\overline{a})$ according to the requirements above.
        Finally, \begin{enumerate}
                \item $b=0$ if and only if $\infty \in C$,
                \item $b'=0$ if and only if $0 \in C$,
                \item if $\infty \notin C$, then $C$ has radius $1/|b|$,
                \item if $\infty \notin C$, then $C$ has centre $a/b$.
        \end{enumerate}
\end{proposition}

\begin{definition}[{\cite[Definition 3.6]{VisOne}}]
        \label{def:curv}
        For any circle $C$ in $\widehat{\CC}$, expressing it as in Proposition \ref{prop:gencirc}, we call $b$ the \emph{curvature} (elsewhere sometimes called a \emph{bend}), $b'$ the \emph{co-curvature}, and $a$ the \emph{curvature-centre}.
\end{definition}

\begin{proposition}[{\cite[Proposition 3.7]{VisOne}}]
\label{prop:curvature}
Consider an oriented circle expressed as the image of $\widehat{\RR}$ under a transformation of the form
\[
M =
      \begin{pmatrix} \alpha & \gamma \\ \beta & \delta \end{pmatrix}, \quad \alpha, \beta, \gamma, \delta \in \CC, \;\; |\alpha\delta - \beta\gamma|=1.
\]
The curvature of the circle is given by
\[
        i(  \beta \overline{\delta} - \overline{\beta}\delta),
\]
the co-curvature of the circle is given by
\[
        i( \alpha \overline{\gamma} - \overline{\alpha}\gamma ),
\]
and the curvature-centre is given by
\[
        i(  \alpha \overline{\delta} - \gamma \overline{\beta} ).
\]
\end{proposition}

The curvature and co-curvature of an oriented $K$-Bianchi circle are integer multiples of $\sqrt{-\Delta}$; the integer alone will be referred to as the \emph{reduced curvature} or \emph{reduced co-curvature}, respectively.

\begin{proposition}[{\cite[Propositions 4.1 and 4.4]{VisOne}}]
        \label{prop:Kintersect}
        Two $K$-Bianchi circles may intersect only at $K$-points, and only tangently\footnote{This does not hold for $K = \QQ(\sqrt{-3})$.  The reader is reminded, one final time, that $K \neq \QQ(\sqrt{-3})$ throughout the paper.}.
\end{proposition}

It is useful to be very explicit about which circles are tangent.  We say that $\alpha, \beta \in \OK$ are coprime if the ideals they generate are coprime, i.e. $(\alpha)+(\beta)=(1)$.

\begin{proposition}[{\cite[Proposition 4.3]{VisOne}}]
        \label{prop:tangentfamilies}
        Let $\alpha/\beta \in K$ be such that $\alpha$ and $\beta$ are coprime.  Suppose $|\OK^*| = n$.  Then the collection of oriented $K$-Bianchi circles passing through $\alpha/\beta$ is a union of $n$ generically different $\ZZ$-families, one for each $u \in \OK^*$.  The family associated to $u$ consists of the images of $\widehat{\RR}$ under the transformations
\[
        \begin{pmatrix}
                \alpha & u \gamma + k \tau  \alpha\\
                 \beta & u \delta + k \tau \beta 
        \end{pmatrix}, \quad k \in \ZZ,
\]
where $\gamma$, $\delta$ is a particular solution to $\alpha \delta - \beta\gamma = 1$.  Furthermore,
\begin{enumerate}
        \item The curvatures of the circles in one family form an equivalence class modulo $\sqrt{-\Delta}N(\beta)$.  
        \item The centres of the circles in a given family lie on a single line through $\alpha/\beta$.  
        \item The family given by unit $u$ contains the same circles as the family given by $-u$, but with opposite orientations.
\end{enumerate}
\end{proposition}

\section{$K$-Apollonian packings and immediate tangency}
\label{sec:immtang}

In this section we give two equivalent definitions of a $K$-Apollonian packing.  This allows us to prove Theorem \ref{thm:disjointunion} of the introduction, describing the Schmidt arrangement as a union of its $K$-Apollonian packings.  

\begin{definition}
        The \emph{tangency graph} of $\mathcal{P} \subset \SK$ is the graph whose vertices are the circles of $\mathcal{P}$ and whose edges indicate tangencies.  We say that $\mathcal{P}$ is \emph{tangency-connected} if its tangency graph is connected.  We say that $\mathcal{P}' \subset \OSK$ is \emph{tangency-connected} if the underlying set of unoriented circles is so.
        We say $\mathcal{P}$ \emph{straddles} a circle $C$ if it intersects both the interior and exterior of $C$ nontrivially.
\end{definition}

This definition of straddling coincides with that of \cite[Definition 4.5]{VisOne} in the cases we are considering ($K \neq \QQ(\sqrt{-3})$, where circles intersect only tangently), and is simpler to state.

We repeat Definition \ref{def:apppack} of the introduction here:

\begin{definition}
        We define a \emph{$K$-Apollonian packing of unoriented circles} to be any maximal tangency-connected subset $\mathcal{P}$ of circles of $\mathcal{S}_K$ under the condition that $\mathcal{P}$ does not straddle any circle of $\Scal_K$.
\end{definition}

We will now extend the definition of a $K$-Apollonian packing to oriented circles.

\begin{definition}\label{def:apppack2} We define a \emph{$K$-Apollonian packing of oriented circles} to be any maximal tangency-connected subset $\mathcal{P}$ of oriented circles of $\OSK$ with disjoint interiors under the condition that $\mathcal{P}$ does not straddle any circle of $\OSK$.  
       \end{definition}

       The following lemma verifies that these two definitions correspond nicely under the orientation-forgetting map.

       \begin{lemma}
               A collection of circles $\mathcal{P} \subset \SK$ is a $K$-Apollonian packing of unoriented circles if and only if it can be obtained from a $K$-Apollonian packing of oriented circles by forgetting orientation.
\end{lemma}

\begin{proof}
        First, we show that any tangency-connected, non-straddling $\mathcal{P} \subset \SK$ lifts (under the orientation-forgetting map) to a unique $\mathcal{P}' \subset \OSK$ which is non-straddling and tangency-connected with disjoint interiors.  For, suppose $\mathcal{P} \subset \SK$ does not straddle any circle of $\SK$.  Then, in particular, it cannot intersect both the interior and exterior of any $C \in \mathcal{P}$.  Therefore, $\mathcal{P}$ can be lifted to $\mathcal{P}' \subset \OSK$ by choosing the orientation of each circle $C \in \mathcal{P}$ in such a way that $\mathcal{P}$ is disjoint from the interior of $C$.  This lift is unique if $\mathcal{P}$ contains at least two circles (an assumption we are safe in making; for example, by Proposition \ref{prop:tangentfamilies}, $K$-Apollonian packings contain more than one circle).  The resulting collection has disjoint interiors and the non-straddling property.  Furthermore, $\mathcal{P}'$ is tangency-connected if and only if $\mathcal{P}$ is.

        Conversely, any $\mathcal{P}' \subset \OSK$ which is non-straddling and tangency-connected with disjoint interiors is surely non-straddling and tangency-connected when orientation is forgotten, and this operation is inverse to the lifting just described.

        What remains then, is to show that such pairs $(\mathcal{P}, \mathcal{P}')$ have the property that $\mathcal{P}$ is maximal with respect to being tangency-connected and non-straddling if and only if $\mathcal{P}'$ is maximal with respect to being tangency-connected, non-straddling and having disjoint interiors.

        A $K$-Apollonian packing in $\OSK$ cannot contain both orientations of the same unoriented circle without violating the disjoint interiors requirement (unless it contained only those two circles, but then it would not be maximal).
        It follows then, that if there is a way to add another circle to $\mathcal{P}'$ in $\OSK$ under the stipulated conditions it corresponds to a way to add another circle to $\mathcal{P}$ in $\SK$, and vice versa.
\end{proof}

With this result in place, we may henceforth consider oriented circles.  We wish to give a different characterisation of being a $K$-Apollonian packing of oriented circles, and show that it is equivalent.  We will need the following notion.

\begin{definition}Two oriented $K$-Bianchi circles $C_1, C_2 \in \OSK$ are \emph{immediately tangent} if they are externally tangent in such a way that the pair straddles no circles of $\OSK$.
\end{definition}

In other words, $C_1$ and $C_2$ are \emph{immediately tangent} if they are consecutive in an unoriented family of Proposition \ref{prop:tangentfamilies}, and taken with disjoint interiors.

We recall the result that for any $K$-Bianchi circle, there is exactly one $K$-Bianchi circle immediately tangent to it at a fixed $K$-rational point.  More precisely:

\begin{proposition}[{\cite[Proposition 6.2]{VisOne}}]
        \label{prop:immtang}
        Let $C \in \OSK$ be an oriented $K$-Bianchi circle with $K$-rational point $x$.  Then there exists 
\[
        M_C = \begin{pmatrix} \alpha & \gamma \\ \beta & \delta \end{pmatrix} \in \PGL_2(\OK)
\]
such that $C = M_C(\RR)$ and $x = \alpha/\beta$.  Furthermore, there exists exactly one oriented $K$-Bianchi circle $C' \in \OSK$ immediately tangent to $C$ at $x$, and it is given by $C' = M_{C'}(\RR)$ where
\[
        M_{C'} = 
         \begin{pmatrix} \alpha & -\gamma + \tau \alpha \\ \beta & -\delta + \tau\beta \end{pmatrix}
        = M_C \begin{pmatrix} 1 & \tau \\ 0 & -1 \end{pmatrix} \in \PGL_2(\OK).
\]
\end{proposition}

\begin{definition}
        \label{def:immtang}
        Let $\Pcal$ be a subset of $\OSK$.  We say that $\Pcal$ is \emph{closed under immediate tangency} if it has the property that
for any circle $C \in \Pcal$, any circle $C' \in \OSK$ that is immediately tangent to $C$ is also contained in $\Pcal$.  
We say that $\Pcal$ is a \emph{$K$-Apollonian packing defined by tangency} if it is minimal among non-empty subsets of $\OSK$ which are closed under immediate tangency.
\end{definition}

There are infinitely many circles immediately tangent to any one circle, by Proposition \ref{prop:immtang}, so $K$-Apollonian packings defined by tangency are infinite sets of circles.

This definition can be rephrased slightly, in the language of graph theory.

\begin{definition}
        \label{def:imm-tang-graph}
        The \emph{immediate tangency graph} of $\OSK$ is the graph whose vertices are the circles of $\OSK$ and whose edges are given by the symmetric relation of immediate tangency.  
\end{definition}

With this definition, the following proposition is immediate.

\begin{proposition}
        \label{prop:component}
        A subset $\mathcal{P}$ of $\OSK$ is a $K$-Apollonian packing defined by tangency if and only if it is a connected component of the immediate tangency graph.
\end{proposition}

We now show that Definitions \ref{def:apppack2} and \ref{def:immtang} for a $K$-Apollonian packing coincide.

\begin{proposition}
        \label{prop:twodefs}
        A subset $\mathcal{P}$ of $\OSK$ is a $K$-Apollonian packing if and only if it is a $K$-Apollonian packing defined by tangency.
\end{proposition}

\begin{proof}
        Let $\mathcal{P}$ be a $K$-Apollonian packing.  Let $C$ be a circle of $\mathcal{P}$ and let $z \in C$ be a $K$-rational point.  Let $C'$ be the unique circle immediately tangent to $C$ at $z$.  Our first goal is to show that $\mathcal{P}$ contains $C'$.   This demonstrates that $\mathcal{P}$ is closed under immediate tangency.
        
        If not, since $C$ is in the exterior of $C'$, then all of $\mathcal{P}$ is exterior to $C'$.  Suppose we add $C'$ to $\mathcal{P}$.  This creates a larger tangency-connected set with disjoint interiors.  We show that this larger set is not straddling any circle of $\OSK$, which is a contradiction to maximality.  To see this, suppose it did straddle a circle.  Since $\mathcal{P}$ did not, it must be that $C'$ is inside some circle $D$ that $\mathcal{P}$ is exterior to.  Since they touch at $z$, $D$ is tangent to $C$ and $C'$ at $z$ and separates them (one is interior to $D$, the other exterior).  In other words, $C$ and $C'$ straddle $D$, but this is a contradiction to immediate tangency.  

        Hence $C' \in \mathcal{P}$, and we have demonstrated that $\mathcal{P}$ is closed under immediate tangency.  
        
        To show that $\mathcal{P}$ is minimal as a set closed under immediate tangency, let us imagine removing a collection of circles $\mathcal{S}$ from $\mathcal{P}$, to leave behind a subset $\mathcal{S}' \subset \mathcal{P}$ which is closed under immediate tangency.  Then it is impossible that any circle of $\mathcal{S}$ touches any circle of $\mathcal{S}'$ since any two circles of $\mathcal{P}$ that are touching are immediately tangent (otherwise they would violate the non-straddling property).  Hence $\mathcal{P}$ is tangency-disconnected, a contradiction.  Therefore $\mathcal{P}$ is a $K$-Apollonian packing defined by tangency.

        Conversely, suppose $\mathcal{P}$ is a $K$-Apollonian packing defined by tangency.  First, we show that $\mathcal{P}$ straddles no circle $D \in \OSK$.
       By definition, $D$ is not tangent immediately to any of the circles in its interior.  No circle of $\mathcal{P}$ exterior to $D$ is immediately tangent to a circle interior to $D$, as such a tangency would straddle $D$.  Hence $\mathcal{P}$ is not minimal; its circles inside $D$ could be removed without violating closure under immediate tangency.  This is a contradiction. 
       Therefore we have shown that $\mathcal{P}$ straddles no circles.
       
       We have already observed (Proposition \ref{prop:component}) that $\mathcal{P}$ is tangency-connected.  Next we show that $\mathcal{P}$ has disjoint interiors.  Suppose not:  then some circle $C$ lies inside another, $C'$.  Since $\mathcal{P}$ is tangency-connected, we can assume without loss of generality that $C$ is tangent to $C'$ at some point $z$.  But since $\mathcal{P}$ also includes the circle immediately tangent to $C'$ at $z$, this gives a collection of three distinct circles tangent at a point.  This is disallowed by the non-straddling property just shown.

      Finally, we must show $\mathcal{P}$ is maximal.  Suppose there were another circle $C \notin \mathcal{P}$ such that $\Pcal \cup \{C\}$ is tangency-connected.  We will show $\Pcal \cup \{C\}$ straddles some circle of $\Scal_K$.  Let $z \in C$ be a point of tangency with a circle $D \in \Pcal$.  This is not an immediate tangency, since $C \notin \mathcal{P}$.  Then there is some $C' \neq D$, $C' \in \Pcal$, immediately tangent to $C$.  Thus we have three distinct circles of $\Pcal \cup \{C\}$ sharing a single tangency point:  as a set they must straddle one of their members.
\end{proof}

The following statement is an immediate consequence of the two equivalent definitions.

\begin{theorem}
        \label{thm:disjointunion-oriented}
        The $K$-Apollonian packings of $\OSK$ form a collection of disjoint subsets of $\OSK$ whose union is all of $\OSK$.  
\end{theorem}

\begin{proof}
        This is evident from Definition \ref{def:immtang}, as the $K$-Apollonian packings are exactly the connected components of the immediate tangency graph on $\OSK$.
\end{proof}

\begin{proof}[Proof of Theorem \ref{thm:disjointunion}]
        Forget orientations in Theorem \ref{thm:disjointunion-oriented}.
\end{proof}

Note that the packings are not disjoint as collections of \emph{unoriented} circles, however.  A circle generally belongs to two packings, depending on the orientation assigned:  one living in its interior and another in its exterior.

It is an open question whether tangency-connectedness corresponds exactly to the topological notions of connectedness or path-connectedness for subsets of Schmidt arrangements (it is clearly stronger for arbitrary unions of circles in $\widehat{\CC}$).  See \cite[Section 7]{VisOne} for more on this distinction.  In particular, it is shown in \cite[Theorem 7.1]{VisOne} that $\SK$ is connected if and only if it is tangency-connected if and only if $\OK$ is Euclidean.

\begin{definition}
        \label{defn:fund}
        The $K$-Apollonian packing containing $\widehat{\RR}$ is called the \emph{fundamental packing}, and is denoted $\Pcal_K$.
\end{definition}

See Figure \ref{fig:app-in-sk} for the fundamental packing of $\QQ(i)$.  Other fundamental packings are shown in Figures \ref{fig:2strip}, \ref{fig:7strip} and \ref{fig:11strip}.

\section{Loops in the tangency graph}
\label{sec:loops}

The purpose of this section is to prove the following theorem, from which Theorem \ref{thm:forest} of the introduction follows.  Note that \cite[Theorem 7.5]{VisOne} already guarantees that $\SK$ has infinitely many components when $\Delta \le -15$.

\begin{theorem}
        \label{thm:16noloops}
        Let $K$ be such that $\Delta \le -15$.  Then the immediate tangency graph of $\OSK$ contains no loops.
\end{theorem}

The proof for $\Delta < -15$ will rely on a simple graph theory principle.

\begin{lemma}
        \label{lemma:graphtheory}
        Consider a graph $G$ with vertex set $V$ and a function $f: V \rightarrow \RR$.  Direct each edge of $G$ according to the direction of increase of $f$ (whenever it is not constant).  Suppose that at any vertex $v \in V$, all but at most one edge is directed outward.  Then $G$ contains no loops.  
\end{lemma}

\begin{proof}
 Suppose there is a loop.  Let $v$ be a vertex of that loop.  By assumption, at least one of the two edges of the loop adjacent to $v$ is directed outward.  Moving to the next vertex $w$ along this edge, we enter $w$ along an inward directed edge, and therefore, continuing along the loop, leave $w$ along an outward directed edge.  Hence, as we travel around the loop in this direction, the value of $f$ is increasing, so we can never return to $v$.
\end{proof}

\begin{lemma}
        \label{lemma:onebigger}
        Let $K$ be such that $\Delta < -15$.  Let $C$ be a $K$-Bianchi circle.  Then there is at most one point $x \in C$ at which $C$ is tangent to a $K$-Bianchi circle of absolute curvature less than or equal to the absolute curvature of $C$.
\end{lemma}

By \emph{absolute curvature} we mean the absolute value of the curvature.  Note that, at the point $x$, the circle may be tangent to several circles of absolute curvature less than its own.

\begin{proof}
        Suppose $C = M(\widehat{\RR})$ for
        \[
                M = \begin{pmatrix} \alpha & \gamma \\ \beta & \delta \end{pmatrix},
        \]
        and let $\Lambda = \beta \ZZ + \delta \ZZ$.  Let $b$ be the reduced curvature of $C$.  Then the covolume of $\Lambda$ as a lattice in $\CC$ is $| b\sqrt{-\Delta}/2 |$.  There is at most one element of $\Lambda$, up to multiples, whose norm is less than its covolume, so for all but one $x \in \Lambda$, we have
        \begin{equation}
                \label{eqn:norm-covolume}
                N(x) \ge | b\sqrt{-\Delta}/2 |.
        \end{equation}
        On the other hand, if $C'$ is tangent to $C$ at the unique point with denominator $x$ (here we use that $b \neq 0$), and has curvature $b'$ where $|b'| \le |b|$, then it follows from Proposition \ref{prop:tangentfamilies} that
        \[
                |kN(x) - b| = |b'| \le |b|
        \]
        for some $k \in \ZZ^{>0}$, 
        whence
        \begin{equation}
                \label{eqn:norm-curvature}
                N(x) \le 2|b|.
        \end{equation}
        Comparing the inequalities \eqref{eqn:norm-covolume} and \eqref{eqn:norm-curvature} gives the desired result. 
\end{proof}

The proof of Theorem \ref{thm:16noloops} will rely on this lemma for $\Delta < -15$.  However, the lemma fails for $\Delta = -15$, where we will require an entirely different method.

\begin{proposition}
        \label{prop:15noloops}
        The tangency graph of $K = \QQ(\sqrt{-15})$ contains no loops.
\end{proposition}

\begin{proof}
        Let $K = \QQ(\sqrt{-15})$.  First, we remark that $\PSL_2(\OK)$ is transitive on pairs $(C,z)$ where $C \in \OSK$ and $z \in C$ is a tangency point between circles of $\OSK$.  Thus, we need only show that $\widehat{\RR}$ (with interior below) and any $C \in \SK$ of the form $ki+\widehat{\RR}$, $k \in \ZZ^{>0}$ (with interior above) do not participate as adjacent vertices in any cycle of the graph. Suppose for a contradiction, that they do.   We will call this the \emph{postulated cycle}.  
        
        Note that $k \ge \frac{\sqrt{15}}{2}$ (equality occurs if $C$ is the circle immediately tangent to $\widehat{\RR}$ (oriented negatively) at $\infty$).  Then the finitely many other circles making up the cycle must form a chain of tangent circles reaching from some tangency point $x$ of $\widehat{\RR}$ up to some tangency point $y$ of $C$.  In particular, the chain consists of finitely many circles.  Each is of curvature at least $\sqrt{15}$ or else equal to $0$.  Therefore, the portion of the chain that lies below $\widehat{\RR}+\frac{\sqrt{-15}}{2}$ is bounded away from $\infty$.  

In \cite[Section 7]{VisOne}, it was shown that $\SK$ is disconnected by demonstrating the existence of a \emph{ghost circle} which is contained in the complement of $\SK$.  The ghost circle $G$ is the circle of radius $1/\sqrt{15}$ centred at 
        \[
                \frac12 - \frac{7\sqrt{-15}}{30}.
        \]
        Its existence immediately implies the existence of infinitely many other ghost circles formed by $\PSL_2(\OK)$ images of $G$.  The union of these circles is contained in the complement of $\SK$.  These images include an infinite `ghost chain' of tangent circles extending horizontally to $\infty$ in both directions, and separating $\widehat{\RR}$ and $\widehat{\RR}+\frac{\sqrt{-15}}{2}$.  To see this, illustrated in Figure \ref{fig:ghostchain}, let $G'$ be the reflection of $G$ in $\widehat{\RR}$, and let $G''$ be the translation of $G$ by $\frac{-1+\sqrt{-15}}{2}$.  The circles $G'$ and $G''$ are tangent at $\frac{1+\sqrt{-15}}{4}$.  The union of translates of the pair $G'$ and $G''$ by $\ZZ$ is the `ghost chain' of tangent circles which contradicts the existence of the postulated cycle, and the theorem is proved.
\end{proof}

\begin{figure}
        \includegraphics[height=4in]{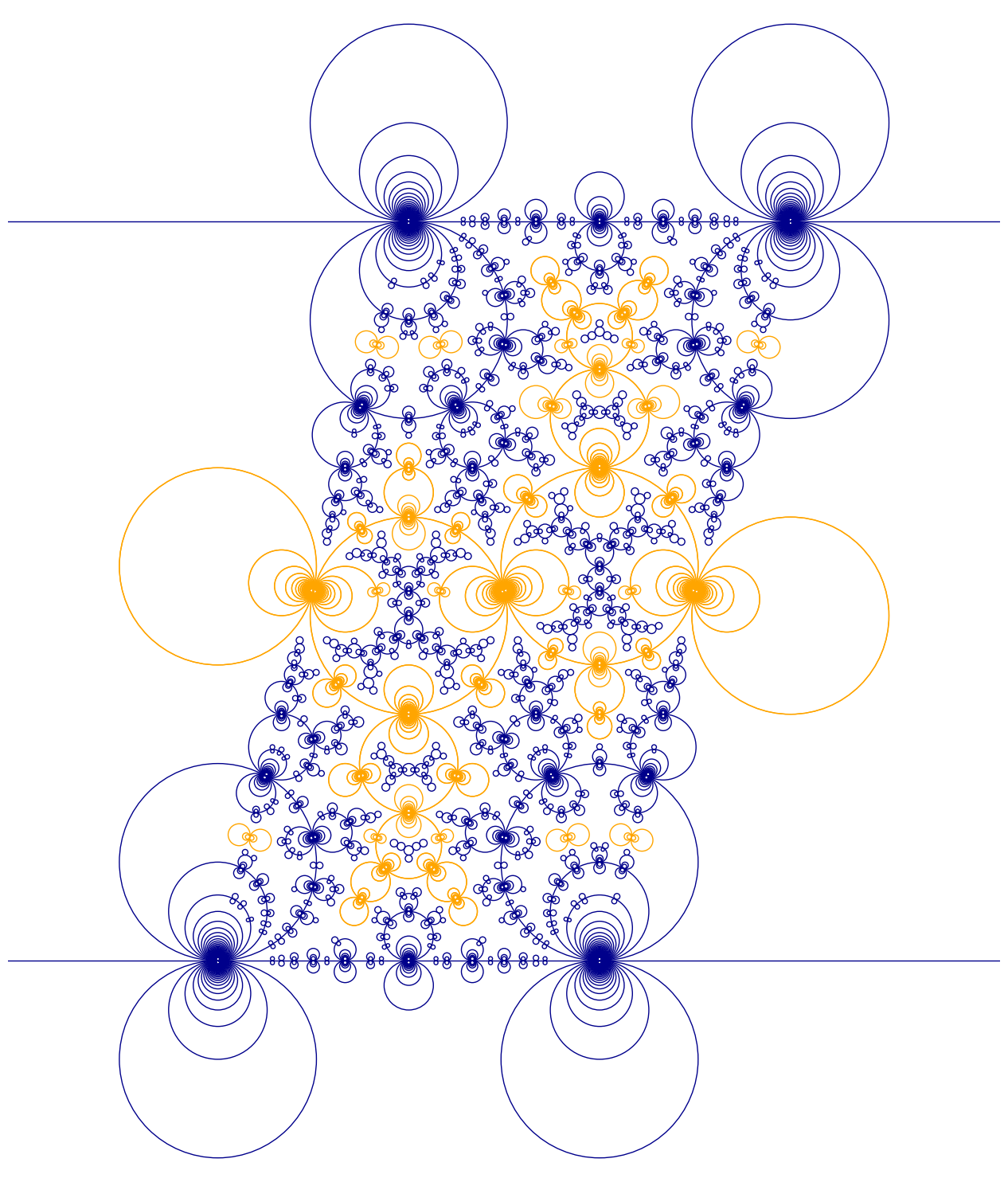} 
        \caption{The Schmidt arrangement of $\QQ(\sqrt{-15})$, in blue, with ghost circles, shown in yellow.}
\label{fig:ghostchain}
\end{figure}

\begin{proof}[Proof of Theorem \ref{thm:16noloops}]
        Suppose first that $\Delta < -15$.  Let $G$ be the immediate tangency graph of $\OSK$.  Let $f$ be the function of absolute curvatures on the vertices of the graph.  The edges of $G$ may be labelled by the points of tangency.  If one does so, then at any vertex, there is at most one adjacent edge with any given label, by Proposition \ref{prop:immtang}.  Combining this observation with Lemma \ref{lemma:onebigger} verifies the hypotheses of Lemma \ref{lemma:graphtheory}, and we are done.  
        
        The case of $\Delta = -15$ is completed by Proposition \ref{prop:15noloops}.
\end{proof}

In the Euclidean cases, we will see that the immediate tangency graph does contain loops.

\section{The space of circles}
\label{sec:space}

We introduce an embedding of the space of circles in Minkowski space which is described in \cite{\gtone} and elaborated upon beautifully by Kocik \cite{\Kocik}; it is a natural viewpoint from the perspective of the spin homomorphism and hyperbolic space as the space of Hermitian forms; see, for example, \cite{\EGM}.  Associating a circle to a Hermitian form goes back to Bianchi \cite[Chapter XV, p. 272]{MR0245501}.  See also \cite{MR2361159}.

Let $\MM$ be the vector space $\RR^4$ endowed with an inner product of signature $3,1$ given by the Gram matrix
\[
        G_M = \begin{pmatrix}
                0 & -\frac{1}{2} & 0 & 0 \\
     -\frac{1}{2} & 0 & 0 & 0 \\
                0 & 0 & 1 & 0 \\
                0 & 0 & 0 & 1
        \end{pmatrix}.
\]
The associated quadratic form will be denoted $M$, for Minkowski.  This space is identified with the collection of Hermitian matrices via
\[
\begin{pmatrix} b' \\ b \\ r \\ m \end{pmatrix} \leftrightarrow \begin{pmatrix} b' & r+mi \\
                                          r - mi & b \end{pmatrix}.
\]
Under this identification, the self-product of a vector under $G_M$ is equal to the determinant of the corresponding matrix.  In what follows, we will use this identification implicitly.

\newcommand{\OMp}{\operatorname{O}_M^+(\RR)}
\newcommand{\SOMp}{\operatorname{SO}_M^+(\RR)}
\newcommand{\OM}{\operatorname{O}_M(\RR)}

The space $\MM$ (considered as Hermitian matrices $T$) comes with a metric-preserving action of $\PGL_2(\CC)$ via
\[
        \gamma \cdot T := \frac{\gamma T \gamma^\dagger}{\Norm(\det \gamma)}, \quad \gamma \in \PGL_2(\CC),
\]
where $\dagger$ indicates the conjugate transpose.  There is also an action by $\Mob$, where complex conjugation $\mathfrak{c}$ acts by $\mathfrak{c} \cdot T = \overline{T}$.  This allows us to define a map 
        \[
                \rho: \Mob \rightarrow \OMp
        \]
        where $\OMp$ is the collection of time-preserving isometries of $\MM$ (in our description, this is equivalent to preserving the sign of $b+b'$), also known as the orthochronous Lorentz group.  This map is an isomorphism, which we call the \emph{exceptional isomorphism}, closely related to the famous {spin homomorphism}.  It can also be extended to $\GM$, so that we have the following isomorphisms and subgroup relations:
\[
        \xymatrix{
                \Mobplus \ar@{^{(}->}[r] \ar_\rho^{\cong}[d] &  \Mob  \ar@{^{(}->}[r] \ar_\rho^{\cong}[d] & \GM \ar_\rho^{\cong}[d] \\
                SO_M^+(\RR) \ar@{^{(}->}[r] & O_M^+(\RR)  \ar@{^{(}->}[r]  & O_M(\RR) \\
        }
\]
For more on this, see \cite{\gtone}.

Two Hermitian forms $H_i(u,v)$, $i=1,2$, on a $\CC$-vector space $V$ are called \emph{equivalent} or \emph{isometric} over $\CC$ if $H_1( \phi(u), \phi(v) ) = H_2(u,v)$ for some $\CC$-isomorphism $\phi: V \rightarrow V$.  The equivalence class of a Hermitian form and its determinant form a complete set of invariants for the orbits of the $\PGL_2(\CC)$-action on Hermitian forms described above.  All Hermitian forms of positive determinant are equivalent over $\CC$ (see, for example, \cite{\Lewis} and \cite[Chapter 9]{\EGM}).  Therefore the action of $\PGL_2(\CC)$ on the locus $M=1$ is transitive.

\newcommand{\OCirc}{\widehat{\Circ}}

Let $\Circ$ be the collection of circles in $\widehat{\CC}$.  This is a set with a $\Mob$ action (where $\mathfrak{c}$ acts by conjugation on $\widehat{\CC}$).  Denote by $\OCirc$ the collection of oriented circles, which is also a set with an action by $\Mob$.

We define a Pedoe map
\[
        \pi: \OCirc \rightarrow \MM
\]
which will endow the space of circles with the structure of a hypersufrace in an inner product space (see \cite{\Kocik}).  The map is defined as follows:
\[
\pi(C) = \begin{pmatrix} b' & r + mi \\ r-mi & b \end{pmatrix} \leftrightarrow \begin{pmatrix} b' \\ b \\ r \\ m \end{pmatrix},
\]
where $b$ is the curvature, $b'$ the co-curvature, and $r+mi$ the curvature-centre of $C$.  The image of $\pi$ is a Hermitian matrix of determinant $1$ (by Proposition \ref{prop:gencirc}), hence the image lies on the hyperboloid $M=1$ in $\MM$. 

\begin{proposition}
        The map $\pi: \OCirc \longrightarrow \{ v \in \MM : M(v) = 1 \}$ is a $\Mob$-equivariant bijection via $\rho$.  
\end{proposition}

\begin{proof}
        Bijectivity follows from Proposition \ref{prop:gencirc} (note that $M(v)=1$ is exactly \eqref{eqn:cocurv}).
        That $\pi$ respects the $\Mob$-action is a direct computation.  
\end{proof}

In particular, the oriented $K$-Bianchi circles are in bijection with the orbit of $(0,0,0,1)^t \in \MM$ under $\rho(\Mob(\OK))$.

One may verify that $\pi$ may also be computed in the following way.  Express a circle $C$ as $C = M_C(\RR)$ for 
\[
        M_C = \begin{pmatrix} \alpha & \gamma \\ \beta & \delta \end{pmatrix} \in \PGL_2(\CC), \quad \Norm(\det(M_C)) = 1.
\]
Then
\[
        \pi(C) = 
        i(N - N^\dagger), \text{ where }
        N = \begin{pmatrix}
                \alpha \overline{\gamma} & \alpha\overline{\delta} \\
                \beta \overline{\gamma} & \beta\overline{\delta} \\
        \end{pmatrix}.
\]
It is a brief computation that this map agrees with $\pi$.
Note that, if we write $-C$ for the circle $C$ with opposite orientation, then $\pi(-C) = -\pi(C)$.

The inner product on Minkowski space now gives us an inner product on circles.  It carries geometric information about the circles in question.  Following Kocik, we will call this the \emph{Pedoe product}.

\begin{figure}
        \includegraphics[height=1in]{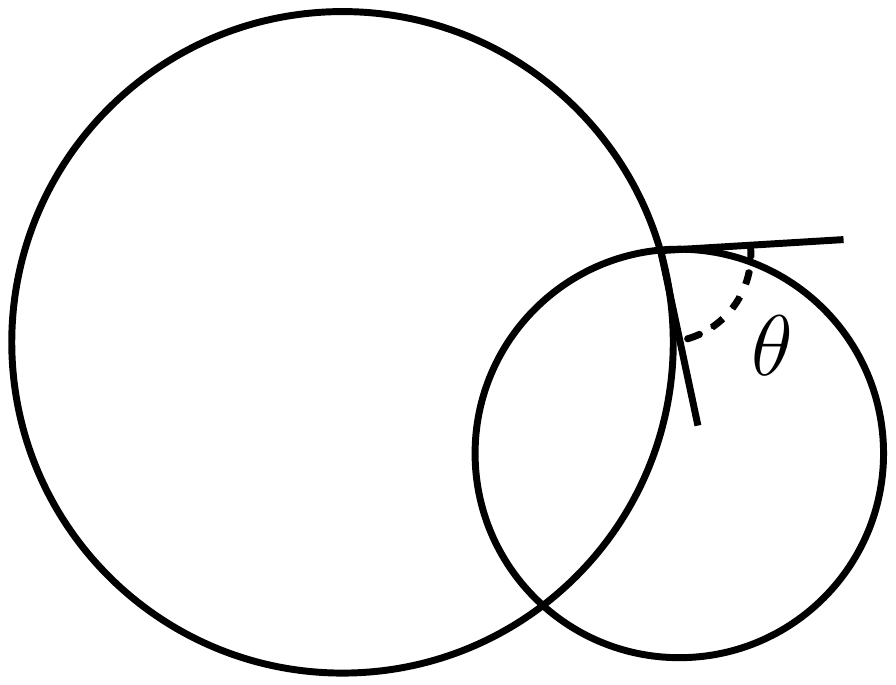} 
        \caption{The angle of intersection of two circles.}
\label{fig:anglecircles}
\end{figure}

\begin{proposition}[Proposition 2.4 of \cite{\Kocik}]
        \label{prop:pedoeproduct}
        Let $v_i = \pi(C_i)$ for two circles $C_1, C_2$ which are not disjoint.  Then $\langle v_1, v_2 \rangle = \cos \theta$, where $\theta$ is the angle between the two circles as in Figure \ref{fig:anglecircles}.  In particular,
                \begin{enumerate}
                        \item $\langle v_1, v_2 \rangle = -1$ if and only if the circles are tangent externally \includegraphics[width=0.5in]{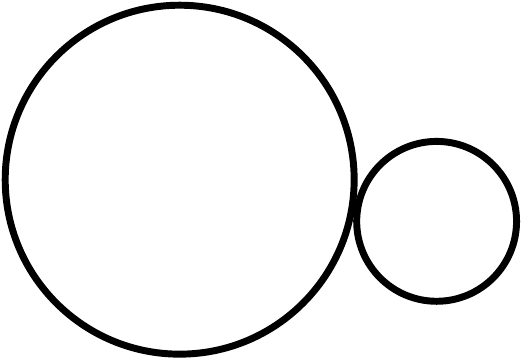}
                        \item $\langle v_1, v_2 \rangle = 1$ if and only if the circles are tangent internally \includegraphics[width=0.38in]{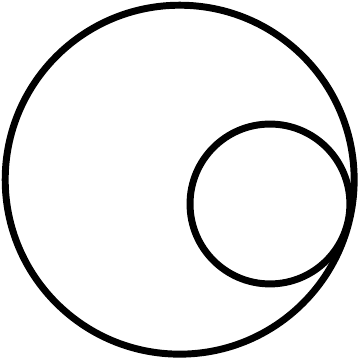}
                        \item $\langle v_1, v_2 \rangle = 0$ if and only if the circles are mutually orthogonal \includegraphics[width=0.47in]{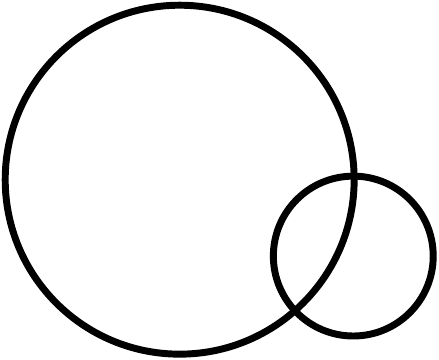}
                \end{enumerate}

This extends via the Law of Cosines, so that the inner product of disjoint circles is
\begin{equation}
        \label{eqn:lawcosines}
        \langle v_1, v_2 \rangle = \frac12\left( {d^2b_1b_2} - b_2/b_1 - b_1/b_2 \right),
\end{equation}
where $b_1$ and $b_2$ are the curvatures and $d$ is the distance between centres.
\end{proposition}

\section{The Gaussian integers and Descartes configurations}
\label{sec:gaussian}

The study of Apollonian circle packings depends on the study of the action of the so-called Apollonian group on the space of Descartes configurations.  We will develop a similar theory for $K$-Apollonian packings in other imaginary quadratic fields.  Therefore, we will review briefly the case of Apollonian circle packings for comparison.  See \cite{\gtone} for details.

We will consider Apollonian circle packings formed from circles in $\widehat{\CC}$ (any Apollonian circle packing may be realized this way).  Let $W_D$ be a matrix whose columns are $\pi(C_i)$ for any four oriented circles $D: C_1, C_2, C_3, C_4$.  By Proposition \ref{prop:pedoeproduct}, these four circles are in Descartes configuration, or this is the case once all orientations are reversed, if and only if 
\begin{equation}
        \label{eqn:descmat}
        W_D^t G_M W_D = \begin{pmatrix}
                1 & -1 & -1 & -1 \\
                -1 & 1 & -1 & -1 \\
                -1 & -1 & 1 & -1 \\
                -1 & -1 & -1 & 1 \\
        \end{pmatrix} =: R.
\end{equation}
Let $\Dcal$ denote the space of matrices satisfying this equation, known as the space of Descartes configurations\footnote{Note that these are ordered, oriented configurations in the sense of \cite{\gttwo}.}.  
Since $G_M$, $R$, and the elements of $\Dcal$ are all invertible, we obtain
\[
 W_D R^{-1} W_D^t = G_M^{-1}.
\]
This collection of 16 quadratic equations in the curvatures, co-curvatures and curvature-centres of four circles are the full `Descartes relations' as in \cite{MR1903421}: the circles are in Descartes configuration if and only if these 16 equalities hold.  In particular, the upper left corner of this matrix equality is the relation on curvatures \eqref{eqn:descquad}.

Given four circles in Descartes configuration (write $\bfv_1, \bfv_2, \bfv_3, \bfv_4 \in \MM$), and a chosen subset of three of them, say $\bfv_1, \bfv_2, \bfv_3$, the fourth circle $\bfv_4$ can be swapped out for its alternative $\bfv_4'$, i.e. the unique other circle which forms a Descartes configuration with $\bfv_1, \bfv_2, \bfv_3$.  This takes one configuration in an Apollonian circle packing $\Pcal$ to another in the same packing.  The relation that describes this swap in $\MM$ is very simple:
\[
        \bfv_i + \bfv_i' = 2\left( \sum_{j \neq i} \bfv_j \right).
\]
In particular, if the original curvatures are $a,b,c,d$, then the new curvature is
\[
        d' = 2(a+b+c) - d.
\]
There are four ways to swap out a circle.  This is accomplished by right multiplication on $\Dcal$ by the following four matrices of $O_R(\RR)$ of order two:
\begin{equation}
        \label{eqn:iappgens}
        \begin{pmatrix}
                -1 & 0 & 0 & 0 \\
                2 & 1 & 0 & 0 \\
                2 & 0 & 1 & 0 \\
                2 & 0 & 0 & 1 \\
        \end{pmatrix},
        \begin{pmatrix}
                1 & 2 & 0 & 0 \\
                0 & -1 & 0 & 0 \\
                0 & 2 & 1 & 0 \\
                0 & 2 & 0 & 1 \\
        \end{pmatrix},
        \begin{pmatrix}
                1 & 0 & 2 & 0 \\
                0 & 1 & 2 & 0 \\
                0 & 0 & -1 & 0 \\
                0 & 0 & 2 & 1 \\
        \end{pmatrix},
        \begin{pmatrix}
                1 & 0 & 0 & 2 \\
                0 & 1 & 0 & 2 \\
                0 & 0 & 1 & 2 \\
                0 & 0 & 0 & -1 \\
        \end{pmatrix}.
\end{equation}
Right multiplication by one of the matrices \eqref{eqn:iappgens} corresponds to inversion in a circle orthogonal to three of the four circles of the quadruple, as in Figure \ref{fig:appswap}.  For example, the four swaps on the so-called \emph{base cluster} shown in Figure \ref{fig:baseqqi} are accomplished by the following M\"obius transformations (given in the same order as \eqref{eqn:iappgens}):

\begin{equation}
        \label{eqn:appmob}
        z \mapsto \frac{(2i+1)\overline{z}-2}{2\overline{z}+2i-1}, \quad
        z \mapsto -\overline{z}+2, \quad
        z \mapsto \frac{\overline{z}}{2\overline{z}-1}, \quad
        z \mapsto -\overline{z}.
\end{equation}

\begin{figure}
        \includegraphics[height=2.5in]{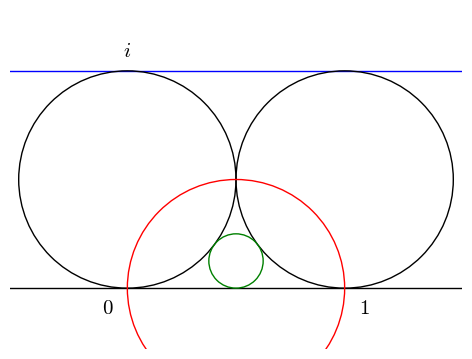}
        \caption[A swap for $\Scal_{\QQ(i)}$]{A swap for $\Scal_{\QQ(i)}$.  The blue and black circles (including two straight lines) form the base quadruple (see Figure \ref{fig:baseqqi}) of the $\QQ(i)$-Apollonian packing containing $\widehat{\RR}$.  If swapping by preserving the black circles, the blue circle is replaced with the green one, by inversion in the red circle.  The new quadruple is formed of the black and green circles.}
        \label{fig:appswap}
\end{figure}

\begin{figure}
        \includegraphics[height=3in]{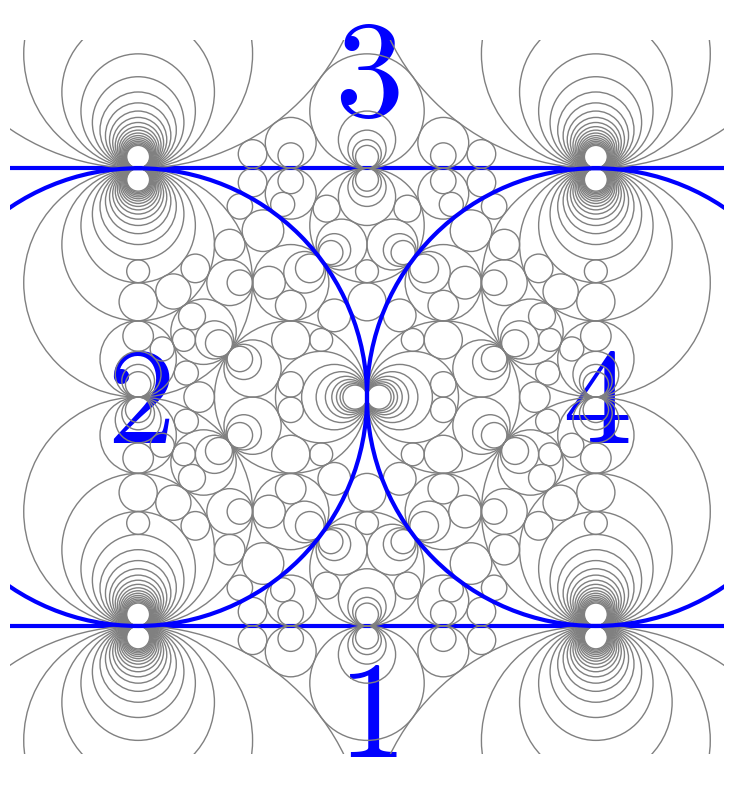} 
        \caption[Base quadruple for $\Scal_{\QQ(i)}$]{Base quadruple for $\Scal_{\QQ(i)}$.
                The coordinates of the four circles in $\MM$ are given, in the labelled order, by the columns of the following matrix:
{\tiny ${\begin{pmatrix}
0 & 0 & 2 & 2 \\ 0 & 2 & 0 & 2 \\ 0 & 0 & 0 & 2 \\ -1 & 1 & 1 & 1 \end{pmatrix}}$}.
                }
\label{fig:baseqqi}
\end{figure}

The group generated by the matrices \eqref{eqn:iappgens} is commonly referred to as the algebraic Apollonian group, which we will denote $\widehat{\Acal_{\QQ(i)}}$.  The group generated by the M\"obius transformations \eqref{eqn:appmob} is commonly referred to as the geometric Apollonian group, which we will denote $\Acal_{\QQ(i)}$.   While the algebraic Apollonian group does not act on individual circles, only ordered quadruples, the geometric Apollonian group is a group of M\"obius transformations acting on circles, unordered or ordered quadruples.  However, both groups have the property that the orbit of the base quadruple gives the full fundamental Apollonian packing.  Furthermore, the two are isomorphic as groups.  The following section will elaborate on the isomorphism.

\begin{theorem}[{Graham, Lagarias, Mallows, Wilks, Yan \cite[Proof of Theorem 4.3]{\gtone}}]
        \label{thm:appfree}
        There are no relations on the matrices of \eqref{eqn:iappgens} besides the fact that they are of order two.  Therefore, the group $\widehat{\Acal_{\QQ(i)}}$, and also $\Acal_{\QQ(i)}$, is a free product of the four copies of $\ZZ/2\ZZ$.
\end{theorem}

This group is the basic tool in arithmetic results concerning curvatures in Apollonian circle packings, because of the following.

\begin{theorem}[{Graham, Lagarias, Mallows, Wilks, Yan \cite[Theorem 4.3]{\gtone}}]
        \label{thm:gtoneorbits}
        Let $\Pcal$ be an Apollonian circle packing.
 \begin{enumerate}
         \item The full set of Descartes configurations contained in $\Pcal$ is a union of $48$ orbits of $\widehat{\Acal_{\QQ(i)}}$. 
         \item Fix a Descartes quadruple $D \in \Pcal$.  Then there are $48$ matrices $W_D \in \Dcal$ (i.e., satisfying \eqref{eqn:descmat}) representing this quadruple.  Each of the $48$ orbits contains exactly one of these $48$ matrices.
 \end{enumerate}
\end{theorem}

The $48$ matrices representing a quadruple are formed by reordering the circles $24$ ways, and reversing the orientation of all four simultaneously.

\section{Cluster spaces and the algebraic-geometric correspondence}
\label{sec:spaces}

\begin{definition}
        A \emph{cluster space} is a set of the form
\[
       S_R := \{ W \in M_{4 \times 4}(\RR) : W^t G_M W = R \},
\]
where $R$ is a fixed invertible matrix.
\end{definition}

The space $\Dcal$ of Descartes quadruples of the last section is such a space.  
When the columns of the $W$ lie on $M=1$ (i.e. the diagonal of $R$ consists of $1$'s), this can be considered the collection of quadruples of circles, considered in $\MM$, which are in a particular configuration (specified by $R$) with respect to the Pedoe product.  In general, we loosen this requirement, and may interpret columns as sums of circles.

The purpose of the next two results is to show that $S_R$ is a principal homogeneous space under left and right actions on $S_R$ by matrix groups isomorphic to $\OM$.  

\begin{definition}
        The left action of $\OM$ on $S_R$ by matrix multiplication on the left is called the \emph{geometric action}.
\end{definition}

\begin{proposition}
        \label{prop:left}
        The set $S_R$ is a principal homogeneous space for the geometric action.
\end{proposition}

\begin{proof}
Since $N \in \OM$ preserves the form $M$,
\[
        (N W)^t G_M N W
        = 
        W^t G_M W ,
\]
and so this action preserves $S_R$.  If $R$ is invertible, then $W \in S_R$ are invertible and so the element of $\OM$ taking any $W_1$ to $W_2$, namely $N:= W_2W_1^{-1} \in \OM$,
exists and is unique.  That is, $\OM$ is freely transitive on $S_R$.  
\end{proof}

The reason for the name \emph{geometric} is that, restricting to $\OMp$, the left action can be considered a M\"obius action on circles via the exceptional isomorphism $\rho$. 

The special case of $R=G_M$ gives $S_R = \OM$.

\begin{definition}
        Write $\operatorname{O}_R(\RR)$ for the matrices preserving the quadratic form associated to Gram matrix $R$.
        The right action of $\operatorname{O}_{R}(\RR)$ on $S_R$ by matrix multiplication on the right is called the \emph{algebraic action}.
\end{definition}

\begin{proposition}
        \label{prop:algtorsor}
        The set $S_R$ is a principal homogeneous space for the algebraic action.
        Furthermore, $\operatorname{O}_{R}(\RR) \cong \OM$.
\end{proposition}

\begin{proof}
        Let $W_0 \in S_R$.  Then $\operatorname{O}_{R}(\RR) = W_0^{-1} \OM W_0 \cong \OM$ preserves the quadratic form given by Gram matrix $R$, which form is isomorphic to $M$ over $\RR$.  Therefore right multiplication by this group preserves $S_R$.  The proof is as for the last proposition.
\end{proof}

The algebraic action is sometimes called the Apollonian action.  It acts by linear combination on 4-tuples of vectors in $\MM$.  It cannot be thought of as arising from an action on circles or vectors alone; it only acts on $S_R$.  

The isomorphism between $\operatorname{O}_R(\RR)$ and $\OM$ given in Proposition \ref{prop:algtorsor} is called the \emph{algebraic-geometric correspondence} depending on $W_0 \in S_R$, which we will denote
\[
        \sigma_{W_0} : \OM \rightarrow \operatorname{O}_R(\RR),
        \quad M \mapsto W_0^{-1}MW_0 .
\]

Both the geometric and algebraic actions have a manifestion in the action of M\"obius.  To discuss this, we have to interpret $S_R$ as \emph{clusters} of circles.

\begin{definition}
        Let $n \ge 4$, $n \in \ZZ$.  A \emph{cluster type} is a finite-to-one function 
        \[
                f: S_R \rightarrow {\Circ}^n,
        \]
        on a cluster space, with image lying in the collection of $n$-tuples of oriented circles, and having the form
        \[
                f(W) = \left( \pi^{-1}\left( \sum_{i=1}^4 a_{i,j}W_i
                \right) \right)_{j=1}^n,
        \]
        where $a_{i,j} \in \ZZ$ and $W_1, \ldots, W_4$ denote the columns of $W$.  In other words, a cluster type determines a collection of circles by linear combination on the columns of $W$.  A cluster type is given by the data of an invertible matrix $R$ and the $a_{i,j} \in \ZZ$, $1 \le i \le 4$, $1 \le j \le n$.  
        A collection of circles in the image of a given cluster type is called an \emph{ordered cluster} of the given type.  An \emph{unordered cluster} is any collection obtained from an ordered cluster by forgetting ordering.
\end{definition}

The previous definition is motivated by the notion of a \emph{Descartes quadruple}, which is the cluster type given by $n=4$,
\[
        R = \begin{pmatrix} 
                1 & -1 & -1 & -1 \\
                -1 & 1 & -1 & -1 \\
                -1 & -1 & 1 & -1 \\
                -1 & -1 & -1 & 1
        \end{pmatrix},
\]
and $a_{i,j} = \delta_{i,j}$.

It is immediate to verify that under these definitions, the M\"obius action corresponds to the geometric action, which justifies its name.  
The algebraic action can also be interpreted as an action of $\Mob$, via the algebraic-geometric correspondence.  This is complicated by the need for a \emph{base cluster} $W_0$ to define the algebraic-geometric correspondence, and by the need to consider ordered clusters as elements of $\Mob$.  The precise statement is as follows, and is verified by direct computation.

\begin{proposition}
        \label{prop:algebraic}
        Fixing an ordered base cluster $B$ given by $W_B \in S_R$, the left- and right-multiplication action of $\Mob$ on $\Mob$ corresponds to the algebraic and geometric actions of $\operatorname{O}_R(\RR)$ and $\OM$, respectively, on $S_R$, in the sense that the following diagram commutes (where all arrows between sets are bijections, all arrows between groups are isomorphisms, and tightly curved arrows represent group actions):

        \begin{center}\includegraphics[height=2in]{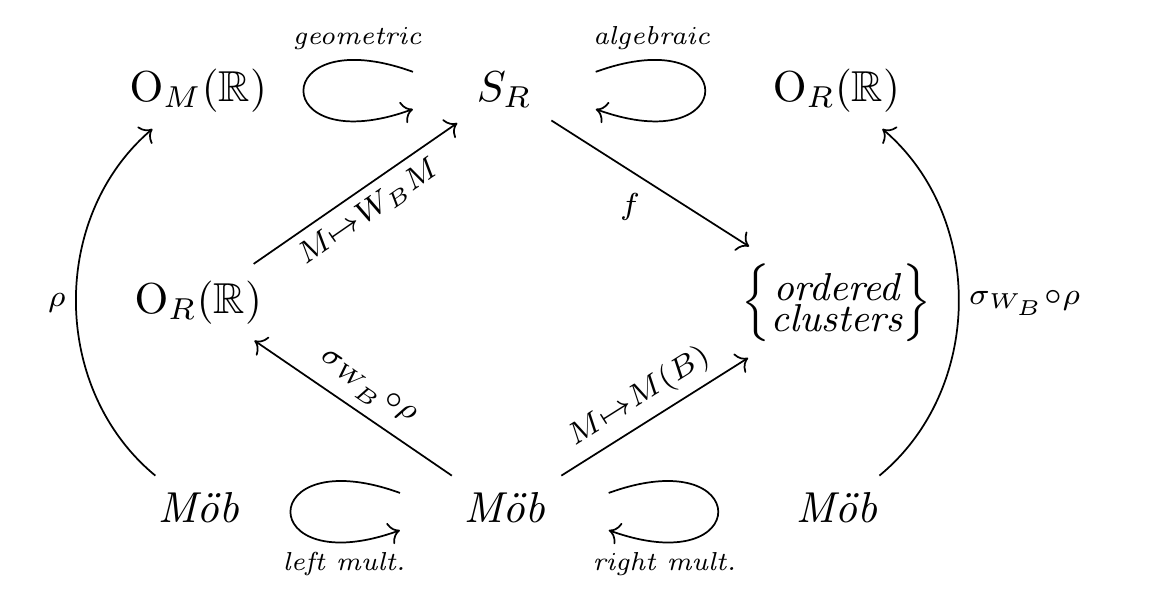}
        \end{center}
\end{proposition}

In other words, an element of $\Mob$ can be interpreted as its image on the base cluster, and then the right multiplication action of $\Mob$ on $\Mob$ `is' the algebraic action on clusters.  When the right multiplication of $\Mob$ on $\Mob$ is interpreted in this way, we will simply refer to it as the algebraic action.  A consequence of Proposition \ref{prop:algebraic} is the following.

\begin{proposition}
        \label{prop:leftrightorbit}
        The geometric and algebraic orbits of the base cluster under any subgroup $G < \Mob$ agree.
\end{proposition}

\section{Apollonian groups, weak and strong, algebraic and geometric}
\label{sec:app}

In this section we define the various flavours of Apollonian groups for a general imaginary quadratic field.  
Recall that $\Pcal_K$ denotes the fundamental packing (Definition \ref{defn:fund}), and let $\overline{\mathcal{P}_K}$ denote its closure in $\widehat{\CC}$.
The \emph{limit set} (or \emph{residual set}) $\Lambda(G)$ of a subgroup $G \subset \Mob$ is the accumulation set of the orbit of the origin.  

\begin{definition}A \emph{weak Apollonian group} for the imaginary quadratic field $K \neq \QQ(\sqrt{-3})$, or \emph{weak $K$-Apollonian group}, is a finitely generated Kleinian group $\Acal_K < \langle \PSL_2(\OK), \mathfrak{c} \rangle  < \Mob$ whose
                         limit set is the closure of the fundamental Apollonian packing for $K$, i.e. $\Lambda(\Acal_K) = \overline{\mathcal{P}_K}$.
\end{definition}

There is no reason to assume that there is a unique weak $K$-Apollonian group for a given field $K$.  It is easy to give examples of weak $K$-Apollonian groups.

\begin{theorem}
        \label{thm:weak}
Let $K$ be an imaginary quadratic field.  Let $\Acal$ be the subgroup of $\Mob$ generated by $\PSL_2(\ZZ)$ and the matrix $V =  \begin{pmatrix} 1 & \tau \\ 0 & -1 \end{pmatrix}$.
        Then $\Acal$ is a weak $K$-Apollonian group.
\end{theorem}

\begin{proof}
        As a subset of the Bianchi group, $\Acal$ is Kleinian.  Since $\PSL_2(\ZZ)$ is finitely generated, so is $\Acal$.
        We use Proposition \ref{prop:immtang}.  Given a circle $M(\widehat{\RR})$, the immediately tangent circles are exactly those given by $M'(\widehat{\RR})$ where  $M' \in M\PSL_2(\ZZ)V$.  Therefore the orbit of $\widehat{\RR}$ includes all of the fundamental packing.  Since we have $\PSL_2(\ZZ)<\Acal$, all of $\widehat{\QQ}$ is in the orbit of $0$, so that $\widehat{\RR}$ is in the limit set $\Lambda(\Acal)$, and so are all its images, i.e. all of $\Pcal_K$.  On the other hand, the orbit of $0$ lies within $\Pcal_K$.  Therefore $\Lambda(\Acal) = \overline{\mathcal{P}_K}$.
\end{proof}

        The proof illustrates that the fundamental packing is exactly the orbit of $\widehat{\RR}$ under $\Acal$.  All other $K$-Apollonian packings are orbits of left cosets of $\Acal$.

\begin{theorem}
        \label{thm:limit}
        Any $K$-Apollonian packing is of Hausdorff dimension $\delta_K > 1$.
\end{theorem}

\begin{proof}
        It suffices to prove this for the packing $\Pcal_K$, by the M\"obius action.  But $\Pcal_K$ is the limit set of a weak $K$-Apollonian group by Theorem \ref{thm:weak}. 
        For a finitely-generated Kleinian group, the limit set must be one of the following:  totally disconnected, a circle, or of Hausdorff dimension $> 1$ (see, for example, \cite[Corollary 1.8]{\BishopJones} and the citations therein).   However, $\Pcal_K$, since it contains $\widehat{\RR}$ and other circles, is neither totally disconnected, nor a circle.
\end{proof}

Let $X$ be a subgroup of $H(\ZZ)$, where $H$ is a semi-simple Lie group, and let $G=\Zcl(X)$ be its Zariski closure in $H$.  Then $X$ is called \emph{thin} if it is of infinite index in $G(\ZZ)$.  We are interested in the case $H = \operatorname{O}^+_M \cong \Mob$.  We will consider a weak $K$-Apollonian group $\Acal$ to be a subgroup of $\operatorname{O}^+_M(\ZZ) < \operatorname{O}_M(\ZZ)$.

\begin{theorem}
        \label{thm:thin}
        Any weak Apollonian group $\Acal$ for $K$ is thin, and its Zariski closure is either $\operatorname{O}_M$ or $\operatorname{SO}_M$.  
\end{theorem}

In any given situation, to determine which Zariski closure is obtained, it suffices to check whether the generators of $\Acal$ all lie in $\operatorname{SO}_M$.

\begin{proof}
        First, we show that $\Acal$ is of infinite index in $O_M(\ZZ)$.  For, its index is equal to the number of $K$-Apollonian packings in $\Scal_K$, which is infinite (for example, there is a strip packing in every horizontal strip $k \le (z-\overline{z})i/\sqrt{-\Delta} \le k+1$ for $k \in \ZZ$). 
        On the other hand, $\Acal$ is infinite, as its limit set is infinite.

Let $\Gcal$ be the Zariski closure of $\Acal$.  It is necessarily an algebraic subgroup of $O_M$ defined over $\ZZ$.  Therefore $\Gcal(\RR)$ must be a Lie subgroup of $O_M(\RR)$.  Our proof imitates that in \cite[Lemma 1.6(ii)]{MR2832824}.  The classification of the a priori possibilities for $\Gcal$ are:
\begin{enumerate}
        \item A finite subgroup.
        \item A torus or parabolic subgroup.
        \item A subgroup fixing a form of signature $(1,2)$.
        \item $O_M$.
        \item $SO_M$.
\end{enumerate}

We eliminate the first three possibilities in turn.  First, $\Gcal(\ZZ)$ is not finite as $\Acal$ is not finite.  The second two possibilities are subgroups of dimension $\le 2$.  Any finitely generated Kleinian subgroup in those dimensions is geometrically finite.  Therefore the limit set has Hausdorff dimension at most $1$. 
However, the residual set of $\Acal$ has Hausdorff dimension $>1$, by Theorem \ref{thm:limit}.  Therefore $\Gcal = O_M$ or $SO_M$.  
In either case, as a subgroup of infinite index in $O_M(\ZZ)$, $\Acal$ is of infinite index in $\Gcal(\ZZ)$.  Therefore it is thin.
\end{proof}

We now observe that the methods above also provide a proof that the subgroup $E_2(\OK)$ of $\PSL_2(\OK)$ generated by elementary matrices is thin whenever $\OK$ is non-Euclidean.

\begin{theorem}
        \label{thm:Ethin}
        When $\OK$ is non-Euclidean, the groups $E_2(\OK)$ are thin.
\end{theorem}

\begin{proof}
        The group $E_2(\OK)$ is generated by $\PSL_2(\ZZ)$ and the matrix
        \[
                W = \begin{pmatrix}
                        1 & \tau \\
                        0 & 1
                \end{pmatrix}.
        \]

        First we show that $\Lambda(E_2(\OK))$ is the tangency-connected component of $\widehat{\RR}$ in $\SK$.  Consider a circle given by $M(\widehat{\RR})$.  Then the circles $MNW^n(\widehat{\RR})$, where $N \in \PSL_2(\ZZ)$ are exactly those tangent to $\widehat{\RR}$ in $\SK$ (see Proposition \ref{prop:tangentfamilies}).  Therefore the orbit of $\widehat{\RR}$ under $E_2(\OK)$ is the tangency-connected component of $\SK$ containing $\widehat{\RR}$.  Since $\PSL_2(\ZZ) < E_2(\OK)$, all of $\widehat{\RR}$ is in the limit set, so the entire tangency-connected component of $\SK$ is $\Lambda(E_2(\OK))$.
       
        By \cite[Theorem 7.1]{VisOne}, there are infinitely many tangency-connected components in $\SK$.  Hence $E_2(\OK)$ is of infinite index in $\PSL_2(\OK)$. 
        By the same method as Theorem \ref{thm:limit}, this component has Hausdorff dimension exceeding $1$.  The rest of the proof of thinness is as in Theorem \ref{thm:thin}. 
\end{proof}

Next we define \emph{strong} $K$-Apollonian groups, also called simply \emph{$K$-Apollonian groups}.  These are weak Apollonian groups with extra structure captured in their relationship to a certain {cluster type}.  We are motivated by the traditional Apollonian group, which is described by its relationship to Descartes quadruples.
\begin{definition}
        \label{defn:algapp}
        Let $\Acal < \Mob$ and fix a cluster space and corresponding cluster type.  Suppose that $\Acal$ is a weak Apollonian group for $K$ for which the set of unordered clusters in $\Pcal_K$ is a principal homogeneous space for $\Acal$.  Then we say that $\Acal$ is a \emph{strong $K$-Apollonian group with respect to the cluster type} or simply a \emph{$K$-Apollonian group} when no confusion will occur.
        If a base cluster is fixed, then $\Acal$ corresponds under the algebraic-geometric correspondence to a subgroup $\widehat{\Acal}$ of $O_R(\ZZ)$.  In this case, $\widehat{\Acal}$ is called an \emph{algebraic $K$-Apollonian group}, and $\Acal$ is called a \emph{geometric} $K$-Apollonian group when distinction is necessary.
\end{definition}

The group $\widehat{\Acal}$ acts on \emph{ordered} clusters but not on unordered clusters or circles.

\begin{theorem}
        \label{thm:app-union}
        The group $\Acal_{\QQ(i)}$ defined in Section \ref{sec:gaussian} is a geometric $\QQ(i)$-Apollonian group, and $\widehat{\Acal_{\QQ(i)}}$ is an algebraic $\QQ(i)$-Apollonian group.
\end{theorem}

\begin{proof}
        Theorem \ref{thm:gtoneorbits}, via Proposition \ref{prop:leftrightorbit}, tells us that the geometric action of $\Acal_{\QQ(i)}$ is transitive on the set of unordered Descartes quadruples.  Furthermore, since the algebraic action is induced by the action of $\operatorname{O}_R(\RR)$ on $S_R$, we also know from Theorem \ref{thm:gtoneorbits} that there are no automorphisms of a cluster in $\Acal_{\QQ(i)}$'s algebraic action, hence no elements fixing the base cluster in its geometric action.  This proves that the set of unordered quadruples is a principle homogeneous space for $\Acal_{\QQ(i)}$'s geometric action.  That its limit set is the strip packing is well known.
\end{proof}

A general method of proving groups are Apollonian will be developed in Section \ref{sec:suff}.

\section{Topographical Groups}
\label{sec:topographical}

For this section, we concern ourselves with certain special subgroups of $\PGL_2(\ZZ)$.

\begin{definition}
        \label{defn:superbasis}
        A \emph{superbasis} is a triple $(a,b,c)$ of points of $\widehat{\QQ}$ which are pairwise distinct modulo all primes.  
       \end{definition}

The use of the terminology \emph{superbasis} is borrowed from Conway and Fung \cite[The First Lecture]{\Conway}:  for them, $\mathbf{a}= [a_1,a_2], \mathbf{b} = [b_1, b_2], \mathbf{c}=[c_1,c_2] \in \ZZ^2$ form a superbasis if they are primitive vectors (i.e. $\gcd(a_1,a_2)=\gcd(b_1,b_2)=\gcd(c_1,c_2)=1$), and each pair forms a basis for $\ZZ^2$.  It is evident that the two definitions are naturally in bijection, where the vectors $\mathbf{a}, \mathbf{b}, \mathbf{c} \in \PP^1(\ZZ)$ represent the elements $a = a_1/a_2,b = b_1/b_2,c=c_1/c_2 \in \widehat{\QQ}$.  

Conway shows that the graph whose vertices are unordered superbases, where an edge indicates that two superbases share a basis, is a single tree of valence three, shown in Figure \ref{fig:topograph}.  Conway calls this graph the \emph{topograph}.

A matrix of $\PGL_2(\ZZ)$ with columns $\mathbf{a}$ and $\mathbf{b}$ can be interpreted as the superbasis $\mathbf{a}$, $\mathbf{b}$, $\mathbf{a+b}$ as above.  This gives a six-to-one map
\[
        \phi: \PGL_2(\ZZ) \rightarrow \{ \mbox{unordered superbases} \}.
\]
The matrices which are interpreted as the superbasis $0, 1, \infty$, in some order, are
\[
      S :=  \left \{ 
        \begin{pmatrix} 1 & 0 \\ 0 & 1 \end{pmatrix},
        \begin{pmatrix} 0 & 1 \\ 1 & 0 \end{pmatrix},
        \begin{pmatrix} 1 & -1 \\ 0 & -1 \end{pmatrix},
        \begin{pmatrix} 0 & -1 \\ 1 & -1 \end{pmatrix},
        \begin{pmatrix} 1 & -1 \\ 1 & 0 \end{pmatrix},
        \begin{pmatrix} 1 & 0 \\ 1 & -1 \end{pmatrix}
      \right \} < \PGL_2(\ZZ).
\]
The set $S$ is a group isomorphic to $S_3$.  Left cosets of $S$ inside $\PGL_2(\ZZ)$ are in bijection with unordered superbases.  

\begin{definition}\label{defn:topgp}
        A \emph{topographical group} is a subgroup of $\PGL_2(\ZZ)$ for which $\phi$ gives an isomorphism from its Cayley graph under right multiplication (with respect to suitable generators) to Conway's topograph.
\end{definition}

By the Cayley graph under right multiplication, it is meant that the edges $g$ and $gs$ (not $sg$) are joined for each generator $s$.

There are only finitely many topographical groups.  For, suppose $H$ has a Cayley graph isomorphic to Conway's topograph under $\phi$.  In particular, it is generated by three elements of order $2$ which take $\{0,1,\infty\}$ to $\{0, -1, \infty\}$, $\{0,1,1/2\}$ and $\{1,2,\infty\}$, respectively.  There are only finitely many possibilities mapping $\{ 0, 1, \infty \}$ to each of these sets, and most possibilities are not of order $2$.  The remaining elements are:
\begin{gather*}
\gamma_1 = \begin{pmatrix} -1 & 2 \\ -1 & 1 \end{pmatrix}, \quad
\gamma_2 = \begin{pmatrix} 1 & -1 \\ 2 & -1 \end{pmatrix}, \quad
 \gamma_3 = \begin{pmatrix} 0 & 1 \\ -1 & 0 \end{pmatrix}, \\
\rho_1 = \begin{pmatrix} -1 & 0 \\ 0 & 1 \end{pmatrix}, \quad
\rho_2 = \begin{pmatrix} -1 & 2 \\ 0 & 1 \end{pmatrix}, \quad
\rho_3 = \begin{pmatrix} -1 & 0 \\ 2 & 1 \end{pmatrix}.
\end{gather*}
 
Write $\Pi = \PGL_2(\ZZ)$ and $\Gamma = \PSL_2(\ZZ)$.  Write $\Pi(N)$ and $\Gamma(N)$ for their congruence subgroups of level $N$, respectively.

\begin{theorem}\label{thm:topo-equiv}
               The only topographical groups which are normal in $\Pi$ are $G := \Gamma^3 = \langle \gamma_1, \gamma_2, \gamma_3 \rangle$ and $P := \Pi(2) = \langle \rho_1, \rho_2, \rho_3 \rangle$.
\end{theorem}

Note that these two groups are normal but not characteristic:  the outer automorphism of $\PGL_2(\ZZ)$ maps one to the other (see \cite{MR0460483, Jimm} for more on the outer automorphism).  One is a congruence subgroup of $\PGL_2(\ZZ)$ and the other is not; this is a general phenomenon with regards to the outer automorphism \cite{MR859146}.

\begin{proof}
        Suppose a group $H$ is normal and unordered bases form a principal homogeneous space for $H$ under the M\"obius action.  Then $H$ is normal of index $6$ in $\Pi = \PGL_2(\ZZ)$.  From the classification of normal subgroups of small index in $\Gamma = \PSL_2(\ZZ)$ due to Newman \cite{MR0204375}, we immediately observe that $H \cap \Gamma$ is one of $\Gamma(2)$, $\Gamma'$, or $\Gamma^3$, where $\Gamma'$ represents the commutatator subgroup, and $\Gamma(2)$ the congruence subgroup of level $2$.  The first two of these groups are of index $6$ in $\Gamma$ and the latter is of index $3$.  All of these groups contain $\Gamma(12)$.  But the only normal subgroups of $\Pi$ of index $6$ containing $\Gamma(12)$ are $G = \Gamma^3$ and $P = \Pi(2)$ (see \cite{MR2254377}), whose intersections with $\Gamma$ are exactly $\Gamma^3$ and $\Gamma(2)$, respectively.

        Let
\[
        G_0 = \langle \gamma_1, \gamma_2, \gamma_3 \rangle,
\]
where
\[
\gamma_1 = \begin{pmatrix} -1 & 2 \\ -1 & 1 \end{pmatrix}, \quad
\gamma_2 = \begin{pmatrix} 1 & -1 \\ 2 & -1 \end{pmatrix}, \quad
 \gamma_3 = \begin{pmatrix} 0 & 1 \\ -1 & 0 \end{pmatrix},
\] 
and
\[
        P_0 = \langle \rho_1, \rho_2, \rho_3 \rangle,
\]
where
\[
\rho_1 = \begin{pmatrix} -1 & 0 \\ 0 & 1 \end{pmatrix}, \quad
\rho_2 = \begin{pmatrix} -1 & 2 \\ 0 & 1 \end{pmatrix}, \quad
\rho_3 = \begin{pmatrix} -1 & 0 \\ 2 & 1 \end{pmatrix}.
\]
The generators $\gamma_i$ and $\rho_i$ are each of order $2$.   Next we verify that $G_0 = G$ and $P_0 = P$.
If we set, as usual,
\[
S = \begin{pmatrix} 0 & 1 \\ -1 & 0 \end{pmatrix},
T = \begin{pmatrix} 1 & 1 \\ 0 & 1 \end{pmatrix} \in \PGL_2(\ZZ),
\]
then 
$
\gamma_1 = TST^{-1}$,
$\gamma_2 = T^{-1}ST$, and
 $\gamma_3 = S$.
 Since $\gamma_1\gamma_2\gamma_3 = T^3$, $G = \Gamma^3 < G_0$.  But $\Gamma^3$ is normal, so it contains all conjugates of the element $S$, hence $G_0 < \Gamma^3 = G$.  That the $\rho_i$ generate $P = \Pi(2)$ is immediate. 
 
 Finally, to verify that $G$ and $P$ are topographical is now a simple computation using the Cayley graph on the generators above.
\end{proof}

\begin{figure}
        \includegraphics[height=3in]{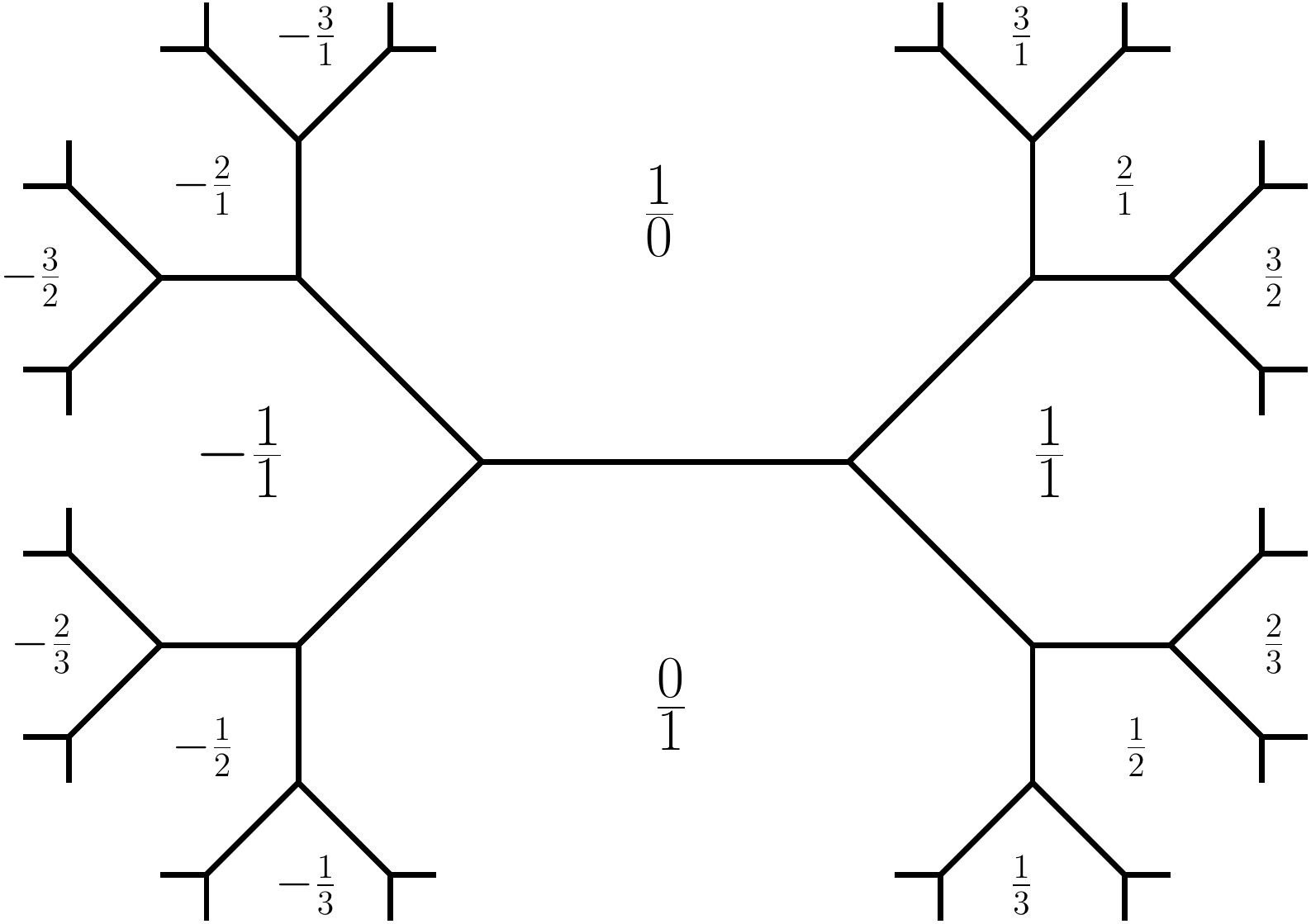} 
        \caption[Conway's Topograph]{Conway's topograph is a valence-three tree embedded in the plane such that it breaks up the plane into many infinite regions, each labelled with an element of $\PP^1(\ZZ)$.  The vertices represent unordered superbases $a,b,c$ where $a,b,c,$ are given by the three regions adjacent to the vertex.  For example, the right-central vertex corresponds to the superbasis $0,1, \infty$.}
\label{fig:topograph}
\end{figure}

We will need the following lemma.  Let 
\[
B = \left\{ \begin{pmatrix} 1 & n \\ 0 & 1 \end{pmatrix} : n \in \ZZ \right\},
\]
which is the stabilizer of $\infty$.

\begin{lemma}
        \label{lemma:gamma-borel}
        \[
                G \cap B = \left\langle \gamma_1\gamma_2\gamma_3 \right\rangle = \left\{ \begin{pmatrix} 1 & 3n \\ 0 & 1 \end{pmatrix} : n \in \ZZ
                \right\} \cong \ZZ.
\]
\end{lemma}

We will use the notation $G_B := G \cap B$.

\begin{proof}
        We use the identification of the Cayley graph of $G$ under right multiplication with the topograph, as in the last theorem.  One can draw the topograph in the plane so that the infinite `regions' between branches are labelled with elements of $\widehat{\QQ}$, and the `shoreline' consists of those superbases and bases containing that element, in the sense of Conway (see \cite[The First Lecture]{MR1478672} for details).  The stabilizer of $\infty$ is exactly those elements of $G$ which map $\infty, 0, -1$ to a superbasis surrounding the region $\infty$, \emph{and keep $\infty$ in first position}.  Travelling along the `shore' of the $\infty$ region by explicit computation, we see that this is exactly the group generated by $\gamma_1 \gamma_2 \gamma_3$.  Finally,
\[
\gamma_1\gamma_2\gamma_3 = \begin{pmatrix} 1 & 3 \\ 0 & 1 \end{pmatrix}.
\]
\end{proof}

The structure of $\Acal_{\QQ(i)}$ and the collection of quadruples has much in common with the structure of topographical groups and the collection of superbases.  In particular, the collection of (unordered) Descartes quadruples can be considered the vertices of a graph, where edges indicate that two quadruples share a subset of three circles.  This graph is a tree of valence four which is the Cayley graph of $\Acal_{\QQ(i)}$ with respect to the four given generators.  Each generator swaps out one circle of a quadruple, changing it into an adjacent quadruple.  For more on this, see \cite{AppPrevious}.  We will imitate this structure for other fields in later sections.

Finally, it is worth remarking that superbases can be interpreted as the tangency points of a triple of intervals covering $\widehat{\RR}$; such intervals can be put in bijection with elements of $\RR^3$ lying on a hypersuface; this develops a story analogous to that of Section \ref{sec:space} with the exceptional isomorphism $\operatorname{O}^+_{2,1}(\RR) \cong \PGL_2(\RR)$ in place of $\operatorname{O}^+_{3,1}(\RR) \cong \PGL_2(\CC)$, and quadratic forms in place of Hermitian forms.  This has pleasing connections to the rational parametrisation of Pythagorean triples.

\section{Sufficient conditions for Apollonian groups}
\label{sec:suff}

In this section we provide a set of sufficient conditions guaranteeing a group is Apollonian.  The main theorem, Theorem \ref{thm:sufficient-clusters}, will be used to verify all the example Apollonian groups in the remainder of the paper.

Let $\Mob_{\widehat{\RR}}$ be the set of M\"obius transformations fixing $\widehat{\RR}$.  Define
\[
        \eta: \Mob_{\widehat{\RR}} \rightarrow \PGL_2(\RR)
\]
to be the restriction of the action of such a M\"obius transformation to $\widehat{\RR}$.  This map is surjective.

We will need the following technical term for the theorem.
\begin{definition}
        The \emph{base prong} for $\SK$ is the collection of four $K$-Bianchi circles consisting of $\widehat{\RR}$ and three circles immediately tangent to $\widehat{\RR}$ at $0$, $1$ and $\infty$.
        A \emph{three-prong} is a collection of four circles that has the same pairwise Pedoe products as the \emph{base prong}.  Equivalently, the \emph{three-prongs} are the clusters in the cluster space containing the base prong.  A three-prong is said to be \emph{centred on} any of its circles which is tangent to the other three.  
                \end{definition}

                In the case of $\QQ(i)$, three-prongs are exactly Descartes quadruples and a prong is centred on all of its circles. In other cases, the central circle is unique. 
                Evidently, the base prong exists and is unique for a particular Schmidt arrangement $\SK$.  The matrix $R$ for which the base prongs correspond to the cluster space $S_R$ is given explicitly in the next section, but is not needed at the moment.

                \begin{lemma}
                        \label{lemma:3prong}
                        Three-prongs centred on $\widehat{\RR}$ are in bijection with superbases.
                \end{lemma}

                \begin{proof}
                        Any $K$-Bianchi circle tangent to $\widehat{\RR}$ at $\alpha/\beta$ (in lowest terms) is given by
                        \[
                                \begin{pmatrix}
                                        \alpha & -\gamma + \alpha k\tau \\
                                        \beta & -\delta + \beta k \tau
                                \end{pmatrix}
                        \]
                        where $\alpha\delta-\beta\gamma = 1$, $\alpha,\delta,\beta,\gamma \in \ZZ$, and where $k=1$ if and only if the circle is immediately tangent (Proposition \ref{prop:tangentfamilies}).  Therefore its Pedoe embedding into $\MM$ is given by Proposition \ref{prop:curvature}:
                        \[
                                b = k(\overline{\tau}-\tau)\beta^2i, \quad
                                b' = k(\overline{\tau}-\tau)\alpha^2i, \quad
                                a = i(-(\alpha\delta-\beta\gamma) + k(\overline{\tau}-\tau)\alpha\beta) = i(k(\overline{\tau}-\tau)-1).
                        \]
                        Then the Pedoe product of two $K$-Bianchi circles tangent to $\widehat{\RR}$ at $\alpha_1/\beta_1$ and $\alpha_2/\beta_2$ (in lowest terms), is
                        \[
                                1+\frac{1}{2}k_1k_2(\overline{\tau}-\tau)^2(\alpha_1\beta_2 - \alpha_2\beta_1)^2
                        \]
                        where $k_1, k_2 \in \ZZ$.  Therefore, we obtain
                        \[
                                1+\frac{1}{2}(\overline{\tau}-\tau)^2
                        \]
                        if and only if $\alpha/\beta$ and $\gamma/\delta$ satisfy $\alpha\delta-\beta\gamma=\pm 1$ and the circles are immediately tangent.  The result follows.
                \end{proof}

\begin{theorem}
        \label{thm:sufficient-clusters}
        Let $\Acal < \langle \PSL_2(\OK), \mathfrak{c} \rangle < \Mob$ is a finitely generated Kleinian group.  Suppose that $\Acal$ takes circles in $\Pcal_K$ to circles in $\Pcal_K$.
        Suppose further that there is a cluster type, and a cluster (call it the \emph{base cluster}), for which the following conditions hold:
        \begin{enumerate}[(i)]
                \item \label{item:basecontained} The base cluster is contained in $\Pcal_K$,
                \item \label{item:baseunique} The base cluster is the unique cluster containing the base prong,
                \item \label{item:nearbysuper} In the orbit of the base cluster, there are clusters containing the three base prongs centred on $\widehat{\RR}$ with tangencies at $(0,1/2,1)$, $(1,2,\infty)$ and $(-1,0,\infty)$.
                \item \label{item:tangentstep} There are elements of $\Acal$ taking a cluster with three-prong centred on $\widehat{\RR}$ to a cluster with three-prong centred on the circles immediately tangent to $\widehat{\RR}$ at $0$, $1$ and $\infty$.
                \item \label{item:noauto} No automorphism of the base cluster is contained in $\Acal$,
                \item \label{item:limit} If $\Bcal_1$ and $\Bcal_2$ are two three-prongs centred on the same circle $C$, but otherwise sharing no circles, and if $\Ccal_1$ and $\Ccal_2$ are clusters containing $\Bcal_1$ and $\Bcal_2$, respectively, then $\Ccal_1$ and $\Ccal_2$ have no circles in common besides $C$.
                \item \label{item:connected} The base cluster is tangency connected.
        \end{enumerate}
        Then, $\Acal$ is an Apollonian group for $K$.
\end{theorem}

\begin{proof}
        First, we will show that all clusters are in the orbit of the base cluster.  The orbit of the base cluster under the geometric (left multiplication) or algebraic (right multiplication) actions is the same (see Proposition \ref{prop:leftrightorbit}).  This allows us to focus on the algebraic action when needed.
        
        To begin, we will show that any cluster in $\Pcal_K$ which contains a three-prong centred on $\widehat{\RR}$ is in the orbit of the base cluster.  It suffices to observe that
  \begin{enumerate}
          \item \label{item:onecluster} there is exactly one cluster in $\Pcal_K$ containing any given three-prong centred on any fixed circle (this is by assumptions \eqref{item:basecontained} and \eqref{item:baseunique}, and the M\"obius action),
          \item the three-prongs centred on $\widehat{\RR}$ are in bijection with superbases (Lemma \ref{lemma:3prong}), and 
          \item the orbit of the base cluster contains a cluster centred on $\widehat{\RR}$ at any superbasis.  This follow by induction from \eqref{item:nearbysuper}.  More precisely, the topograph is connected, and the three superbases specified are the three superbases adjacent to $(0,1,\infty)$ in the topograph.
  \end{enumerate}

  Let $C$ be the circle immediately tangent to $\widehat{\RR}$ at $0$, $1$ or $\infty$.  By assumptions \eqref{item:tangentstep} and \eqref{item:basecontained}, there is, in the orbit of the base cluster, a cluster containing a three-prong centred on $C$.  But then by the last observation, we see that in the orbit of the base cluster there is a cluster containing a three-prong centred on any circle immediately tangent to $\widehat{\RR}$.

  By induction we see that in fact for any circle $C$ in $\Pcal_K$, since it is obtained from $\widehat{\RR}$ by a finite chain of immediate tangencies, there is in the orbit of the base cluster, a cluster containing a three-prong centred on $C$.  Every cluster contains some three-prong centred on some circle, since the base cluster does by assumption, and that all clusters have the same pairwise Pedoe products.  So by observation \eqref{item:onecluster}, every cluster is in the orbit of the base cluster.  Hence all clusters are in the orbit of any cluster, by the algebraic action.

  That $\Acal$ acts freely and transitively on the clusters of $\Pcal_K$ now follows from assumption \eqref{item:noauto}.

  Now we must show that $\Lambda(\Acal) = \overline{\Pcal_K}$.  
  First, the immediate tangency points of any fixed circle $C \in \Pcal_K$ with other circles of $\Pcal_K$ are dense in $C$, by Proposition \ref{prop:immtang}.
  From this observation, it suffices to show that all tangency points of all clusters lie in $\Lambda(\Acal)$ to conclude that $\Pcal_K \subset \Lambda(\Acal)$.  Since the orbit of $0$ is contained in $\Pcal_K$, it would follow that $\Lambda(\Acal) = \overline{\Pcal_K}$.

  First, $0$ is in $\Lambda(\Acal)$, and therefore by the action of $\Acal$, every cluster contains at least one point of $\Lambda(\Acal)$ among its tangency points.  

  Next we will show that there is a sequence of clusters whose tangency points all approach $0$.  First observe that there is a sequence of disjoint superbases approaching $0$.  To finish the proof, we will show that the tangency points of the clusters containing the superbases also approach $0$.  Consider any two superbases in this sequence which are disjoint; then the clusters they define may not share any circles besides $\widehat{\RR}$ by \eqref{item:limit}.  Since there are only finitely many circles of curvature below any fixed bound, the curvatures in the cluster (besides that of $\widehat{\RR}$) must be approaching $\infty$.  For a tangency connected cluster (item \eqref{item:connected}), this implies the tangency points are all approaching $0$.

        By the observation of the previous paragraph, there is a sequence of clusters whose tangency points all approach $0$.  By M\"obius symmetry, for any tangency point $x$ of $\Pcal_K$, there is a sequence of clusters whose tangency points all approach $x$.  Since every cluster carries a point of $\Lambda(\Acal)$ among its tangency points, this shows that $x \in \Lambda(\Acal)$ and we are done.
\end{proof}

In the following sections we define example Apollonian groups for each imaginary quadratic field.  These definitions are not unique, and the author has endeavoured to find pleasing choices.

\section{$K$-clusters for imaginary quadratic fields}

In this section, we develop a general theory that yields an Apollonian group for any imaginary quadratic $K$ (save $\QQ(\sqrt{-3})$, as usual).  In the Euclidean cases, it is possible to replace this group with one which is freely generated by elements of order two, and this is the purpose of later sections.  We begin, however, taking $K$ generally.

\begin{definition}
        \label{def:WD}
Given a set, $D$, of four oriented circles $C_1, \ldots, C_4$ corresponding to vectors $\bfv_1, \ldots, \bfv_4$ in $\MM$, define the matrix $W_D$ formed of the columns
\begin{equation}
        \label{eqn:cluster0}
        \bfv_1, \bfv_2, \bfv_3, \bfv_4.
\end{equation}
We say that $C_1, \ldots, C_4$ form a \emph{$K$-cluster}, or \emph{$K$-quadruple}, if $W_D$
satisfies the relationship
        $W_D^t G_M W_D = R$,
where the matrix $R = R_0$ is defined to
\[
        R_0 := \begin{pmatrix}
                1 & -1 & -1 & -1 \\
                -1 & 1 & 1+\frac12\Delta & 1+\frac12\Delta \\
                -1 & 1+\frac12\Delta & 1 & 1+\frac12\Delta \\
                -1 & 1+\frac12\Delta & 1+\frac12\Delta & 1\\
        \end{pmatrix}.
\]
\end{definition}

\begin{definition}
In the case that $\Delta \equiv 1 \pmod 4$, it will be convenient to change variables so that $W_D$ is the matrix represented by the columns
\begin{equation}
        \label{eqn:cluster1}
        \bfv_1, \frac{1}{2}(-\bfv_2 + \bfv_3+\bfv_4), \frac{1}{2}(\bfv_2 - \bfv_3 + \bfv_4), \frac{1}{2}(\bfv_2+\bfv_3-\bfv_4).
\end{equation}
Then, $R$ becomes
\[
        R_1 := \begin{pmatrix}
                1 & -1/2 & -1/2 & -1/2 \\
                -1/2 & \frac{1-\Delta}{4} & \frac{1+\Delta}{4} & \frac{1+\Delta}{4} \\
                -1/2 & \frac{1+\Delta}{4} & \frac{1-\Delta}{4} & \frac{1+\Delta}{4} \\
                -1/2 & \frac{1+\Delta}{4} & \frac{1+\Delta}{4} & \frac{1-\Delta}{4} \\
        \end{pmatrix}.
\]
\end{definition}

The change of variables is given by $R_0 = S^t R_1 S$ where
\[
        S = \begin{pmatrix}
                1 & 0 & 0 & 0 \\
                0 & 0 & 1 & 1 \\
                0 & 1 & 0 & 1 \\
                0 & 1 & 1 & 0 \\
        \end{pmatrix}.
\]
The interpretation of the cluster space as a space of clusters of circles is unchanged.
This is done to simplify the generators in the next definition.

We remark that 
\[
        R_0^{-1} = \frac{1}{\Delta}
        \begin{pmatrix}
                \Delta + 3 & 1 & 1 & 1 \\
                1 & -1 & 1 & 1 \\
                1 & 1 & -1 & 1 \\
                1 & 1 & 1 & -1 \\
        \end{pmatrix}\quad \mbox{ and } \quad
        R_1^{-1} = \frac{1}{\Delta}
        \begin{pmatrix}
                \Delta + 3 & 2 & 2 & 2 \\
                2 & 0 & 2 & 2 \\
                2 & 2 & 0 & 2 \\
                2 & 2 & 2 & 0 \\
        \end{pmatrix}.
\]
In particular, that tells us the ``Descartes equation'' for $K$-clusters: the curvatures $a,b,c,d$ of four circles in a $K$-cluster satisfy $2(a^2 + b^2 + c^2 + d^2) - (a+b+c+d)^2 - (\Delta+4)a^2 = 0$.  

\begin{figure}
        \includegraphics[height=3in]{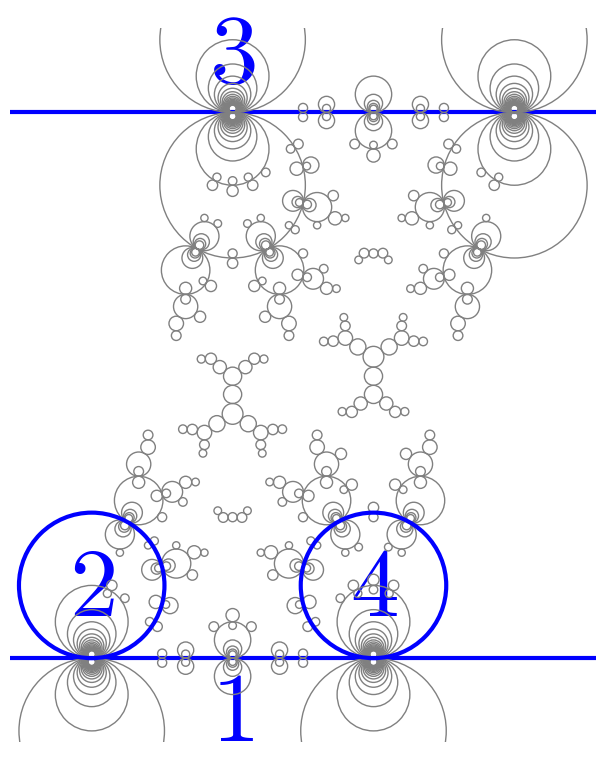} 
        \caption[The base $K$-cluster]{The base $K$-cluster, here shown for $\Scal_{\QQ(\sqrt{-15})}$.} 
\label{fig:kcluster}
\end{figure}

In particular, a $K$-cluster always consists of a three-prong as defined in the last section, and $S_R$ is the cluster space of three-prongs.  In the case $\Delta \neq -4$, these three circles are disjoint from one another, as in Figure \ref{fig:kcluster}.  The definition above reduces to the usual Descartes quadruple (where they are all mutually tangent) when $\Delta = -4$.

\begin{definition}
        \label{def:Ak}
We define the matrix group $\widehat{\Acal_K^0} \subset O_{R}(\RR)$ to be the group generated by the following generators.  For the case $\Delta \equiv 0 \pmod 4$:
\begin{equation}
        \label{eqn:kgens0}
        \begin{pmatrix}
                1 & 2 & 0 & 0 \\
                0 & -1 & 0 & 0 \\
                0 & 2 & 0 & 1 \\
                0 & 2 & 1 & 0
        \end{pmatrix},
        \begin{pmatrix}
                1 & 0 & 2 & 0 \\
                0 & 0 & 2 & 1 \\
                0 & 0 & -1 & 0 \\
                0 & 1 & 2 & 0
        \end{pmatrix},
        \begin{pmatrix}
                1 & 0 & 0 & 2 \\
                0 & 0 & 1 & 2 \\
                0 & 1 & 0 & 2 \\
                0 & 0 & 0 & -1 
        \end{pmatrix},
        \begin{pmatrix}
                0 & 1+\Delta/4 & 1 & 1+\Delta/4 \\
                0 & 1 & 0 & 0 \\
                1 & -1-\Delta/4 & 0 & -1-\Delta/4  \\
                0 & 0 & 0 & 1 
        \end{pmatrix},
\end{equation}
while for the case $\Delta \equiv 1 \pmod 4$:
\begin{equation}
        \label{eqn:kgens1}
        \begin{pmatrix}
                1 & -1 & 1 & 1 \\
                0 & -1 & 2 & 2 \\
                0 & 0 & 0 & 1 \\
                0 & 0 & 1 & 0
        \end{pmatrix},
        \begin{pmatrix}
                1 & 1 & -1 & 1 \\
                0 & 0 & 0 & 1 \\
                0 & 2 & -1 & 2 \\
                0 & 1 & 0 & 0
        \end{pmatrix},
        \begin{pmatrix}
                1 & 1 & 1 & -1 \\
                0 & 0 & 1 & 0 \\
                0 & 1 & 0 & 0 \\
                0 & 2 & 2 & -1 
        \end{pmatrix},
        \begin{pmatrix}
                0 & 1 & \frac{\Delta+3}{4} & 0 \\
                1 & 1 & -\frac{\Delta-1}{4} & -1 \\
                0 & 0 & 1 & 0 \\
                1 & 0 & -\frac{\Delta+3}{4} & 0 
        \end{pmatrix}.
\end{equation}

\end{definition}
We wish to associate to this a group of M\"obius transformations via the algebraic-geometric correspondence.
 Define the \emph{base cluster} $D$ to be given by the images of $\widehat{\RR}$ under
\[
        \begin{pmatrix}
                1 & 0 \\ 0 & 1
        \end{pmatrix}, \quad
        \begin{pmatrix}
                -1 & \tau \\ 0 & 1
        \end{pmatrix}, \quad
        \begin{pmatrix}
                0 & -1 \\ -1 & \tau-1 
        \end{pmatrix}, \quad
        \begin{pmatrix}
                1 & 1 - \tau \\ 1 & -\tau
        \end{pmatrix}.
\]
The associated matrix $W_{{D}}$ to this $K$-cluster is, for $\Delta \equiv 0 \pmod 4$,
\[
        W_{{D}} = W_{{D}}^0 := 
        \begin{pmatrix}
                0 &  0 & \eta & \eta \\
                0 & \eta & 0 & \eta \\
                0 & 0 & 0 & \eta \\
                -1 & 1 & 1 & 1
        \end{pmatrix},
\]
and for $\Delta \equiv 1 \pmod 4$,
\[
        W_{{D}} = W_{{D}}^1 := 
        \begin{pmatrix}
                0 &  \eta & 0 & 0 \\
                0 & 0 & \eta & 0 \\
                0 & \eta/2 & \eta/2 & -\eta/2 \\
                -1 & 1/2 & 1/2 & 1/2
        \end{pmatrix},
\]
where $\eta = i(\overline{\tau}-\tau) =  \sqrt{-\Delta}$.  Note that $W_{{D}}^0 = W_{{D}}^1 S$.

The four generators of $\widehat{\Acal_K^0}$, under the algebraic-geometric correspondence via the base quadruple, become
\begin{equation}
        \label{eqn:sixmob}
        \begin{pmatrix}
                -1 & 2 \\
                -1 & 1 
        \end{pmatrix}, 
        \begin{pmatrix}
                1 & -1 \\
                2 & -1 
        \end{pmatrix}, 
        \begin{pmatrix}
                0 & 1 \\
                -1 & 0 
        \end{pmatrix}, 
        \begin{pmatrix}
                1 & 1-\tau \\
                0 & -1 
        \end{pmatrix}.
\end{equation}
The reader will recognise the first three matrices as $\gamma_1,\gamma_2,\gamma_3$, so that $G < \widehat{\Acal_K^0}$.
The last matrix takes $\widehat{\RR}$ to the circle tangent to $\widehat{\RR}$ at $\infty$.  
The group generated by these M\"obius transformations will be called $\Acal_K^0$.
Note that $\Acal_{\QQ(i)} \neq \Acal_{\QQ(i)}^0$, where $\Acal_{\QQ(i)}$ is as defined in Section \ref{sec:gaussian}.

\begin{theorem}
        \label{thm:general-free}
        Suppose $\Delta \leq -15$.   
        Then $\Acal_K^0$ (or $\widehat{\Acal_K^0}$) has the presentation
        \[
                \langle s_1, s_2, s_3, r : s_1^2= s_2^2= s_3^2 = r^2 = 1, rs_1s_2s_3 = s_3s_2s_1r \rangle.
        \]
        In fact, this presentation may be realised by taking $s_1, s_2, s_3, r$ to be the matrices \eqref{eqn:kgens0} or \eqref{eqn:kgens1}, depending on whether $\Delta \equiv 0$ or $1 \pmod 4$.
\end{theorem}

To prove this, we need a result from the Bass-Serre theory of groups acting on trees (in our case, the tree will be a tangency-connected component of the graph of immediate tangencies of $\mathcal{P}$; see Definition \ref{def:imm-tang-graph}).

        \begin{theorem}[Serre {\cite[Section 1.4]{Trees}}]
                \label{thm:bass-serre}
                Suppose that a group $G$ acts on a tree $X$ with inversion (i.e. there exists $g \in G$ and $e$ an edge of $X$ such that $g \cdot e$ is again $e$, but with orientation reversed).  Suppose the action is transitive on vertices and transitive on edges.  Let $v$ be a vertex of $X$ and $e$ be an adjacent edge.  Let $G_v$ be the stabilizer of $v$, $G_e$ be the stabilizer of $e$, and let $G'$ be the stabilizer of $v$ and $e$ (hence preserving the direction of $e$).  Then $G'$ is of index two in $G_e$ and
                \[
                        G \cong G_v *_{G'} G_e,
                \]
                where $*$ denotes the free product with amalgamation.
        \end{theorem}

        We will also need a lemma.

        \begin{lemma}
                \label{lemma:stab}
                The stabilizer of $\widehat{\RR}$ in $\Acal_K^0$ is $G$.
        \end{lemma}

        \begin{proof}
                Write $\gamma_1, \gamma_2, \gamma_3, \rho$ for the generators \eqref{eqn:sixmob} of $\Acal_K^0$.  It is clear that $G$ is in the stabilizer, since $G = \langle \gamma_1, \gamma_2, \gamma_3 \rangle$.

                To show the opposite containment, let $w$ be a word in the $\gamma_i$ and $\rho$ which lies in $\PGL_2(\RR)$.  Build up the word $w$ character-by-character from the left, applying the algebraic right action to the base cluster, to produce a series of three-prongs.  A character changes the central circle of the three-prong only if it is equal to $\rho$.  When it changes, it changes to an immediately tangent circle.  Since the immediate tangency graph is a tree (Theorem \ref{thm:16noloops}), the sequence of central circles must be of the form
                \[
                        \widehat{\RR}, C_1, C_2, \ldots, C_{n-1}, C_n, C_{n-1}, \ldots, C_2, C_1, \widehat{\RR}.
                \]
                Define the middle substring $s = \rho \gamma_{i_1} \cdots \gamma_{i_j} \rho$ of the word corresponding to the central circles
                \[
                        C_{n-1}, C_n, C_{n-1}.
                \]
                Then this word, acting on the right, fixes $C_{n-1}$ and therefore is an element of $\PGL_2(\RR)$.  
                
                There are two possible cases:
                        $s \in G$
                        or $s \notin G$.
                In the first case, we can replace this part of the word with a string formed of $\gamma_i$'s.  Repeating this process, we either eliminate all $\rho$'s, so that $w \in G$, or we arrive in the second case.  Therefore it suffices to show that whenever $\gamma \in G$ and $\rho\gamma \rho \in \PGL_2(\RR)$, then $\rho \gamma \rho \in G$.

                Let $G' = \rho G \rho$.  Then it is a direct computation that $G' \cap \PGL_2(\RR) = G_B$, and we are done.
        \end{proof}

        \begin{proof}[Proof of Theorem \ref{thm:general-free}]
                The group $\widehat{\Acal_K^0}$ is isomorphic to the group ${\Acal_K^0}$ of M\"obius transformations given by generators \eqref{eqn:sixmob}.  As M\"obius transformations act on circles, preserving tangencies, and ${\Acal_K^0}$ takes a $K$-Apollonian packing $\mathcal{P}$ back to itself, it acts on the (immediate) tangency tree of $\mathcal{P}$ (it is a tree by Theorem \ref{thm:16noloops}).  The fourth generator of \eqref{eqn:sixmob}, call it $r$,
                reverses the edge between circle $\widehat{\RR}$ and $\widehat{\RR} + \tau$.  Note that the topographical group $G$ is a subgroup of ${\Acal_K^0}$.

                The orbit of $r G$ acting on $\widehat{\RR}$ is the collection of $K$-Bianchi circles immediately tangent to $\widehat{\RR}$.  Thus ${\Acal_K^0}$ maps $\widehat{\RR}$ to all immediately tangent circles.  Therefore the action is transitive on vertices.
               The orbit of $\infty$ under $G$ is $\widehat{\QQ}$.  This implies that the action is transitive on edges (which correspond to tangency points). 

               The stabilizer of $\infty$ is $\langle G_B, r \rangle$ (see Lemma \ref{lemma:gamma-borel} for the definition of $G_B$).  Since conjugation by $r$ acts as the non-trivial automorphism of $G_B \cong \ZZ$ (a simple direct computation), we have that this stabilizer is isomorphic to the non-trivial semi-direct product $\ZZ \rtimes (\ZZ/2\ZZ)$.
               The stabilizer of $\widehat{\RR}$ is $G$ by Lemma \ref{lemma:stab}.
               The stabilizer of the directed edge $\widehat{\RR} \xrightarrow{\infty} (\widehat{\RR} + \tau)$ is $G_B$.

               With these data, Theorem \ref{thm:bass-serre} describes the structure of $\Acal_K^0$ as
       $G *_{G_B} \langle G_B, r \rangle$.
       Recalling that $G$ has presentation $\langle \gamma_1, \gamma_2, \gamma_3 \rangle$ and with reference to Lemma \ref{lemma:gamma-borel}, the presentation is now a direct computation.
        \end{proof}

\begin{theorem}
        \label{thm:gen-union}
        Suppose $\Delta \le -15$.  Then the group $\Acal_K^0$ is a geometric Apollonian group for $K$, and $\widehat{\Acal_K^0}$ is an algebraic Apollonian group for $K$.
\end{theorem}

\begin{proof}
        This proof uses Theorem \ref{thm:sufficient-clusters}.  That $\Acal_K^0$ is a finitely generated Kleinian group is evident.  Since its generators take $\widehat{\RR}$ to itself or immediately tangent circles, it takes elements of $\Pcal_K$ to $\Pcal_K$.  It now suffices to verify some details about the base cluster and $\Acal_K^0$.  Items \eqref{item:basecontained} and \eqref{item:baseunique} are immediate.  
        
        It was noted above that $\Acal_K^0$ contains the topographical group $G$, which implies item \eqref{item:nearbysuper}.  
        
        For item \eqref{item:tangentstep}, it suffices to exhibit such elements: the last element of \eqref{eqn:sixmob}, together with its conjugations by $\gamma_1$ and $\gamma_3$, will suffice.
        
        To verify \eqref{item:noauto}, we must appeal to the tangency tree.  In particular, the automorphisms of the base cluster must stabilize $\widehat{\RR}$; this stabilizer is exactly $G$ by Lemma \ref{lemma:stab}.  But $G$ contains no automorphisms of superbases, so there are no automorphisms of the base cluster in $\Acal_K^0$.

        Finally, items \eqref{item:limit} and \eqref{item:connected} are immediate.
\end{proof}

\section{Cubes and cubicles in $K = \QQ(\sqrt{-2})$}

\newcommand{\Qtwo}{\QQ(\sqrt{-2})}

In the case of $K = \QQ(\sqrt{-2})$, there is a very pretty Apollonian group (an alternative to the general one in the last section), and we will elaborate on it somewhat here.

In the graph of tangencies, we find cubes, as Figure \ref{fig:basecube}.
A cube is defined to be eight circles in the following arrangement specified by the Pedoe products. Here $W_C$ is a $4 \times 8$ matrix which has as columns the eight $\pi(C_i)$:
\begin{equation}
        \label{eqn:qtwomat}
        W_C^t G_M W_C = \begin{pmatrix}
                1 & -1 & -3 & -1 & -5 & -3 & -1 & -3\\
               -1 & 1 & -1 & -3 & -3 & -5 & -3 & -1\\
               -3 & -1 & 1 & -1 & -1 & -3 & -5 & -3\\
               -1 & -3 & -1 & 1 & -3 & -1 & -3 & -5\\
               -5 & -3 & -1 & -3 & 1 & -1 & -3 & -1\\
               -3 & -5 & -3 & -1 & -1 & 1 & -1 & -3\\
               -1 & -3 & -5 & -3 & -3 & -1 & 1 & -1\\
               -3 & -1 & -3 & -5 & -1 & -3 & -1 & 1\\
        \end{pmatrix}
\end{equation}
This matrix represents the following arrangement of tangencies:
\[
        \xymatrix@R=1em@C=1em{
                & \bfv_1 \ar@{-}[rr] \ar@{-}[dd] \ar@{-}[ld] & & \bfv_2 \ar@{-}[dd]\ar@{-}[ld] \\
   \bfv_4 \ar@{-}[rr] \ar@{-}[dd] & & \bfv_3 \ar@{-}[dd]\\
                & \bfv_7 \ar@{-}[rr] \ar@{-}[ld] &&\bfv_8  \ar@{-}[ld]\\
              \bfv_6 \ar@{-}[rr] && \bfv_5 
        }
\]
It is evident that a cube contains eight $\QQ(\sqrt{-2})$-clusters as defined in the previous section, corresponding to the eight vertices of the cube (considered up to reordering the three circles tangent to a central one).  Any such $\QQ(\sqrt{-2})$-cluster defines a unique cube (this follows from verifying the fact for the base cluster).

We define a cubicle to be a subset of four circles of a given cube so that no two are tangent.  There are two cubicles in a cube.  A cubicle satisfies this arrangement\footnote{Incidentally, this gives a relation on the curvatures of a cubicle: $8(a^2+b^2 + c^2 + d^2) = 3(a+b+c+d)^2$ which we do not need here.}:
\begin{equation}
        \label{eqn:qtwomatcubicle}
        W_D^t G_M W_D = \begin{pmatrix}
                1 & -3 & -3 & -3 \\
                -3 & 1 & -3 & -3 \\
                -3 & -3 & 1 & -3 \\
                -3 & -3 & -3 & 1 \\
        \end{pmatrix}.
\end{equation}
A cubicle is contained in a unique cube.  By considering a single cubicle in $\QQ(\sqrt{-2})$ (we use Figure \ref{fig:basecube}), one can obtain equations that determine the cube from the cubicle\footnote{It is also nice to observe that all body diagonal sums agree: $\bfv_1 + \bfv_5 = 
        \bfv_2 + \bfv_6 = 
        \bfv_3 + \bfv_7 = 
        \bfv_4 + \bfv_8,$ 
and that on each individual face, diagonal sums again agree: $\bfv_1 + \bfv_3 = \bfv_2 + \bfv_4$ etc.}:
\begin{align*}
        2\bfv_2 &=\bfv_1 + \bfv_3 - \bfv_6 + \bfv_8, \\
        2\bfv_4 &=\bfv_1 + \bfv_3 + \bfv_6 - \bfv_8, \\
        2\bfv_5 &=-\bfv_1 + \bfv_3 + \bfv_6 + \bfv_8, \\
        2\bfv_7 &=\bfv_1 - \bfv_3 + \bfv_6 + \bfv_8.
\end{align*}

Any individual face of a cube consists of four $\Qtwo$-Bianchi circles satisfying the following relations (in particular, they are tangent in a loop):
\begin{equation}
        \label{eqn:qtwomatface}
        W_D^t G_M W_D = \begin{pmatrix}
                1 & -1 & -3 & -1 \\
                -1 & 1 & -1 & -3 \\
                -3 & -1 & 1 & -1 \\
                -1 & -3 & -1 & 1 \\
        \end{pmatrix}.
\end{equation}
Then there are exactly two ways to complete a face to a cube of $\Qtwo$-Bianchi circles (again, this need only be verified on the base cube).  This implies that given a cube, there are six \emph{swaps} one can perform which fix one side and swap out the remaining four circles.  Geometrically, this is accomplished by reflecting in the circle $C$ orthogonal to all circles in the fixed side, as shown in Figure \ref{fig:2swap}.

\begin{figure}
        \includegraphics[height=3.5in]{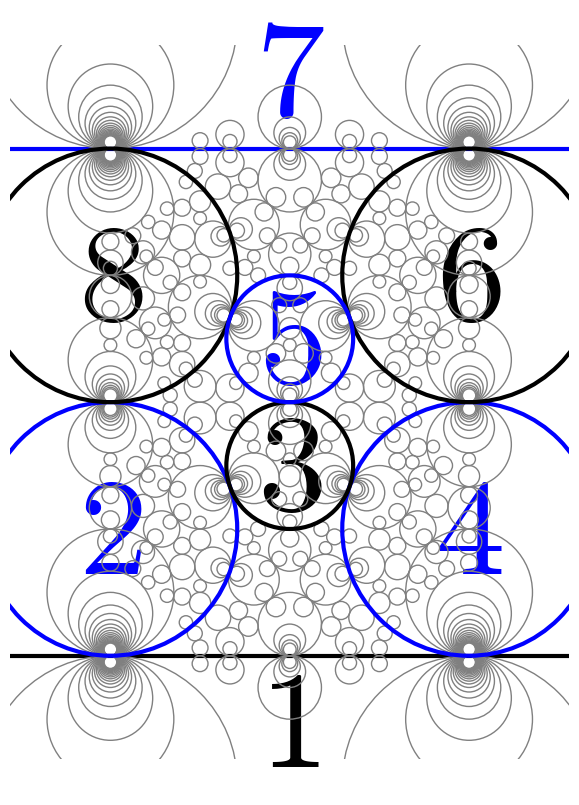} 
        \caption[Base cube for $\Scal_{\QQ(\sqrt{-2})}$]{Base cube for $\Scal_{\QQ(\sqrt{-2})}$.  The coordinates of the eight circles in $\MM$ are given, in the labelled order, by the columns of the following matrix:
                {\tiny 
                $$\begin{pmatrix}
                        0 &  0 & 2\sqrt{2}  & 2\sqrt{2} & 4\sqrt{2} & 4\sqrt{2} & 2\sqrt{2} & 2\sqrt{2} \\
                        0 &  2\sqrt{2} & 4\sqrt{2}  & 2\sqrt{2} & 4\sqrt{2} & 2\sqrt{2} & 0 & 2\sqrt{2}  \\
                        0 &  0 & 2\sqrt{2}  & 2\sqrt{2} & 2\sqrt{2} & 2\sqrt{2} & 0 & 0 \\
                        -1 & 1 & 3  & 1 & 5 & 3 & 1 & 3\\
        \end{pmatrix}.$$
} The two cubicles forming the cube are shown in blue and black, respectively.}
\label{fig:basecube}
\end{figure}

\begin{figure}
        \includegraphics[height=2.5in]{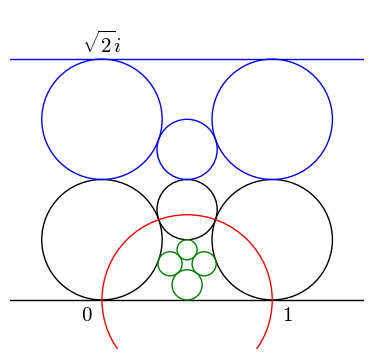}
        \caption{The blue and black circles (including two straight lines) form the base cube of the $\QQ(\sqrt{-2})$-Apollonian packing containing $\widehat{\RR}$.  If swapping through the side represented by black circles, the blue circles are replaced with green ones, by inversion in the red circle.  The new cube is formed of the black and green circles.}
        \label{fig:2swap}
\end{figure}

\begin{figure}
        \includegraphics[height=2.5in]{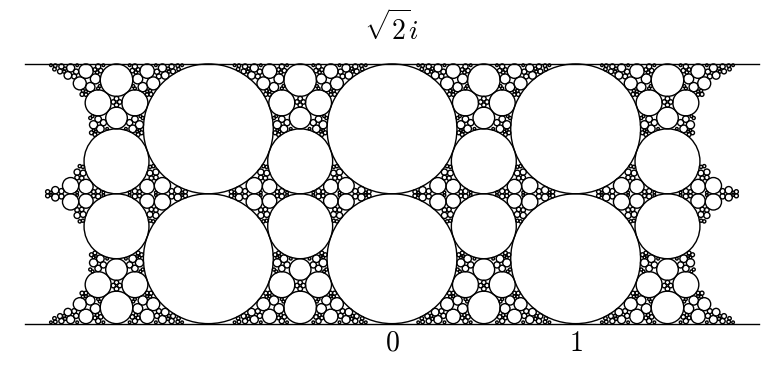} 
        \caption[A $\QQ(\sqrt{-2})$-Apollonian packing.]{The $\QQ(\sqrt{-2})$-Apollonian packing containing $\widehat{\RR}$.}
\label{fig:2strip}
\end{figure}

These six swaps can be described in terms of the cube as follows.  
Fixing $\bfv_1,\ldots, \bfv_4$, the two ways to complete to a cube containing this side are to add $\bfv_5, \ldots, \bfv_8$ or $\bfv_5',\ldots,\bfv_8'$ where
\begin{align*}
        \bfv_5 + \bfv_5' &= -2\bfv_1 +4\bfv_2 + 4\bfv_4 \\
        \bfv_6 + \bfv_6' &= 4\bfv_1 - 2\bfv_2 + 4\bfv_3 \\
        \bfv_7 + \bfv_7' &= 4\bfv_2 - 2\bfv_3 + 4\bfv_4 \\
        \bfv_8 + \bfv_8' &= 4\bfv_1 +4\bfv_3 -2\bfv_4
\end{align*}

If we express a cube as a real $4 \times 4$ matrix whose columns are the cubicle $\bfv_1, \bfv_3, \bfv_6, \bfv_8$, then these swaps correspond to right multiplication by
\begin{equation}
        \label{eqn:Aqtwogens}
        \begin{pmatrix}
                1 & 0 & 3 & 3 \\
                0 & 1 & 3 & 3 \\
                0 & 0 & 0 & -1 \\
                0 & 0 & -1 & 0 \\
        \end{pmatrix},
       \begin{pmatrix}
                1 & 3 & 0 & 3 \\
                0 & 0 &  0 & -1 \\
                0 & 3 & 1 & 3 \\
                0 & -1 & 0 & 0 \\
        \end{pmatrix},
       \begin{pmatrix}
                0 & 0 & 0 & -1 \\
                3 & 1 & 0 & 3 \\
                3 & 0 & 1 & 3 \\
                -1 & 0 & 0 & 0 \\
        \end{pmatrix},
\end{equation}
\[
       \begin{pmatrix}
                1 & 3 & 3 & 0 \\
                0 & 0 & -1 & 0 \\
                0 & -1 & 0 & 0 \\
                0 & 3 & 3 & 1 \\
        \end{pmatrix},
       \begin{pmatrix}
                0 & 0 & -1 & 0 \\
                3 & 1 & 3 & 0 \\
                -1 & 0 & 0 & 0 \\
                3 & 0 & 3 & 1 \\
        \end{pmatrix},
       \begin{pmatrix}
                0 & -1 & 0 & 0 \\
                -1 & 0 & 0 & 0 \\
                3 & 3 & 1 & 0 \\
                3 & 3 & 0 & 1 \\
        \end{pmatrix}.
\]

\begin{definition}
        \label{def:AQtwo}
        Call the group generated by these matrices $\widehat{\Acal_{\Qtwo}}$.
\end{definition}

The M\"obius transformations which realise the six swaps on the base cube of Figure \ref{fig:basecube} are given by (in the same order as the generators \eqref{eqn:Aqtwogens}),
\begin{gather}
        \label{eqn:2mob}
        z \mapsto \frac{\overline{z}}{2\overline{z}-1}, \quad
        z \mapsto -\overline{z}+2, \quad
        z \mapsto \frac{(3+2\sqrt{-2})\overline{z}-4}{4\overline{z} - 3 + 2\sqrt{-2}}, \\
        z \mapsto -\overline{z}, \quad
        z \mapsto \frac{(1+2\sqrt{-2})\overline{z}-2}{4\overline{z} - 1 + 2\sqrt{-2}}, \quad
        z \mapsto \frac{(1+2\sqrt{-2})\overline{z} - 4}{2\overline{z} - 1 + 2\sqrt{-2}}. \notag
\end{gather}

\begin{definition}
Call the group generated by these transformations $\Acal_{\QQ(\sqrt{-2})}$.
\end{definition}

\begin{definition}
        \label{def:cube-graph}
        The \emph{cube graph} is the graph whose vertices are all cubes (considered without regard to the ordering of the constituent circles); and whose edges indicate cubes sharing a side.
\end{definition}

\begin{theorem}
        \label{thm:qtwotree}
        The cube graph is a forest of trees of valence $6$.  Consequently,
        $A_{\Qtwo}$ is freely generated by the generators \eqref{eqn:Aqtwogens}.  In particular, it is a free product of six copies of $\ZZ/2\ZZ$.
\end{theorem}

\begin{proof}
        Let $H$ be the cube graph.  Any one component of the graph of ordered cubes under the six available swaps is the Cayley graph of $\Acal_{\QQ(\sqrt{-2})}$ under right multiplication and generators \eqref{eqn:Aqtwogens}.  We wish to demonstrate that any component of $H$ is a tree.  If so, then the Cayley graph cannot have a loop, as it would reduce to a loop in the cube graph under forgetting orientation.  The theorem follows.
        
        Let $G$ be the tangency graph of a single packing, i.e. the graph whose vertices are circles, and whose edges represent tangencies.  This graph is the same for any packing, so assume the packing is bounded away from $\infty$.  This graph can be embedded on the sphere $\widehat{\CC}$ by placing each vertex at the center of the corresponding circle, so that edges pass through tangency points (the exception being the bounding circle, where we can take its center to be $\infty$).  Therefore it is planar.

        Suppose that $H$ has a loop.  This loop is a finite loop of $n\ge2$ cubes lying in one packing, so its circles create an induced subgraph $L$ of $G$ (for that packing).  We will demonstrate that $L$ cannot exist inside the planar graph $G$.

        To do so, we will consider building the graph $L$ cube-by-cube.  Begin with one adjacent pair of cubes in $L$, sharing four vertices in a cycle $C$.  The cycle breaks the plane up into two regions.  The two cubes lie in the two different regions (except that they both include the boundary).  (By the M\"obius action, verification on the base cube suffices.)  Continuing to add cubes to form $L$, at each stage we are adding edges and vertices of $G$ inside an existing face of the construction.  In particular, we are either adding a cube inside the cycle $C$ or outside it.  Finally, to complete the loop $L$, we must join one cube inside $C$ to one outside $C$, using another cycle $C' \neq C$.  This is impossible inside the planar graph $G$.
\end{proof}

An alternate method of proof is to show that $L$ contains $4n$ vertices and $8n$ edges, but no cycles of length $3$, which is impossible for a planar graph.  Interestingly, a similar count works for $\QQ(i)$ and $\QQ(\sqrt{-7})$, but fails for $\QQ(\sqrt{-11})$.  The proof above generalises to all cases with little modification.

\begin{theorem}
        \label{thm:2union}
        The group $\Acal_{\QQ(\sqrt{-2})}$ is a geometric Apollonian group for $\QQ(\sqrt{-2})$ and $\widehat{\Acal_{\QQ(\sqrt{-2})}}$ is an algebraic Apollonian group for $\QQ(\sqrt{-2})$.
\end{theorem}

\begin{proof}
        Our cluster type is that given above (cubes), and the base cube is as in Figure \ref{fig:basecube}.  We will verify the hypotheses of Theorem \ref{thm:sufficient-clusters}; many of the verifications are simple and similar to those in the proof of Theorem \ref{thm:gen-union}.  Item \eqref{item:nearbysuper} is verified by generators 1, 2 and 4 of \eqref{eqn:2mob}.  Item \eqref{item:tangentstep} is verified by identity transformation, since the base cube contains a three-prong centred on any of its circles.
        
        That no automorphism of the base cluster is contained in $\Acal_{\QQ(\sqrt{-2})}$ (item \eqref{item:noauto}) is a consequence of Theorem \ref{thm:qtwotree}:  such an automorphism would indicate a nontrivial path in the Cayley graph whose endpoints were two different orderings of the same cube.  Forgetting orientation, this would give a loop in the cube graph.
\end{proof}

\section{Tents and belts in $K = \QQ(\sqrt{-7})$}
\newcommand{\Qseven}{\QQ(\sqrt{-7})}

In this case, and for $\QQ(\sqrt{-11})$, we will be somewhat more brief in our description.
We define a $\Qseven$-Descartes configuration, called a \emph{tent}, to be five circles in the arrangement specified by the relation
\[
        W_D^t G_M W_D = \begin{pmatrix}
                1 & -1 & -5/2 & -1 & - 5/2 \\
               -1 & 1 & -1 & -5/2 & -1  \\
             -5/2 & -1 & 1 & -1 & - 5/2 \\
               -1 & -5/2 & -1 & 1  & -1 \\
             -5/2 & -1 & -5/2 & -1 & 1  \\
        \end{pmatrix},
\]
on the matrix $W_D$ whose columns are the five circles.  The graph of tangencies looks like this:
\[
        \xymatrix@R=1em@C=1em{
                & \bfv_2 \ar@{-}[d] \ar@{-}[dl] \ar@{-}[dr] & \\
 \bfv_1 \ar@{-}[dr] & \bfv_5 \ar@{-}[d]  & \bfv_3 \ar@{-}[dl] & \\
                & \bfv_4 & \\
        }
\]

There is one relation among these circles:
\[
        \bfv_1 + \bfv_3 + \bfv_5 = 2(\bfv_2 + \bfv_4).
\]
Three of the five circles ($\bfv_1, \bfv_3, \bfv_5$) are distinguished as being tangent to fewer other circles in the tent (2 instead of 3).  Call such a circle a \emph{peak}.  Removing a peak leaves four circles in a cycle, called a \emph{belt}.  Given a belt, there are exactly two tents containing it.  Given a belt $\bfv_1, \bfv_2, \bfv_3, \bfv_4$, the two peaks $\bfv_5$ and $\bfv_5'$ that complete it to a tent satisfy
\[
        \bfv_5 + \bfv_5' = \bfv_1 + \bfv_2 + \bfv_3 + \bfv_4.
\]
The other peak $\bfv_5'$ is tangent to $\bfv_1$ and $\bfv_3$, thus:
\[
        \xymatrix@R=1em@C=1em{
                & \bfv_2  \ar@{-}[dl] \ar@{-}[dr] & \\
 \bfv_1 \ar@{-}[dr] \ar@{-}[r] & \bfv_5'  \ar@{-}[r] & \bfv_3 \ar@{-}[dl] & \\
                & \bfv_4 & \\
        }
\]
Therefore $\bfv_2$ and $\bfv_4$ become new peaks.  It is appropriate to renumber the resulting tent so we have
\begin{equation*}
        \bfv_1' = \bfv_2, \quad \bfv_2' = \bfv_1, \quad 
        \bfv_3' = \bfv_4, \quad \bfv_4' = \bfv_3, \quad \bfv_5' = \bfv_1 + \bfv_2 + \bfv_3 + \bfv_4 - \bfv_5.
\end{equation*}
An example is shown in Figure \ref{fig:7-swap}.

\begin{figure}
        \includegraphics[height=3in]{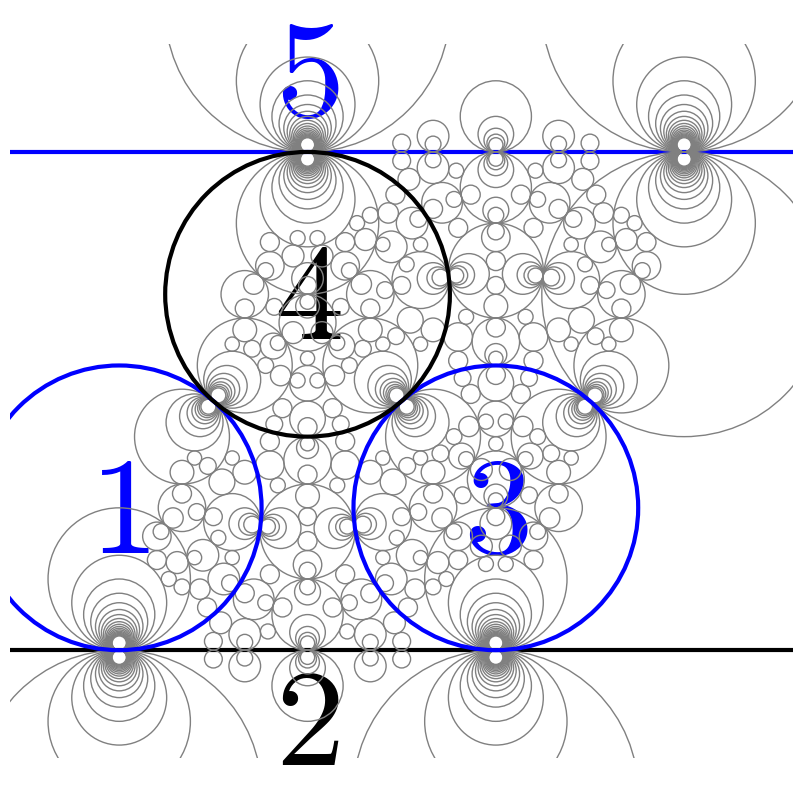} 
        \caption[Base tent for $\Scal_{\QQ(\sqrt{-7})}$]{Base tent for $\Scal_{\QQ(\sqrt{-7})}$.  The coordinates of the five circles in $\MM$ are given, in the labelled order, by the columns of the following matrix: \tiny
                \[
                        \begin{pmatrix}
                                0 & 0 & \sqrt{7} & \sqrt{7} & \sqrt{7} \\
                                \sqrt{7} & 0 & \sqrt{7} & \sqrt{7} & 0 \\
                                0 & 0 & \sqrt{7} & \sqrt{7}/2 & 0 \\
                                1&-1&1&5/2&1\\
                        \end{pmatrix}
                \]
\normalsize Circles tangent to three others in the tent are shown in black, while those tangent to two others are shown in blue.}
\label{fig:7basetent}
\end{figure}

\begin{figure}
        \includegraphics[height=2.5in]{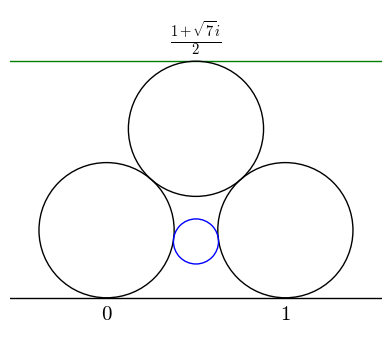}
        \caption{The green and black circles (including two straight lines) form the base tent of the $\QQ(\sqrt{-7})$-Apollonian packing containing $\widehat{\RR}$.  If swapping out the peak given by the green circle, the green circle is replaced with the blue one.  The new tent is formed of the black and blue circles.  Note that the black circles are not individually fixed; they are permuted.}
        \label{fig:7-swap}
\end{figure}

\begin{figure}
        \includegraphics[height=2.5in]{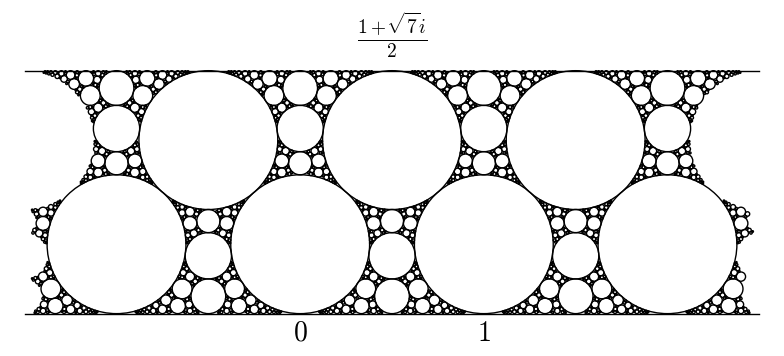} 
        \caption[A $\QQ(\sqrt{-7})$-Apollonian packing.]{The $\QQ(\sqrt{-7})$-Apollonian packing containing $\widehat{\RR}$.}
\label{fig:7strip}
\end{figure}

We will write a tent as a $4 \times 4$ matrix whose columns are $\bfv_1, \bfv_2, \bfv_3, \bfv_4$; call this a \emph{tentbase}.  Four circles form a tentbase if and only if the matrix of their columns, $W_D$, satisfies
\[
        W_D^t G_M W_D = \begin{pmatrix}
                1 & -1 & -5/2 & -1  \\
               -1 & 1 & -1 & -5/2   \\
             -5/2 & -1 & 1 & -1 \\
               -1 & -5/2 & -1 & 1   \\
        \end{pmatrix} =: R.
\]
There are three moves that will replace a tent with another that shares a belt:  each of the three peaks can be swapped out. These correspond to multiplying the tentbase on the right by these three matrices of order two:
\begin{equation}
        \label{eqn:qsevengens}
        \begin{pmatrix}
                0 & 1 & 0 & 0 \\
                1 & 0 & 0 & 0 \\
                0 & 0 & 0 & 1 \\
                0 & 0 & 1 & 0 \\
        \end{pmatrix},
        \begin{pmatrix}
                -2 & 0 & 0 & -1 \\
                  3 & 0 & 1 & 2 \\
                0 & 1 & 0 & -1 \\
                  3 & 0 & 0 & 2 \\
        \end{pmatrix},
        \begin{pmatrix}
                0 & -1 &  0 & 1 \\
                0 & 2 & 3 & 0  \\
                0 & -1 & -2 &  0  \\
                1 & 2 & 3 & 0  \\
        \end{pmatrix}.
\end{equation}

\begin{definition}
        Denote the subgroup of $O_R(\RR)$ generated by the matrices \eqref{eqn:qsevengens} by $\widehat{\Acal_{\Qseven}}$.  
\end{definition}
The three swaps \eqref{eqn:qsevengens} are not realised as inversions in circles as in the case of $\QQ(i)$ and $\QQ(\sqrt{-2})$.  On the base tent of Figure \ref{fig:7basetent}, they correspond, via the algebraic-geometric correspondence, to the following M\"obius transformations, respectively:
\begin{equation}
        \label{eqn:7mob}
        z \mapsto \frac{\frac{1 + \sqrt{-7}}{2}z +  \frac{1-\sqrt{-7}}{2}}{\frac{3+\sqrt{-7}}{2}z + \frac{-1-\sqrt{-7}}{2}}, \quad
        z \mapsto -z- \frac{-3-\sqrt{-7}}{2}, \quad
        z \mapsto \frac{z}{
        \frac{1 - \sqrt{-7}}{2}z-1}.
\end{equation}

\begin{definition}
        The \emph{$7$-tent graph} is the graph whose vertices are all tents (considered without regard to ordering); and whose edges indicate tents sharing a belt.
\end{definition}

\begin{theorem}
        \label{thm:7free}
        The $7$-tent graph is a forest of trees of valence 3.  Consequently,
        $A_{\Qseven}$ is freely generated by the generators \eqref{eqn:qsevengens}.  In particular, it is a free product of three copies of $\ZZ/2\ZZ$.
\end{theorem}

\begin{proof}
        The proof is exactly as in Theorem \ref{thm:qtwotree}:  any two tents lie in the two regions created by their shared belt.
\end{proof}

\begin{theorem}
        \label{thm:7union}
        The group $\Acal_{\QQ(\sqrt{-7})}$ is a geometric Apollonian group for $\QQ(\sqrt{-7})$ and the group $\widehat{\Acal_{\QQ(\sqrt{-7})}}$ is an algebraic Apollonian group for $\QQ(\sqrt{-7})$.
\end{theorem}

\begin{proof}
        Tents are the clusters, and the base tent is as in Figure \ref{fig:7basetent}.  Most of the verifications of the hypotheses of Theorem \ref{thm:sufficient-clusters} are simple or similar to those in the proof of Theorems \ref{thm:gen-union} and \ref{thm:2union}; we consider only \eqref{item:nearbysuper}, \eqref{item:tangentstep} and \eqref{item:noauto}.
        
        Composing each pair of generators \eqref{eqn:7mob} in each possible order, we obtain 6 transformations.  It is a computation to verify that, applying these 6 transformations to the base cluster, we obtain among the results clusters containing the three-prongs on $(0,1/2,1)$, $(1,2,\infty)$ and $(-1,0,\infty)$.  This verifies \eqref{item:nearbysuper}.
        
        Item \eqref{item:tangentstep} is verified by the three swaps \eqref{eqn:7mob}.  Finally, \eqref{item:noauto} is verified, as in the proof of Theorem \ref{thm:2union}, by Theorem \ref{thm:7free}.
\end{proof}

\section{Tents and belts in $K = \QQ(\sqrt{-11})$}
\label{sec:11}

\newcommand{\Qeleven}{\QQ(\sqrt{-11})}

This case is similar to the case $\Qseven$.  A \emph{tent} $D$ consists of $10$ circles, and contains four \emph{belts} (loops of tangent circles) of $6$ circles each, in the following arrangement:
 \[
        \xymatrix@C=1em@R=1em{
                \bfv_1 \ar@{-}[dr]  & & \bfv_2  \ar@{-}[ll]\ar@{-}[rr] & & \bfv_3 \ar@{-}[dl] \ar@{-}[ddl] \\
                                  & \bfv_{10} \ar@{-}[dr] & & \bfv_8 \ar@{-}[dl]  & \\
                                  & \bfv_6 \ar@{-}[ddr] \ar@{-}[luu] & \bfv_7 \ar@{-}[d] &  \bfv_4 \ar@{-}[ddl] & \\
                                  &  & \bfv_9 \ar@{-}[d] & & \\
                 & & \bfv_5 & & \\
        }
\]

Let $W_D$ denote the matrix whose columns are the vectors in $\MM$ corresponding to these circles, say $\bfv_i$, $i=1,\ldots,10$ of $\MM$.  We have
 \[\scriptsize
         W_D^t G_M W_D = 
         \begin{pmatrix}
                 1 & -1 & -9/2 & -13/2 & - 9/2 & - 1 & -9/2 & - 13/2 & -13/2 & -1\\
                 -1 & 1 & -1 & -9/2 & -13/2 & - 9/2  & -13/2&-9/2&-10&-9/2\\
                 - 9/2 & - 1 & 1 & -1 & -9/2 & -13/2 & -9/2 & -1 & -13/2 & -13/2  \\
                 -13/2 & - 9/2 & - 1 & 1 & -1 & -9/2 & -13/2 & -9/2 & -9/2 & -10 \\
                 - 9/2 & -13/2 & - 9/2 & - 1 & 1 & -1 & -9/2 & -13/2 & -1 & -13/2 \\
                 -1 & - 9/2 & -13/2 & - 9/2 & - 1 & 1 & -13/2 & -10 & -9/2 & -9/2 \\
                 -9/2 & -13/2 & -9/2 & -13/2 & -9/2 & -13/2 & 1 & -1 & -1 & -1 \\
                 -13/2 & -9/2 & -1 & -9/2 & -13/2 & -10 & -1 & 1 & -9/2 & -9/2 \\
                 -13/2 & -10 & -13/2 & -9/2 & -1 & -9/2 & -1 & -9/2 & 1 & -9/2 \\
                 -1 & -9/2 & -13/2 & -10 & -13/2 & -9/2 & -1 & -9/2 & -9/2 & 1 \\
         \end{pmatrix}.
         \normalsize
 \] 
 The ten circles of a tent span a vector space of dimension $4$.  A presentation of the relations is:
\begin{align*}
         5\bfv_2 &= - 4 \bfv_1 - 4 \bfv_3 + \bfv_5 + \bfv_7\\
         5\bfv_4 &=  \bfv_1 - 4 \bfv_3 -4 \bfv_5 + \bfv_7\\
         5\bfv_6 &= - 4 \bfv_1 + \bfv_3 -4 \bfv_5 + \bfv_7\\
         5\bfv_8 &=  \bfv_1 - 4 \bfv_3 + \bfv_5 -4 \bfv_7\\
         5\bfv_9 &=  \bfv_1 + \bfv_3 -4 \bfv_5 -4 \bfv_7\\
      5\bfv_{10} &= - 4 \bfv_1 + \bfv_3 + \bfv_5 -4 \bfv_7
 \end{align*}
 A tent contains four belts.  Within a belt, the sum of opposite circles is invariant, e.g.
 \[
         \bfv_1 + \bfv_4 = \bfv_2 + \bfv_5 = \bfv_3 + \bfv_6.
 \]
 Since there are four belts, one obtains four vectors; these are independent, and we therefore represent a tent as a matrix with these four columns, say
 \[
         \bfv_1 + \bfv_4, \bfv_1 + \bfv_8, \bfv_1 + \bfv_9, \bfv_3 + \bfv_9.
 \]
 There are exactly two tents containing a single belt.  Therefore there are four swaps one can perform.  Each swap preserves one belt and changes three belts.  An example is shown in Figure \ref{fig:11-swap}.  The resulting matrices, in terms of the representation above, are
\begin{equation}
        \label{eqn:qelevengens}
        \begin{pmatrix}
                1 & 3 & 3 & 3 \\
                0 & -1 & 0 & 0 \\
                0 & 0 & -1 & 0 \\
                0 & 0 & 0 & -1 \\
        \end{pmatrix},
        \begin{pmatrix}
                -1 & 0 & 0 & 0 \\
                3 & 1 & 3 & 3 \\
                0 & 0 & -1 & 0 \\
                0 & 0 & 0 & -1 \\
        \end{pmatrix},
        \begin{pmatrix}
                -1 & 0 & 0 & 0 \\
                0 & -1 & 0 & 0 \\
                3 & 3 & 1 & 3 \\
                0 & 0 & 0 & -1 \\
        \end{pmatrix},
        \begin{pmatrix}
                -1 & 0 & 0 & 0 \\
                0 & -1 & 0 & 0 \\
                0 & 0 & -1 & 0 \\
                3 & 3 & 3 & 1 \\
        \end{pmatrix}.
\end{equation}
Note that the circles of the belt are permuted; for example, the first of these has the effect
\[
        \mathbf{v}_1' = \mathbf{v}_4, \mathbf{v}_2' = \mathbf{v}_5, \mathbf{v}_3' = \mathbf{v}_6, \mathbf{v}_4' = \mathbf{v}_1, \mathbf{v}_5' = \mathbf{v}_2, \mathbf{v}_6' = \mathbf{v}_3.
\]

\begin{figure}
        \includegraphics[height=3in]{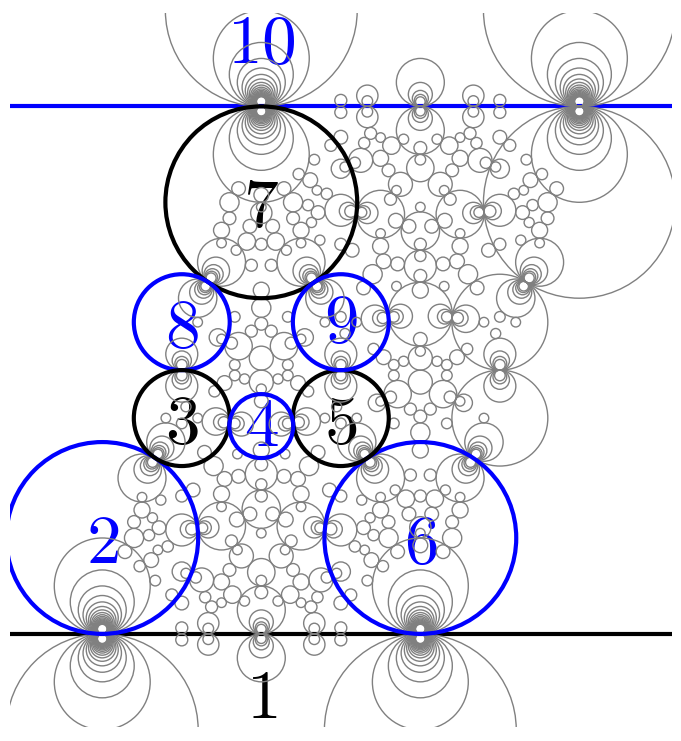} 
        \caption[Base tent for $\Scal_{\QQ(\sqrt{-11})}$]{Base tent for $\Scal_{\QQ(\sqrt{-11})}$.  The coordinates of the ten circles in $\MM$ are given, in the labelled order, by the columns of the following matrix: \tiny
                \[
                        \begin{pmatrix}
                                0&0&\sqrt{11}&2\sqrt{11}&2\sqrt{11}&\sqrt{11}&2\sqrt{11}&2\sqrt{11}&3\sqrt{11}&\sqrt{11} \\
                                0&\sqrt{11}&2\sqrt{11}&3\sqrt{11}&2\sqrt{11}&\sqrt{11}&\sqrt{11}&2\sqrt{11}&2\sqrt{11}&0 \\
                                0&0&\sqrt{11}/2&3\sqrt{11}/2&3\sqrt{11}/2&\sqrt{11}&\sqrt{11}/2&\sqrt{11}/2&3\sqrt{11}/2&0 \\
                                -1&1&9/2&13/2&9/2&1&9/2&13/2&13/2&1 \\
                        \end{pmatrix}
                \]
\normalsize Circles tangent to three others in the tent are shown in black, while those tangent to two others are shown in blue.}
\label{fig:11basetent}
\end{figure}

\begin{figure}
        \includegraphics[height=2.5in]{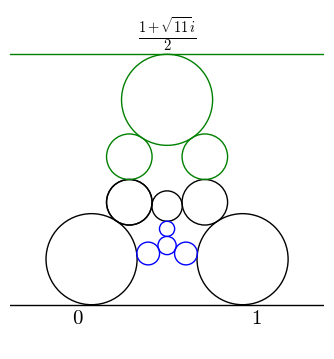}
        \caption{The green and black circles (including two straight lines) form the base tent of the $\QQ(\sqrt{-11})$-Apollonian packing containing $\widehat{\RR}$.  If swapping out the peak given by the green circles, the green circles are replaced with the blue ones.  The new tent is formed of the black and blue circles.  Note that the black circles are not individually fixed; they are cyclically permuted.}
        \label{fig:11-swap}
\end{figure}

\begin{figure}
        \includegraphics[height=2.5in]{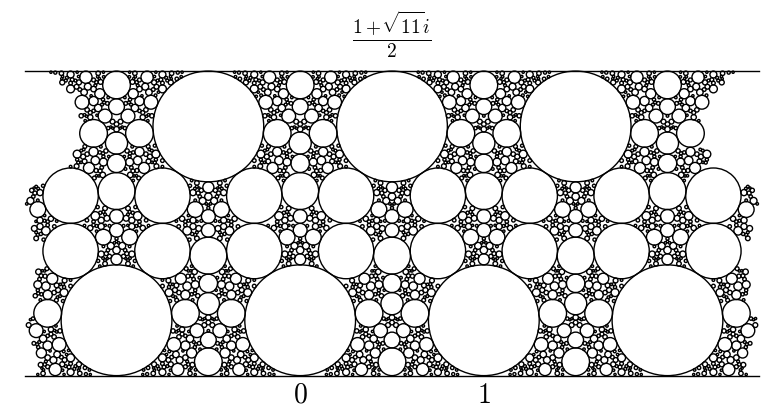} 
        \caption[A $\QQ(\sqrt{-11})$-Apollonian packing.]{The $\QQ(\sqrt{-11})$-Apollonian packing containing $\widehat{\RR}$.}
\label{fig:11strip}
\end{figure}

\begin{definition}
        Denote the subgroup of $O_R(\RR)$ generated by these four matrices by $\widehat{\Acal_{\Qeleven}}$.  
\end{definition}

The four M\"obius maps performing these swaps on the base tent of Figure \ref{fig:11basetent} (in the same order as \eqref{eqn:qelevengens}) are:
\begin{equation}
        \label{eqn:11mob}
        z \mapsto \frac{ \frac{3+\sqrt{-11}}{2} \overline{z} - 2 }{3 \overline{z} + \frac{-3+\sqrt{-11}}{2} }, \quad
        z \mapsto \frac{ \frac{1+\sqrt{-11}}{2} \overline{z} - 2 }{2 \overline{z} + \frac{-1+\sqrt{-11}}{2} }, \quad
        z \mapsto \frac{ \frac{3+\sqrt{-11}}{2} \overline{z} - 3 }{2 \overline{z} + \frac{-3+\sqrt{-11}}{2} }, \quad
        z \mapsto \frac{ (2+\sqrt{-11}) \overline{z} - 4 }{4 \overline{z} -2+\sqrt{-11} }.
\end{equation}

\begin{definition}
        Denote the subgroup of $\Mob$ generated by these four matrices by ${\Acal_{\Qeleven}}$.  
\end{definition}

\begin{definition}
        The \emph{$11$-tent graph} is the graph whose vertices are all tents (considered without regard to ordering); and whose edges indicate tents sharing a belt.
\end{definition}

\begin{theorem}
        \label{thm:11free}
        The $11$-tent graph is a forest of trees of valence 4.  Consequently,
        $A_{\Qeleven}$ is freely generated by the generators \eqref{eqn:qelevengens}.  In particular, it is a free product of four copies of $\ZZ/2\ZZ$.
\end{theorem}

\begin{proof}
        The proof is exactly as in Theorem \ref{thm:qtwotree} and Theorem \ref{thm:7free}.
\end{proof}

\begin{theorem}
        \label{thm:11union}
        The group ${\Acal_{\QQ(\sqrt{-11})}}$ is a geometric Apollonian group for $\QQ(\sqrt{-11})$, and
        the group $\widehat{\Acal_{\QQ(\sqrt{-11})}}$ is an algebraic Apollonian group for $\QQ(\sqrt{-11})$.
\end{theorem}

\begin{proof}
        The proof is very similar to Theorem \ref{thm:7union}.  As before, compositions of generators verify \eqref{item:nearbysuper}, and the single generators verify \eqref{item:tangentstep}.
\end{proof}

\section{Curvatures in $K$-Apollonian packings}
\label{sec:curvatures}

This section is devoted to some computational data supporting Conjecture \ref{conj:new}.  We first record a basic result on curvatures.

\begin{table}
\caption[Obstruction to curvatures at the prime $2$]{For each discriminant $\Delta$, the table shows the smallest power of $2$, $M$, that explains the obstructions at the prime $2$, and the observed residue sets $S_M$ modulo $M$.}
\begin{tabular}{|l|l|l|}
        \hline
        $\Delta$ &$M$ & $S_M$ \\
        \hline
        $\Delta \equiv 0,4,16 \pmod{32}$ & $2$ & $S_2=\{ 0,1 \}, \{ 1\}$ \\
        $\Delta \equiv 8,24 \pmod{32}$ &  $4$ &$S_4=\{ 0,2,3 \}, \{ 0,1,2 \}$ \\
        $\Delta \equiv 12 \pmod{32}$ &  $4$ &$S_4=\{ 0,1 \}, \{ 1,2\}, \{ 2,3 \}, \{0,3\}$ \\
        $\Delta \equiv 20 \pmod{32}$ &  $4$ &$S_4=\{ 1 \}, \{ 3 \}, \{0,1,2,3\}$ \\
        $\Delta \equiv 28 \pmod{32}$ &  $8$ &$S_8=\{ 0,1,4 \}, \{ 2,3,6,7\}, \{ 0,4,5 \}$ \\
        otherwise & none & none \\
        \hline
\end{tabular}
        \label{table2}
\end{table}

\begin{table}
\caption[Obstruction to curvatures at the prime $3$]{For each discriminant $\Delta$, the obstruction at the prime $3$ is explained by the first power, $3$.  The table shows the observed residue sets $S_3$ modulo $3$.}
\begin{tabular}{|l|l|}
        \hline
        $\Delta$ & $S_3$ \\
        \hline
        $\Delta \equiv 5,8 \pmod{12}$ & $S_3=\{ 0,1 \}, \{ 0,2 \}$ \\
        otherwise & none \\
        \hline
\end{tabular}
        \label{table3}
\end{table}

\begin{theorem}
        Let $\Pcal$ be a $K$-Apollonian circle packing.  Then the reduced curvatures of $\Pcal$ have no common factor.
\end{theorem}

\begin{proof}
        $\Pcal$ is generated by immediate tangency.  Theorem 4.7 of \cite{VisOne} states that if $C_1$ and $C_2$ are immediately tangent, then 
        \[
                \operatorname{curv}(C_1) + \operatorname{curv}(C_2) = N(x)
        \]
        where $x$ is the denominator of the point of tangency.  If all the curvatures of the packing are divisible by some prime $p$, then $p \mid N(x)$.
        Fixing $C_1$, the value of $x$ for all immediately tangent $C_2$ are exactly those 
        $x$ in a rank two $\ZZ$-lattice $\beta \ZZ + \delta \ZZ \subset \OK$.  Here, $\beta$ and $\delta$ can be taken to be the lower row of some matrix $M \in \PSL_2(\OK)$ taking $\widehat{\RR}$ to $C_1$ (in particular, $(\beta, \delta) = \OK$).
       
        Therefore, $p$ divides $N(\beta)$, $N(\delta)$ and $N(\beta + \delta)$.  If $p \mid N(x)$ then $x \in \mathfrak{p}$ for some prime ideal $\mathfrak{p}$ above $p$.  Therefore at least two of $\beta$, $\delta$ and $\beta+\delta$ lie in the same prime ideal; but this is a contradiction to the fact that any pair of them generates $\OK$.
\end{proof}

Using Sage Mathematics Software \cite{Sage}, some computer experiments were performed.  The author computed the complete set of reduced curvatures in various $K$-Apollonian packings modulo various moduli $n$.  The results suggest Conjecture \ref{conj:new}.  In particular, each Apollonian packing was observed to omit certain modular equivalence classes.  The behaviour of individual primes is independent, so that we can discuss the \emph{obstruction at a prime $p$} as all equivalence classes modulo powers of $p$ that cannot occur.  The obstruction is \emph{explained} by $p^k$ if $p^k$ is the largest power of $p$ needed to describe the obstruction.  In experiments, the only obstructions that occurred were at $2$ (always explained by $2$, $4$ or $8$) and at $3$ (always explained by $3$).  This explains the number $24$ in Conjecture \ref{conj:new}.  Tables \ref{table2} and \ref{table3} give the observed allowable sets of residues.  It is conjectured that these tables are complete.

\bibliography{app-pack-bib}{}
\bibliographystyle{plain}

\end{document}